\documentclass[11pt]{amsart}
\usepackage{amssymb,amsmath,amsfonts,amsthm,enumerate}
\usepackage{graphicx}
\usepackage{color}
\usepackage{psfrag}
\usepackage{hyperref}
\usepackage{soul}

\newtheorem{theorem}{Theorem}[section]
\newtheorem{lemma}[theorem]{Lemma}
\newtheorem{proposition}[theorem]{Proposition}
\newtheorem{addendum}[theorem]{Addendum}
\newtheorem{corollary}[theorem]{Corollary}
\newtheorem{question}{Question}
\newtheorem{ltheorem}{Theorem} 
\newtheorem{lcorollary}[ltheorem]{Corollary}

\newtheorem*{maintheorem}{Main Theorem}

\newtheorem*{claim}{Claim}
\theoremstyle{definition}
\newtheorem{definition}[theorem]{Definition}

\newtheorem{remark}[theorem]{Remark}

\newcommand{\cG}{\mathcal{G}}

\def\real{\mathbb{R}}

\def\integer{\mathbb{Z}}
\def\natural{\mathbb{N}}
\def\supp{\operatorname{supp}}

\def\vol{\operatorname{vol}}

\def\id{\operatorname{id}}
\def\cF{\mathcal{F}}
\def\cW{\mathcal{W}}

\def\fh{\mathfrak{h}}

\def\cQ{\mathcal{Q}}
\def\cP{\mathcal{P}}

\def\cN{\mathcal{N}}

\def\cA{\mathcal{A}}

\def\quand{\quad\text{and}\quad}

\def\tmu{\tilde{\mu}}
\def\proj{\mathbb{R}P^{d-1}}

\def\GL{GL(d,\real)}

\def\Dif{\operatorname{Diff}}

\def\Gap{\operatorname{H}^{tr}}

\def\Dif{\operatorname{Diff}}

\newcommand{\norm}[1]{{\left\lVert  #1  \right\rVert}}

\title[Invariance principle and non-compact center foliations]{Invariance principle\\ and non-compact center foliations}
\author{Sylvain Crovisier, Mauricio Poletti}
\date{\today}
\thanks{S.C. and M.P. were partially supported by the ERC project 692925 NUHGD.
M.P. was also partially supported by Instituto Serrapilheira, grant ``Jangada Din\^amica: Impulsionando Sistemas Din\^amicos na Regi\~ao Nordeste" and Fondation Louis D.-Institut de France (project coordinated by Marcelo Viana).}

\newcommand{\information}{{
  \bigskip
  \footnotesize
	
  \textbf{Sylvain Crovisier}: \textsc{CNRS-Laboratoire de Math\'ematiques d'Orsay, UMR 8628, Universit\'e Paris-Saclay, Orsay Cedex 91405, France} \par\nopagebreak
  \textit{E-mail:} \texttt{sylvain.crovisier@universite-paris-saclay.fr}
}
{
  \bigskip
  \footnotesize
	
  \textbf{Mauricio Poletti}: \textsc{Departamento de Matem\'atica, Universidade Federal do Cear\'a, Campus do PICI, Bloco 914, CEP 60455-760.  Fortaleza – CE, Brasil.}\par\nopagebreak
  \textit{E-mail:} \texttt{mpoletti@mat.ufc.br}
}}

\begin{document}

\begin{abstract}
We prove a generalization of a so called ``invariance principle" for partially hyperbolic diffeomorphisms: if an invariant probability measure has all its center Lyapunov exponents equal to zero then the measure admits a center disintegration that is invariant by stable and unstable holonomies.  This was known for systems admitting a foliation by compact center leaves, and we extend it to a larger class which contains discretized Anosov flows.

We use our result to classify measures of maximal entropy and study physical measures for perturbations of the time-one map of Anosov flows.
\end{abstract}

\maketitle


\section{Introduction}
The objective of this work is to study the relation between two objects that are on the core of hyperbolic dynamics: \emph{invariant foliations} and \emph{invariant measures}.
We can relate them by disintegrating the measures along the leaves of the foliations. 

The properties of disintegrated measures along the leaves of the foliations give important information on the ergodic properties of the dynamics.
An example of such relation is given by Ledrappier-Young~\cite{LY85b} equality which relates the entropy, the Lyapunov exponents and the dimension of the disintegration of invariant measure along the unstable manifolds. 

In partially hyperbolic dynamics the existence of a non-hyperbolic center direction leads to dynamical behaviors that may be very different from hyperbolic systems. The hyperbolicity along the center direction is measured by the center Lyapunov exponents. \emph{What happens when the center Lyapunov exponents do vanish?} 

It is expected that measures with zero center exponents satisfy some rigidity. Such a property was proved for linear cocycles by Ledrappier~\cite{Le86} and extended by Avila and Viana~\cite{Extremal} for smooth cocycles: they showed that the center disintegration of invariant measures with zero center exponents is invariant by a family of holonomies.
These results are known as ``Invariance Principle"\footnote{It is different from the Invariance Principle in probability.} (see Section~\ref{ss.invariant-measures}).

There are many applications of the Invariance Principle. To cite some: positivity of Lyapunov exponents \cite{Almost, AvDmZh-2023, ASV13, LMY2018, Pol16, ObPol}, existence of physical measures~\cite{ViY13}, properties of the measures of maximal entropy \cite{hertz-hertz-tahzibi-ures}, rigidity of the perturbations of time-one maps of Anosov flows \cite{AVW11}, rigidity of exponents for geodesic flows \cite{Butler-2018}, Zimmer's conjecture~\cite{Brown-Dm-Zh-2022}.  

One of the main hypothesis in \cite{Extremal} is that the system preserves a fiber bundle with smooth compact fibers. In particular to apply these results to partially hyperbolic systems, the center foliation is required to be compact and to form a fiber bundle.

In this work we prove a more general version of the Invariance Principle for partially hyperbolic systems whose center does not need to be compact nor to be a fiber bundle. This allows to extend many previous results about center exponents to more general partially hyperbolic dynamics. 

Before stating the setting and our main result, we present two applications about perturbations of time-one maps of Anosov flows.
The first one deals with measures of maximal entropy and follows from our Main Theorem and \cite{buzzi-fisher-tahzibi}. Compare with \cite{hertz-hertz-tahzibi-ures} which addresses the case of a center foliation which is a circle fiber bundle. 

\begin{ltheorem}\label{thm.MME.An.Fl}
If $(\phi_t)$ is a $C^r$ transitive Anosov flow ($2\leq r\leq \infty$) on a compact manifold $M$, then there is a $C^1$ open set $\mathcal{U}\subset\Dif^r(M)$ whose $C^r$ closure contains $\phi_1$, such that any $f\in \mathcal{U}$ admits exactly two ergodic measures of maximal entropy: one with positive center Lyapunov exponent and one with negative center Lyapunov exponent. Both measures are Bernoulli.  
\end{ltheorem}

We also prove some rigidity for maps close to $\phi_1$ whose measure of maximal entropy has vanishing center Lyapunov exponent (these maps are the time-one maps of topological flows), 
see Section~\ref{ss.consequences-mme} for the statements.

The second application discusses the physical measures, i.e. invariant measures having a large basin (see Section~\ref{ss.physical} for the precise definition).

\begin{ltheorem}\label{thm.phys.meas.close.anosov}
If $(\phi_t)$ is a $C^r$ transitive Anosov flow ($2\leq r\leq \infty$) on a compact manifold $M$, then there is a $C^1$ open set $\mathcal{U}\subset\Dif^r(M)$ whose $C^r$ closure contains $\phi_1$, such that any $f\in \mathcal{U}$ admits a unique physical measure. Its basin has full volume in $M$ and its center Lyapunov exponent is negative.
\end{ltheorem}
Compare with~\cite{ViY13} about partially hyperbolic diffeomorphisms whose center foliation is a circle fiber bundle and with~\cite{Dol04b} about perturbations of the time-one maps of $C^\infty$ geodesic flows. For more results on physical measures in our context see Section~\ref{ss.physical}.
%
\medskip

In the following, we consider a $C^1$ diffeomorphism $f$ on a closed manifold $M$,
which is \emph{partially hyperbolic}:
there exists an invariant splitting
\begin{equation*}
TM=E^s\oplus E^c\oplus E^u,
\end{equation*}
a Riemannian metric $\|\cdot\|$,
continuous functions $0<\nu<\gamma<{\hat\gamma}^{-1}<{\hat\nu}^{-1}$
with $\nu$, $\hat{\nu}<1$
such that,
for any unit vectors $v^s\in E^s(x)$, $v^c\in E^c(x)$, $v^u\in E^u(x)$,
$$\|Df(x)v^s \| < \nu(x) < \gamma(x) < \|Df(x)v^c \|  < {\hat{\gamma}(x)}^{-1}  < {\hat{\nu}(x)}^{-1} < \|Df(x)v^u\|.$$
All three sub-bundles $E^s$, $E^c$, $E^u$ are assumed to have positive dimensions.

We are aimed to describe invariant measures $\mu$ along the center direction, when the center Lyapunov exponents vanish.

\subsection{Invariant foliations and holonomies}
The stable and unstable bundles $E^s$ and $E^u$ are uniquely integrable and their integral manifolds form two transverse continuous foliations $\cW^s$ and $\cW^u$, whose leaves are immersed sub-manifolds with the same class of differentiability as $f$. These foliations are referred as the \emph{strong-stable} and \emph{strong-unstable} foliations. They are invariant under $f$; this means:
$$
f(\cW^s(x))= \cW^s(f(x)) \quand f(\cW^u(x))= \cW^u(f(x)),
$$
where $\cW^s(x)$ and $\cW^u(x)$ denote the leaves containing $x\in M$.
One says that $f$ is \emph{accessible} if every $x,y\in M$
can be connected by a path which is the union of finitely many $C^1$ paths tangent to $\cW^s$ or $\cW^u$ leaves.

We say that $f$ is \emph{dynamically coherent} if there also exist invariant foliations $\cW^{cs}$ and $\cW^{cu}$ tangent to the bundles $E^s\oplus E^c$ and $E^c \oplus E^u$ respectively: we call these foliations \emph{center-stable} and \emph{center-unstable}. The intersection of these foliations defines a \emph{center foliation} $\cW^c$:
for any $x\in M$, $\cW^c(x)$ is the connected component of $\cW^{cs}(x)\cap \cW^{cu}(x)$ which contains $x$. Observe that the center and unstable foliations sub-foliate the center-unstable manifolds.\!\!
\smallskip

The Riemannian metric induces a distance $d^c$ along the leaves of $\cW^c$.
\begin{definition}\label{d.quasi-isometric}
The partially hyperbolic diffeomorphism $f$ is \emph{quasi-isome\-tric in the center} if it is dynamically coherent and there exist $K_0\geq 1$, $c_0>0$ such that for every $x,y\in M$ satisfying $\cW^c(x)=\cW^c(y)$ and every $n\in \integer$,
$$
K_0^{-1}d^c(x,y)-c_0 \leq  d^c(f^n(x),f^n(y))\leq K_0 d^c(x,y)+c_0.
$$
\end{definition}
\noindent
This holds for instance when the center leaves are compact and form a fiber bundle.
There is an important class of partially hyperbolic diffeomorphisms, called \emph{discretized Anosov flows} \cite{BFFP, santiago} which are quasi-isometric in the center, but have non compact one-dimensional center leaves; for these systems, each center-leaf is individually fixed, i.e: $f(x)\in \cW^c(x)$ for every $x\in M$. Perturbations of the time-one map of Anosov flows are of this kind.

Given an arc $\gamma$ that connects two points $x,y$ inside a leaf of $\cW^u$,
one defines the unstable holonomy between neighborhoods of $x$ and $y$ inside $\cW^c(x)$ and $\cW^c(y)$.
The holonomy does not extend in general to the whole center leaves; indeed
for $z\in M$, the manifolds $\cW^c(z)$ and $\cW^u(z)$ may have several intersections.
However when $f$ is quasi-isometric in the center, this does not hold for a large set of points $z$
and there exist global holonomies. This is stated in the next theorem which it is interesting on its own.

\begin{ltheorem}\label{thm.global-holonomy}
Let $f$ be a partially hyperbolic, quasi-isometric in the center, $C^1$ diffeomorphism
and $\mu$ be an $f$-invariant measure.
Then there exists a full measure set $X\subset M$ such that for any $x,y\in X$ with $\cW^{cu}(x)=\cW^{cu}(y)$ the following holds:

For any $z\in \cW^c(x)$, the leaves $\cW^u(z)$, $\cW^c(y)$ intersect at a unique point, denoted by $\fh^u_{x,y}(z)$.
The map $\fh^u_{x,y}\colon \cW^c(x)\to \cW^c(y)$ is a homeomorphism.
\end{ltheorem}
The map $\fh^u_{x,y}$ is called \emph{unstable holonomy} between $\cW^c(x)$ and $\cW^c(y)$.

\subsection{Invariant measures under holonomies}\label{ss.invariant-measures}
If $\mu$ is a probability measure on $M$,
its Rokhlin disintegration induces a Radon measure $\mu^c_x$ along the leaf $\cW^c(x)$
of $\mu$-almost every point $x$, which is well defined up to a factor: when $\cW^c(x)=\cW^c(y)$,
there exists $K_{x,y}>0$ such that $\mu^c_x=K_{x,y}.\mu^c_y$.

We say that the center disintegration $\{\mu^c_x\}_{x\in M}$ of $\mu$ is \emph{invariant under unstable holonomies} (or \emph{u-invariant}) if for $\mu$-almost every $x,y$ satisfying $\cW^{cu}(x)=\cW^{cu}(y)$
there exists $K_{x,y}>0$ such that $(\fh^u_{x,y})_*(\mu^c_x)=K_{x,y}.\mu^c_y$,
where the holonomy map $\fh^u_{x,y}$ is uniquely defined by Theorem~\ref{thm.global-holonomy}.

The entropy of an $f$-invariant probability measure is denoted by $h(f,\mu)$.
The Ledrappier-Young entropy along $E^u$ is called \emph{unstable entropy}, denoted $h^u(f,\mu)$. One denotes analogously by $h^s(f,\mu)$ the \emph{stable entropy} along $E^s$.

Now we can state our main theorem.
\begin{maintheorem}
Let $f$ be a partially hyperbolic, quasi-isometric in the center, $C^1$ diffeomorphism
and let $\mu$ be an ergodic measure. If $h(f,\mu)=h^u(f,\mu)$, then the center disintegration $\{\mu^c_x\}_{x\in M}$ is u-invariant. 
\end{maintheorem}
By \cite{LY85b} and \cite{brown_21}, when $f$ has more regularity\footnote{We believe that this also holds for $C^1$ diffeomorphisms,
compare for instance with~\cite{Guo-Liao-Sun-Yang}.} we have the following result.
\begin{lcorollary}\label{cor.zero.exp}
Let $f$ be a partially hyperbolic, quasi-isometric in the center, $C^r$ diffeomorphism, $r>1$,
and let $\mu$ be an ergodic measure. If all the center Lyapunov exponents of $\mu$ are non-positive, then the center disintegration $\{\mu^c_x\}_{x\in M}$ is u-invariant. 
\end{lcorollary}

This kind of result is known as an ``Invariance Principle''. Ledrappier~\cite{Le86} proved a version for
the projective action of linear cocycles and invariant measures whose Lyapunov exponents coincide. This has been generalized by Avila and Viana~\cite{Extremal} for smooth cocycles. Their result applies to partially hyperbolic diffeomorphisms whose center leaves are compact and form a fiber bundle: in this setting, ergodic measures whose center Lyapunov exponents
are all non-positive have u-invariant center disintegration.
More recently Tahzibi and Yang~\cite{AliYang} have proved a version of the Invariance Principle whose statement involves the entropy:
for partially hyperbolic diffeomorphisms which are skew-products over an Anosov diffeomorphism,
an ergodic measure is u-invariant if and only if its unstable entropy coincides with the entropy of its projection in the base. This implies\footnote{In this setting where center leaves are compact, the statement in~\cite{AliYang} is an equivalence and (though it assumes $C^2$ regularity) is slightly stronger than our Main Theorem, since it involves the entropy of the projection instead of the full entropy $h(f,\mu)$.}
our Main Theorem when the center leaves are compact and form a circle fiber bundle.

The main novelty of our result is that we do not need any kind of compactess or fiber bundle structure of the center manifolds: we only need the quasi-isometric property in the center. This allows us to extend many results to more general partially hyperbolic systems (see the following sections to see some of them). At last, let us mention~\cite{ledrappier-xie} which proves an Invariance Principle for general measures,
beyond the partially hyperbolic setting, but assuming that the center Lyapunov exponents do not vanish
(the measure is non-uniformly hyperbolic): in this setting if the entropy of the measure coincides with the Ledrappier-Young entropy along a Pesin strong unstable lamination, then the unstable disintegration and the strong unstable disintegration of the measure coincide.

\begin{remark}
Recently~\cite{TahzibiZhang} 
constructed a volume-preserving derived from Anosov diffeomorphism, with absolutely continuous center unstable foliation, such that the volume has zero center Lyapunov exponents and a center disintegration which is not su-invariant. This contradicts the conclusion of Proposition~\ref{p.cont.dis} below and
shows that some quasi-isometric condition along the center is necessary for proving the u-invariance.
\end{remark}

The Invariance Principle is mainly used to establish the su-invariance of the center disintegration of a measure when the center Lyapunov exponents are zero. Note that the center disintegration is apriori only defined almost everywhere, but
in the next section we discuss the existence of a continuous extension.


\subsection{Continuous families of center disintegrations}
For measures who\-se center Lyapunov exponents vanish and having a local product structure,
Avila-Viana's Invariance Principle provides a continuous extension of the center disintegration over the support of the measure.
In our setting, we get a similar statement; however since the holonomies along invariant foliations are in general only defined locally,
we have to localize the center measures.

We first define the \emph{local unstable holonomies}:
there exists $\delta_0,\varepsilon_0>0$ such that for any $x,y\in M$ with $d(x,y)<\varepsilon_0$,
the plaques $B^{u}_{\delta_0}(x)$ and $B^{cs}_{\delta_0}(y)$ intersect at a unique point denoted by $\fh^u_y(x)$,
where $B^*_\delta(x)$ denotes the $\delta$-ball centered at $x$ inside the leaf $\cW^*(x)$, for $*\in\{s,c,u,cs,cu\}$.

\begin{definition}\label{d.local-center-measure}
Given $X\subset M$ measurable and $\delta>0$, a \emph{family of local center measures}, $\{\nu^c_x\}_{x\in X}$,
is a set of finite measures $\nu^c_x$ supported on $B^c_\delta(x)$
for each $x\in X$. 
The family $\{\nu^c_x\}_{x\in X}$
\emph{extends the center disintegration} $\{\mu^c_x\}$ of a measure $\mu$
if $\mu(X)=1$ and if there is $\varepsilon>0$ such that
for $\mu$-almost every $x\in X$ there exists $K_x>0$ satisfying
$\mu^c_x|_{B^c_{\varepsilon}(x)}=K_x.\nu^c_x|_{B^c_{\varepsilon}(x)}$.

It is continuous if $x\mapsto \nu_x^c$ is continuous for the weak-$\star$ topology.
\end{definition}
\medskip

Let $\cN^{cs\times u}_{\varepsilon}(z)$ be the product neighborhood of $z$ which is the image of
$B^{cs}_\varepsilon(z)\times B^{u}_\varepsilon(z)$ under the homeomorphism $(x,y)\mapsto \fh^u_y(x)$.
\begin{definition}\label{def.mes.prod.st}
A probability measure $\mu$ has \emph{local cs$\times$u-product structure} if there exists $\varepsilon>0$ such that,
for any $x\in \supp(\mu)$, the measure $\nu:=\mu|_{\cN^{cs\times u}_{\varepsilon}(z)}$ is equivalent to a product measure $\nu^{cs}\times \nu^{u}$
with respect to the product structure on $\cN^{cs\times u}_{\varepsilon}(z)$. See also Section~\ref{s.meas.prod.structure}.
\end{definition}
As we will see, local product structures are satisfied by natural classes of measures (equilibrium measures~\cite{buzzi-fisher-tahzibi,climenhaga-pesin-zelerowicz} and some u-Gibbs measure).

The following theorem can be applied to $C^r$ diffeomorphisms ($r>1$) and measures having vanishing center Lyapunov exponents and local cs$\times$u-product structure (as already mentioned before, vanishing exponents imply $h^s(f,\mu)=h^u(f,\mu)=h(f,\mu)$, see~\cite{LY85b,brown_21}). A more precise version will be stated in Section~\ref{ss.continuous-family}.

\begin{ltheorem}\label{thm.continuous.disint}
Let $f$ be a partially hyperbolic, quasi-isometric in the center, $C^1$ diffeomorphism
and let $\mu$ be an ergodic measure.

If $h^s(f,\mu)=h^u(f,\mu)=h(f,\mu)$ and if $\mu$ has local cs$\times$u-product structure, then there exists a
family of local center measures $\{\nu^c_x\}_{x\in \supp(\mu)}$ on the support of $\mu$
which is continuous and extends the center disintegration of $\mu$.
\end{ltheorem}

Let us mention that for conservative perturbations of the time-one maps of Anosov flows, Avila, Viana and Wilkinson~\cite{AVW11} constructed an artificial circle fiber bundle over $f$: although center leaves of $f$ are non-compact, this allowed them to apply \cite{Extremal} (or its version \cite{ASV13} for cocycles over conservative partially hyperbolic maps) and to establish the su-invariance of the center disintegration of the volume when its center Lyapunov  exponents vanish. When one considers measures that are not the volume, we could not work with such 
an artificial fiber bundle and we have to handle the disintegrations on the genuine non-compact center leaves.

\subsection{Consequence (1): Measures of maximal entropy (m.m.e.).}\label{ss.consequences-mme}

The variational principle asserts that the topological entropy of $f$ is equal to the supremum of the entropies of its invariant probabilities.
When $f$ is partially hyperbolic with one-dimensional center, there exists a probability measure of maximal entropy, i.e. which realizes the supremum~\cite{liao-viana-yang}. The case where the center foliation is a fibration with compact leaves has been studied 
in~\cite{hertz-hertz-tahzibi-ures}.

The case of perturbations of time-one maps of a transitive Anosov flows has been considered by 
Buzzi, Fisher and Tahzibi~\cite{buzzi-fisher-tahzibi}. They showed a dichotomy for diffeomorphisms inside some open sets arbitrarily $C^1$ close to the time-one maps of transitive Anosov flows: either all m.m.e. have vanishing center Lyapunov exponent or there exist exactly two ergodic m.m.e. one with positive and one with negative center Lyapunov exponent. Theorem~\ref{thm.MME.An.Fl} above improves this result.

A natural setting for this result are the already mentioned \emph{discretized Anosov flows}, 
i.e. partially hyperbolic diffeomorphisms (1) which are dynamically coherent, (2) whose center foliation is one-dimensional, (3) which act like a flow in the center:
there exists $L>0$ such that $f(x)\in \cW^c(x)$ and $d^c(x,f(x))<L$ for each $x\in M$.

Margulis has constructed~\cite{Mg70} measures of maximal entropy for the geodesic flow of manifolds with negative curvature: they are obtained from a family of measures carried by the unstable leaves, known as a Margulis system of measures. In \cite{buzzi-fisher-tahzibi} the authors extend this construction for $C^2$ discretized Anosov flows whose strong-stable and strong-unstable foliations are minimal:
if $\mu$ is a m.m.e. with non-positive center Lyapunov exponent, then its disintegration $\mu^u$ along the leaves of $\cW^u$ coincide with a system of Margulis measures; these measures are quasi invariant by $cs$-holonomies, and as a consequence $\mu$ has a local cs$\times$u-product structure.

We refine this result giving more information on the measure when the center Lyapunov exponent vanishes; this answers some of the open questions left in \cite[Questions 2 and 3]{buzzi-fisher-tahzibi}.
\begin{ltheorem}\label{thm.MME}
Let $f$ be a $C^2$ discretized Anosov flow such that $\cW^s,\cW^u$ are minimal,
and $\mu$ be an ergodic m.m.e. satisfying $h_{top}(f)=h^u(f,\mu)=h^s(f,\mu)$.
Then the center disintegrated measures $\mu^c_x$ do not contains atoms.

Moreover, if $f$ is accessible, then: (1) the m.m.e. is unique, (2) the disintegrated measures $\mu^c_x$ are equivalent to the Lebesgue measure,
(3) $f$ is the time-one map of a topological Anosov flow (whose orbits coincide with the $\cW^c$-leaves
and which are $C^n$, $n\in \natural$, if $f\in C^n$ along its orbits), (4) the center Lyapunov exponent of any invariant measure is zero.
\end{ltheorem}

The next statement shows the continuity of the m.m.e.
It is analogous to~\cite[Theorem B]{AliYang} (when the center leaves are compact).

\begin{lcorollary}\label{c.continuity.mme}
Let $f$ be $C^2$ discretized Anosov flow, such that $\cW^s,\cW^u$ are minimal,
and $(f_n)$ be $C^1$ diffeomorphisms converging to $f$ in $\Dif^1(M)$,
with ergodic measures $\mu_n$ such that $h(f_n,\mu_n)\to h_{top}(f)$
and $\lambda^c(\mu_n)\geq 0$.

Then, either $\lambda^c(\mu_n)\to 0$ and for all m.m.e. of $f$ the center exponent vanishes,
or $(\mu_n)$ converges to an ergodic m.m.e. $\mu$ satisfying $\lambda^c(\mu)> 0$.

In particular if $f$ has an ergodic m.m.e. which is hyperbolic
(i.e. its center Lyapunov exponent does not vanish), then there exists $\varepsilon>0$ such that any ergodic measure with entropy larger than $h_{top}(f)-\varepsilon$ is hyperbolic.
\end{lcorollary}

There are some interesting questions that emerge from this result.
\begin{question}
Let $\phi_1$ be the time-one map of a transitive Anosov flow which is not a suspension.
Does any diffeomorphism $C^1$-close to $\phi_1$ either have exactly two ergodic m.m.e. (each being hyperbolic)
or coincide with the time-one map of a topological Anosov flow? 
\end{question}

\begin{question}
Is this dichotomy also true for discretized Anosov flow?
\end{question}
We believe the first question is true and the second one is true in some $3$ manifolds, this is an ongoing project with Buzzi and Tahzibi.

\begin{question}
Is the flow in Theorem~\ref{thm.MME} smooth? 
\end{question}
The difficult part is to recover some smoothness on the direction transverse to the flow. In \cite{AVW11} this is done using that the map is volume preserving. We believe there should exist examples in our setting where $f$ is the time-one map of a topological flow that is not smooth.

\begin{question}
In the zero center exponent case (Theorem~\ref{thm.MME}), is $f$ Bernoulli?
\end{question}

\medskip

\subsection{Consequence (2):  Physical measures}\label{ss.physical}
An invariant measure $\mu$ is called \emph{physical} if its basin 
$$
B(\mu):=\left\{x\in M;\lim_{n\to +\infty} \frac{1}{n}\sum_{i=0}^{n-1}\int \varphi (f^i(x))=\int \varphi\mu,\forall\,\varphi\in C^0(M,\real)\right\}
$$
has positive Lebesgue measure.

Sinai~\cite{Si72}, Ruelle~\cite{Ru76b} and Bowen~\cite{Bow75a} have proved that $C^{r}$ uniformly hyperbolic diffeomorphisms (with $r>1$) have finitely many physical measures that describe the statistical behavior of Lebesgue almost every point. In this case the physical measures coincide with the SRB measures, i.e. with measures whose disintegrations along their unstable manifolds are absolutely continuous with respect to the Lebesgue measure.

Theorem~\ref{thm.phys.meas.close.anosov} above provides physical measures for some perturbations of the time-one maps of transitive Anosov flows.


For partially hyperbolic diffeomorphisms, measures having absolutely continuous disintegrations along the leaves of the strong unstable foliation
have been introduced and studied in~\cite{pesin-sinai} and are called \emph{u-Gibbs measures}: they are natural candidates to be physical measures~\cite{CYZ,HYY}. Assuming a weak contraction or expansion in the center, one can conclude that these systems admit physical measures~\cite{ABV,BV}, but there also exist transitive examples with no physical measures~\cite{CYZ}.

This question is better understood under a regularity condition on the center-stable foliation:
we say that $\cW^{cs}$ is \emph{absolutely continuous} if zero Lebesgue measure sets are preserved by holonomies along center-stable leaves between strong-unstable transversals. In this case, the u-Gibbs measure have a local cs$\times$u-product structure.

Viana-Yang have proved~\cite{ViY13} that for accessible partially hyperbolic diffeomorphisms whose center foliation
form a fiber bundle with circle center leaves and whose center-stable foliation is absolutely continuous, then every ergodic u-Gibbs measure is a physical; moreover there exist at most finitely many ones.
The next result extends some results of \cite{ViY13} to our setting and is used for proving Theorem~\ref{thm.phys.meas.close.anosov}.

\begin{ltheorem}\label{thm.phys.meas}
Let $f$ be a $C^r$ ($r>1$) accessible discretized Anosov flow. If $\cW^u$ is minimal and
$\cW^{cs}$ is absolutely continuous, any u-Gibbs measure is a physical measure whose basin has full volume.

Moreover it is unique and its center Lyapunov exponent is non-positive.
\end{ltheorem}

The $cs$-absolute continuity is essential in the proof to have a local $cs\times u$ product structure of the u-Gibbs measure and to apply Theorem~\ref{thm.continuous.disint}. 

\subsection{More about systems which are quasi-isometric in the center}
The quasi-isometric condition in the center is equivalent to require the existence of $0<r_1\leq r_2 \leq r_3$ such that for every $x\in M$ and $n\geq 1$ 
$$
B^c_{r_1}(f^n(x))\subset f^n(B^c_{r_2}(x))\subset B^c_{r_3}(f^n(x)).
$$

\begin{remark}
Only the first inclusion (called \emph{non-shrinking center condition}) is used in the proof of the existence of global holonomies (Theorem~\ref{thm.global-holonomy}), see Section~\ref{s.global.prod.structure}.
\end{remark}

As we mentioned, natural examples of such systems are (1) partially hyperbolic diffeomorphisms whose center foliation
is a fibration with compact leaves, (2) $C^1$ perturbations of the time-one maps of Anosov flows.

We remark that the quasi-isometric condition is also invariant under linear cocycle extension.
For instance, if $f\in\Dif^1(M)$ is partially hyperbolic, then any fiber-bunched $C^1$ map $A:M\to \GL$ induces a projective cocycle $F:M\times \proj\to M\times \proj$:
this is a partially hyperbolic diffeomorphism with center manifold $\cW^c_F=\cW^c_f\times \proj$.
Assuming that $f$ acts quasi-isometrically in the center, $F$ also acts quasi-isometrically in the center.

Many results based on the Invariance Principle discuss the generic positiveness of the upper Lyapunov exponents in the space of H\"older continuous fiber-bunched linear cocycles over some fixed dynamics with some hyperbolicity property, see for example \cite{Almost}, \cite{ASV13}, \cite{Pol16}.
\begin{question}
Do generic H\"older continuous fiber-bunched linear cocycles over partially hyperbolic quasi-isometric in the center $C^1$ diffeomorphisms have positive Lyapunov exponents?
\end{question}
The Main Theorem may be useful to answer this question when the measure has some local $cs\times u$ product structure.

\subsection{Strategy of the proof and organization of the paper}
Consider some cu-square $R$ with c$\times$u-product structure and the measure $\mu^{cu}_R$ on $R$ induced by the center-unstable disintegration of a measure $\mu$. The u-invariance of the family $\{\mu^c_x\}$ is equivalent for $\mu^{cu}_R$ to be a product measure $\mu^u\times \mu^c$ with respect to the product structure of $R$.
We want to prove that for any partition $\cP$ of any cu-square $R$ by vertical strips (i.e. a partition on the u-direction c-saturated inside $R$), the measure $\mu_R^{cu}(P)$ of every strip $P\in \cP$ is equal to the measure of the u-disintegration $\mu^u_x(P)$ for every $P\in \cP$.

This is done by comparing the entropy of $\cP$ and the entropy of the induced partitions on unstable leaves. We call this difference \emph{transverse entropy} of $\cP$, denoted by $\Gap(\cP)$ (see Section~\ref{s.trans.entropy}). By Jensen inequality, $\Gap(\cP)$ is non-negative; it vanishes exactly when $\mu^{cu}(P)=\mu^u_x(P)$ for every $P\in \cP$, $x\in R$.

Let us forget the stable direction for now and let us suppose that there exists some cu-square $R$ and a partition $\cP$ as before with positive transverse entropy. Our goal is to construct a sequence of partitions by vertical strips $\cP_n$ such that, for a large measure set of points $x$, the iterate $f^n(\cP_n(x))$ contains the original cu-square $R$, and the partition $f^n\cP_{n+1}$ refines $\cP$. This will imply $\Gap(\cP_n)\geq n\Gap(\cP)$. We also check that the sequence $(\cP_n)$ satisfies $\limsup_{n\to \infty}\frac{1}{n}\Gap(\cP_n)\leq h_\mu(f)-h^u_\mu(f)$. Consequently, if $h_\mu(f)=h^u_\mu(f)$ we get a contradiction.

Let us discuss the construction of the partitions $\cP_n$. As the center direction is not necessarily expanded it is generally difficult to find iterates such that $f^n(\cP_n(x))$ contains the cu-square $R$.
For that reason, we actually define $\cP_n$ in a set which is larger in the center and contains the original cu-square.
Choosing sets which are sufficiently large along the center allows to apply the quasi-isometric property and ensure the covering property for many iterates. Here appears a technical difficulty: we a priori have a good product structure only for sets which are small in the center direction. For that reason, we have to extract measurable subsets with full measure, which have a good product structure and a large center size, as stated in Theorem~\ref{thm.global-holonomy}: this is done in Section~\ref{s.global.prod.structure}.

Another problem is that some atoms of the partitions $\cP_n$ are very small in the u-direction, even after iterations:
the unstable expansion under $f^n$ may not be enough to ensure that the images of the elements of $\cP_n$ cover the square. To overpass this difficulty, we need to require a property of small boundary for the unstable partition $\cP$:
for a large set of points $x$, the $f^{n-1}$ iterate of $\cP_n(x)$ achieves some uniform size along the unstable direction:
this is done in Section~\ref{s.partition.small.bound}.

The precise construction of the partitions $\cP_n$ is done in Section~\ref{s.gaps.entropy}, where the Main Theorem is proved.
The proof of Theorem~\ref{thm.continuous.disint} is in Section~\ref{s.meas.prod.structure}, those of Theorems \ref{thm.MME.An.Fl} and \ref{thm.MME} are contained in Section~\ref{s.maximal.entropy} and Theorems~\ref{thm.phys.meas.close.anosov} and~\ref{thm.phys.meas} are proved in Section~\ref{s.phys.meas}.

\subsection{Notations and general definitions}
\subsubsection{Balls}\label{ss.ball}
As before we denote by $d^*$ the distance along the leaves of $\cW^*$, $*\in\{u,c,s,cs,cu\}$, 
that are induced by the Riemanniann metric.

Let us recall that for any $x\in M$ and $*\in\{u,c,s,cs,cu\}$, we denote by $B^*_\varepsilon(x)$ the $\varepsilon$-ball centered at $x$ inside $\cW^*(x)$.

For any $l,\delta>0$, we define
$$B^{cs}_{l,\delta}(x)=\cup\{B^c_l(y):\; y\in B^s_\delta(x)\},$$
$$B^{cu}_{l,\delta}(x)=\cup\{B^c_l(y):\; y\in B^u_\delta(x)\}.$$

It will be convenient to fix a scale $\varepsilon_0>0$
and to define the local plaques $\cW^*_{loc}(x):=B^*_{\varepsilon_0}(x)$. Choosing $\varepsilon_0$
small ensures that for any $x,y\in M$
the intersections $\cW^s_{loc}(x)\cap \cW^{cu}_{loc}(y)$, $\cW^{cs}_{loc}(x)\cap \cW^{u}_{loc}(y)$
contain at most one point.

\subsubsection{Disintegrations} Given a probability measure $\mu$,
its Rohklin disintegration along the leaves of the foliation $\cW^*$ is a family of measures denoted by $\{\mu^*_z\}$ (each measure is defined up to a factor and is finite on subsets of $\cW^*_z$ that are compact for the intrinsic topology).

When $R\subset M$ is a measurable subset with positive measure and $\{\cW^*_R\}$ is a measurable partition of $R$
induced by the foliation $\cW^*$, the measures $\mu^*_z$ can be normalized:
to almost every point $z$, one associates a probability measure $\mu^*_{R,z}$ on $\cW^*_R(z)$ such that:
\begin{itemize}
\item[--] $\mu^*_{R,z}$ is constant on each set $\cW^*_R(z)$,
\item[--] for each measurable set $A$, the map $z\mapsto \mu^*_{R,z}(A)$ is measurable,
\item[--] $\int_R \mu^*_{R,z}(A)d\mu=\mu(A)/\mu(R)$.
\end{itemize}

\subsubsection{Holonomies}
Let us consider some foliations $\cW^*$ and two transverse sections $S_1,S_2$.
A map $\fh^*\colon S_1\to S_2$ is an \emph{holonomy map along the leaves of $\cW^*$} if there exists
a continuous map $H:S_1\times [0,1]\to M$ such that for every $(x,t)\in S_1\times [0,1]$,
one has $H(x,t)\in \cW^*(x)$ and $H(x,1)=\fh^*(x)$.
Holonomies along $\cW^*(x)$ are also called $*$-holonomies.


\section{Global product structure}\label{s.global.prod.structure}
In this section we consider a dynamically coherent partially hyperbolic diffeomorphism that satisfies half of Definition~\ref{d.quasi-isometric}, it \emph{does not shrink the center}, i.e. there exist $K_0,c_0>0$ such that for every $x,y\in M$ in a same center leaf, and for every $n\geq0$, $$
K_0^{-1}d^c(x,y)-c_0 \leq  d^c(f^n(x),f^n(y)).
$$
We prove Theorem~\ref{thm.global-holonomy} and build sets with a \emph{fibered c$\times$u-product structure}
(see Section~\ref{ss.fibered-cu}).

\subsection{Existence of global c$\times$u-product structure}
Every $\cW^{cu}$ leaf is sub-foliated by the transverse foliations $\cW^c$ and $\cW^u$, hence has a local product structure.
In general this structure does not extend globally, see Figure~\ref{fig.non.product}.

We say that a set $X$ has a \emph{global c$\times$u-product structure} if for any $x,y\in X$ in a same $\cW^{cu}$ leaf,
the intersection $\cW^{c}(x)\cap \cW^{u}(y)$ contains exactly one point, which furthermore belongs to $X$.
The main result of this section implies Theorem~\ref{thm.global-holonomy}
and asserts that, for any invariant probability measure, there exists a measurable set with full measure which has a global c$\times$u-product structure.

\begin{figure}[h]
\centering
\includegraphics[scale= 0.5]{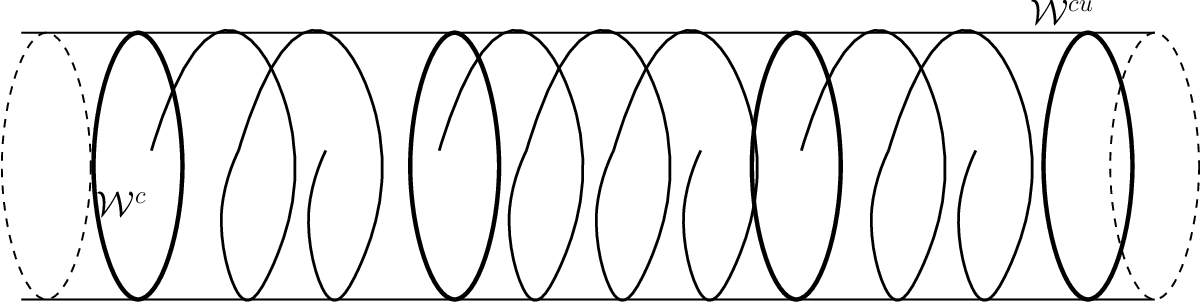}
\caption{A leaf $\cW^{cu}$ without global product structure.}
\label{fig.non.product}
\end{figure}

\begin{theorem}\label{t.prod.struct}
Let $f$ be a dynamically coherent partially hyperbolic diffeomorphism which does not shrink the center
and $\mu$ be an $f$-invariant probability.
Then there is a total measure set $X$, saturated by center leaves such that
for any $x\in X$, and any $y\in\cW^{cu}(x)$, the unstable leaf of $y$ intersects the 
center leaf of $x$ in a unique point.

In particular, for each $x\in X$ there exists a global continuous projection by u-holonomy
$\fh^u\colon \cW^{cu}(x)\to \cW^c(x)$. It induces for each $x,y\in X$ in a same cu-leaf
a homeomorphism $\fh^u_{x,y}\colon \cW^c(x)\to \cW^c(y)$.
\end{theorem}

\subsection{su-sections with a product structure}
Between two unstable plaques that are close enough, there exists a well-defined center-stable holonomy
(the holonomy that minimizes the distance inside center-stable leaves).

\begin{proposition}\label{p.cs-u.product}
For any $x\in M$ and $r>0$, there are $\delta, \delta'>0$ such that
for any $y$ that is $\delta$-close to $x$ there exists a unique holonomy
$\fh^{cs}_{x,y}\colon B^u_r(x)\to \cW^u(y)$ defined by a continuous map $H:B^u_r(x)\times [0,1]\to M$ satisfying:
\begin{itemize}
\item[--] its image contains $y$,
\item[--] for each $z\in B^u_r(x)$, the distance $d^{cs}(z,H(z,t))$ is smaller than $\delta'$.
\end{itemize}
Moreover, for any $y,y'$ $\delta$-close to $x$, the sets $\fh^{cs}_{x,y}(B^u_r(x))$ and $\fh^{cs}_{x,y'}(B^u_r(x))$
are disjoint or equal.
\end{proposition}
\begin{proof}
Take $B^u_r(x)$ and for each $\delta'>0$ take $V_{\delta'}(x)=\cup_{y\in B^u_r(x)}B^{cs}_{\delta'}(y)$. By transversality of $cs$ and $u$ foliations, for $\delta'$ sufficiently small $V_\delta'(x)$ is a tubular neighborhood of $B^u_r(x)$. Hence $B^{cs}_{\delta'}(z)\cap B^{cs}_{\delta'}(y) \neq \emptyset$ for $z,y\in B^u_r(x)$ if and only if $y=z$. 
By continuity of the $u$ foliation and since $W^u(x)$ is homeomorphic to the space $\real^{d^u}$,
 there exists $\delta>0$ and $r'>r$ such that if $d(x,y)<\delta$ then $B^u_{r'}(y)$ intersects every $B^{cu}_{\delta'}(z)$, with $z\in B^u_r(x)$. Moreover (having chosen $\delta'>0$ small) we can assume that $B^u_{2r'}(y)\cap B^{cu}_{\delta'}(z)$
 contains only one point $\fh^{cs}_{x,y}(z)$.

One then defines a continuous map $H:B^u_r(x)\times [0,1]\to M$ satisfying $H(z,0)=z$ and $H(z,1)=\fh^{cs}_{x,y}(z)$,
by considering for each $z\in B^u_r(x)$ the geodesic inside $B^{cs}_{\delta'}(z)$ that connects $z$ and $\fh^{cs}_{x,y}(z)$
(it is parametrized by $[0,1]$ with constant speed). Note that any cs-holonomy
from $B^u_r(x)$ to $\cW^u(y)$,
defined by a map $H':B^u_r(x)\times [0,1]\to M$ satisfying $H'(z,t)\in B^{cs}_{\delta'}(z)$ for every $t\in [0,1]$
coincides with $\fh^{cs}_{x,y}(z)$, by uniqueness of the intersections $B^u_{2r'}(y)\cap B^{cu}_{\delta'}(z)$.

Let us now consider two points $y,y'$ that are $\delta$ close to $x$ and satisfying
$B^u_{r'}(y)\cap B^u_{r'}(y')\neq \emptyset$. For each $z\in B^u_r(x)$,
observe that $\fh^{cs}_{x,y}(z),\fh^{cs}_{x,y'}(z)\in B^u_{2r'}(p)\cap B^{cs}_{\delta'}(z)$.
This again implies that $\fh^{cs}_{x,y}(z)=\fh^{cs}_{x,y'}(z)$.
\end{proof}

The previous proposition justifies the following notation.
\bigskip

\paragraph{\bf Notation.}
For any $x\in M$, $r>0$, let us consider $\delta>0$ small and the holonomies
$\fh^{cs}_{x,y}$ as in the statement of Proposition~\ref{p.cs-u.product}.
Then we denote
$$\cW^{su}_{r,\delta}(x):=\bigcup_{y\in B^s_\delta(x)} \fh^{cs}_{x,y}(B^u_r(x)),$$
The set $\cW^{su}_{r,\delta}(x)$ has a \emph{cs$\times$u-product structure}: it is homeomorphic to the product $B^s_\delta(x)\times B^u_r(x)$ through the map $\Psi\colon (y,z)\mapsto \fh^{cs}_{x,y}(z)$.
\bigskip

\paragraph{\bf Definitions.}
To any points $p=\fh^{cs}_{x,y}(z)$ and $q=\fh^{cs}_{x,y'}(z')$ in $\cW^{su}_{r,\delta}(x)$,
one associates the \emph{local product} $[p,q]:=\fh^{cs}_{x,y}(z')$.
It belongs to $\cW^u(p)\cap \cW^{cs}(q)$.

\noindent
A subset $Z\subset \cW^{su}_{r,\delta}(x)$ has a \emph{cs$\times$u-product structure} if
$$\forall p,q\in Z, \quad [p,q]\in Z.$$
One sets $\cW^u_{Z}(p)=\{q\in Z: [p,q]=q\} \text{ and } \cW^{cs}_Z(p)=\{q\in Z: [q,p]=q\}.$

\subsection{Sets with a fibered c$\times$u-product structure}\label{ss.fibered-cu}
The next key lemma does not essentially depend on the dynamics.

\begin{lemma}\label{l.small.product.structure}
Under the setting of Theorem~\ref{t.prod.struct},
for any $0<L^-<L^+$, there exist $x_0\in \supp(\mu)$, $L\in (L^-,L^+)$, $r>0$ small such that for any $\delta>0$ small,
the set
$Z:=\{z\in W^{su}_{\delta,\delta}(x_0)\colon\; \operatorname{Card}(B^c_{10L}(z)\cap W^{su}_{r,r}(x_0))=1\}$
satisfies:
\begin{enumerate}[(i)]
\item $Z$ has a cs$\times$u-product structure.
\item $B^c_{4L}(z)$ and $B^c_{4L}(z')$ are disjoint for any $z\neq z'$ in $Z$.
\item $\mu(\cup_{z\in Z\cap U} B^c_\ell(z))>0$, for any $\ell>0$ and any neighborhood $U$ of $x_0$.
\item $\operatorname{Card}(B_r^u(y)\cap B^c_{2L}(z_2))=1$, for any $z_1\in Z$, $z_2\in \cW^u_{Z}(z_1)$, $y\in B^c_L(z_1)$;
this gives a  homeomorphism $\fh^u_{z_1,z_2}$ from $B^c_{L}(z_1)$ to a subset of $B^c_{L+1}(z_2)$ containing $B^c_{L-1}(z_2)$.
\item $\operatorname{Card}(B_r^s(y)\cap B^c_{2L}(z_2))=1$, for any $z_1\in Z$, $z_2\in \cW^s_{Z}(z_1)$, $y\in B^c_L(z_1)$; this gives a  homeomorphism $\fh^s_{z_1,z_2}$ from $B^c_{L}(z_1)$ to a subset of $B^c_{L+1}(z_2)$ containing $B^c_{L-1}(z_2)$.
\end{enumerate}
\end{lemma}

\begin{figure}[h]
\centering
\includegraphics[scale= 0.6]{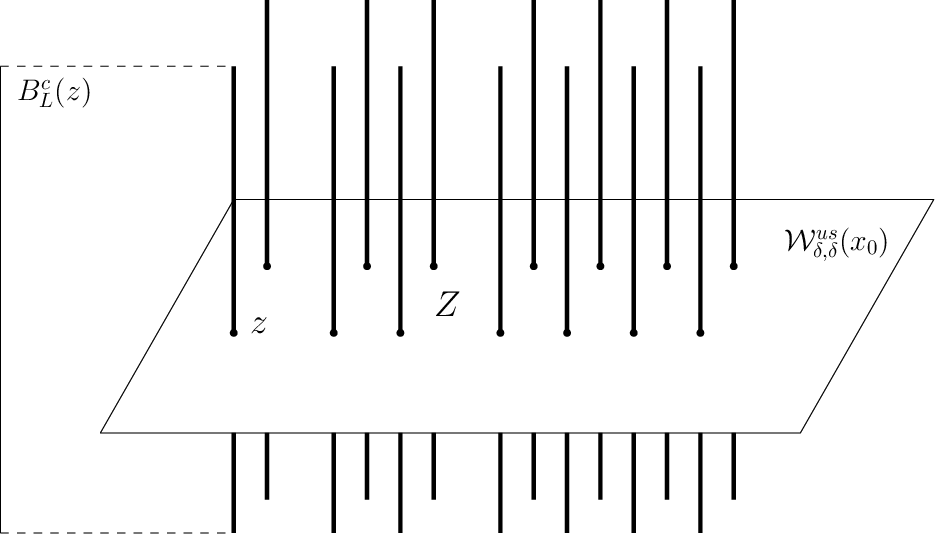}
\caption{The set $\cup_{z\in Z}B^c_{L}(z)$.}
\label{fig.Z}
\end{figure}

\begin{proof} We start with preliminary constructions.
We take $x_0\in \supp(\mu)$, $L\in (L^-,L^+)$ and $r,\gamma>0$ small such that:
\begin{itemize}
\item[(a)] $\cW^{su}_{r,r}(x_0)\cap B^c_{10L^+}(x_0)=\{x_0\}$,
\item[(b)] $d(\cW^{su}_{r,r}(x_0), \partial B^c_{10L}(z)>\gamma$ for $z\in \cW^{su}_{r,r}(x_0)$,
\item[(c)] $2\operatorname{diam}(\cW^{su}_{r,r}(x_0))<\gamma$,
\item[(d)] any 
injective holonomy $\fh\colon B^c_{L}(x)\to B^c_{2L}(y)$ along $\cW^u$ (resp. $\cW^s$), such that
$x,y\in \cW^{su}_{r,r}(x_0)$ and $d^u(\fh(z),z)<r$ (resp. $d^s(\fh(z),z)<r$) for each $z\in B^c_{L}(x)$, satisfies $B^c_{L-1}(y)\subset \fh(B^c_{L}(x)) \subset B^c_{L+1}(y)$.
\end{itemize}
Let us define $\mathfrak{T}:=\cup_{y\in B^c_{11L}(x_0)}\cW^{su}_{r,r}(y)$.
For each $p\in \mathfrak{T}$, let $\cW^{su}_\mathfrak{T}(p):=\cW^{su}_{r,r}(\pi(p))$, $\cW^u_\mathfrak{T}(p):=\cW^u_{loc}(p)\cap \cW^{su}_\mathfrak{T}(p)$ and  $\cW^{ws}_\mathfrak{T}(p):=\cW^{cs}_{loc}(p)\cap \cW^{su}_\mathfrak{T}(p)$. Since $\pi^{-1}(\pi(p))=\cW^{su}_\mathfrak{T}(p)$ has a cs$\times$u-product structure,
one can also write
$\cW^u_\mathfrak{T}(p)=\cW^u_{\pi^{-1}(\pi(p))}(p)$ and
$\cW^{ws}_\mathfrak{T}(p)=\cW^{cs}_{\pi^{-1}(\pi(p))}(p)$.
Up to reducing $r>0$, the set $\mathfrak{T}$ is a tubular neighborhood of $B^c_{11L}(x_0)$, i.e. for any $p\in \mathfrak{T}$:
\begin{itemize}
\item[(e)] there exists a unique point $\pi(p)\in B^c_{11L}(x_0)$, such that
$p\in \cW^{su}_{r,r}(\pi(p))$.
\end{itemize}
The map $\pi\colon \mathfrak{T}\to B^c_{11L}(x_0)$ is continuous.
There are $\delta,\varepsilon\in (0,r)$ such that:
\begin{itemize}
\item[(f)] $B^c_{L}(z)\subset \mathfrak{T}$, for every $z\in \cW^{su}_{\delta,\delta}(x_0)$,
\item[(g)] for $*=ws\text{ or }u$, and for every $z,z'\in \cW^{su}_{\delta,\delta}(x_0)$ with $z'\in \cW^*_{\mathfrak{T}}(z)$,
one has $\cW^*_{\mathfrak{T}}(y)\cap B^c_{10L+\gamma/2}(z')\neq \emptyset$, for every $y\in B^c_{10L}(z)$,
\item[(h)] $d^c(x,y)=\varepsilon \Rightarrow
\tfrac{9\varepsilon}{10}\leq d^c(\pi(x),\pi(y))\leq \tfrac{11\varepsilon}{10}$, for every $z\in \cW^{su}_{\delta,\delta}(x_0)$
and $x,y\in B^c_{L}(z)$.
\end{itemize}
Note that for $z\in \cW^{su}_{\delta,\delta}(x_0)$, the restriction
of $\pi$ to $B^c_{L}(z)$ is locally invertible, since it is the composition
of a u-holonomy with a s-holonomy.
Taking $\delta$ smaller, these properties imply for $z\in \cW^{su}_{\delta,\delta}(x_0)$ and $y\in B^c_{3L}(z)$:
\begin{itemize}
\item[(i)] for any path $\psi\subset B^c_{3L}(z)$ containing $z$
and length $<3L$, $\pi(\psi)$
is homotopic (endpoints fixed) in $B^c_{4L}(x_0)$
to an arc of length $<4L$,
\item[(j)] any path $\psi'\subset B^c_{11L}(x_0)$ containing $\pi(y)$ and length $<4L$, $\pi(\psi)$
has a continuous lift $\widetilde \psi\subset B^c_{10L}(z)$ for $\pi$, which is homotopic (endpoints fixed) in $B^c_{10L}(z)$
to an arc of length $<5L$.
\end{itemize}
\medskip

We then define as in the statement of the lemma:
$$Z:=\{z\in W^{su}_{\delta,\delta}(x_0)\colon\;\; \operatorname{Card}(B^c_{10L}(z)\cap W^{su}_{r,r}(x_0))=1\},$$
and check Items (i), (ii), (iv) and (v). In order to check Item (iii), one may need to change the point $x_0$
(and hence $L,r,\gamma,\delta,\varepsilon$), as we explain below.
\bigskip

\paragraph{\emph{Item (i).}}
We first prove that $Z$ has a cs$\times$u-product structure.
\begin{claim}
For any $z\in Z$, $z'\in \cW^*_{\mathfrak{T}}(z)\cap  \cW^{su}_{\delta,\delta}(x_0)$
and $*=ws\text{ or }u$ we have
$\cW^{su}_{r,r}(x_0)\cap B^c_{10L}(z')\subset \cW^*_{\mathfrak{T}}(z)$.
\end{claim}
\begin{proof}
Let $z''\in \cW^{su}_{r,r}(x_0)\cap B^c_{10L}(z')$.
By Property (g) above, there exists a point $z'''\in B^c_{10L+\gamma/2}(z)\cap \cW^*_{\mathfrak{T}}(z'')$.
Since $z\in Z$, we have either $z'''=z$ or $d^c(z''',z)\geq 10L$.
Observe that the latter can not happen because this will imply (with (c)) that there exists $y\in \partial B^c_{10L}(z)$ with $d^c(z''',y)\leq \gamma/2$, so 
$$
d(\cW^{su}_{r,r}(x_0), \partial B^c_{10L}(z))\leq d(\cW^{su}_{r,r}(x_0),y)\leq \gamma/2+\operatorname{diam}\cW^{su}_{r,r}(x_0)<\gamma,
$$
contradicting Property (b). Consequently $z'''=z$
and hence $z\in \cW^*_{\mathfrak{T}}(z'')$. We have thus proved $z''\in \cW^*_{\mathfrak{T}}(z)$ as announced.
\end{proof}
Let us take $z,z'\in Z$ and consider $z''= [z,z']$.
For any point $y\in B^c_{10L}(z'')\cap \cW^{su}_{r,r}(x_0)$
the claim applied twice implies that $y\in \cW^u_{\mathfrak{T}}(z)\cap \cW^{ws}_{\mathfrak{T}}(z')$,
so $y=z''$.
We have thus proved that
$\operatorname{Card}(B^c_{10L}(z'')\cap \cW^{su}_{r,r}(x_0))=1$ and then $z''\in Z$, concluding the proof of Item (i).
\bigskip

\paragraph{\emph{Item (ii).}} It is a direct consequence of the definition of $Z$.
\bigskip

\paragraph{\emph{Item (iv).}} Its proof is based on the next property.
\begin{claim}
For every $z\in Z$, the projection $\pi:B^c_{3L}(z)\to B^c_{10L}(x_0)$ is injective.
\end{claim}
\begin{proof}
Let us assume by contradiction that there exist
$z\in Z$ and $y_1\neq y_2\in B^c_{3L}(z)$ such that $\pi(y_1)=\pi(y_2)=:y'$.
There exists a geodesic path $\psi_1\subset B^c_{3L}(z)$ from $z$ to $y_1$ whose length is smaller than $3L$. See Figure~\ref{fig.projection}.
\begin{figure}[h]
\centering
\includegraphics[scale= 0.6]{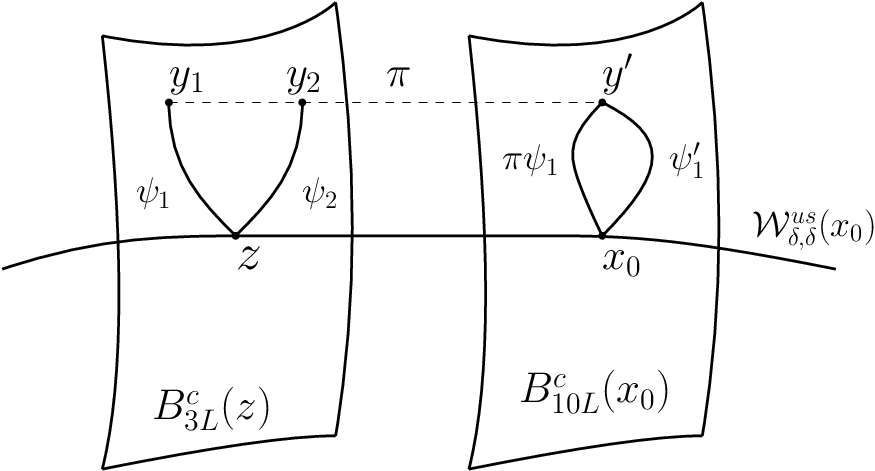}
\caption{Proof of Lemma~\ref{l.small.product.structure}, Item (iv)}
\label{fig.projection}
\end{figure}

The path $\pi(\psi_1)$ connects $x_0$ to $y'$
and, by Property (i), is homotopic (enpoints fixed) to an arc $\psi'_1$ with length smaller than $4L$.
By Property (j), it admits a lift $\widetilde \psi_1$ containing $y_2$ which is
homotopic (endpoints fixed) in $B^c_{10L}$ to an arc $\psi_2$ with length smaller than $5L$.
Note that $\psi_2$ is contained in a $B^c_{8L}(z)$ and connects $y_2$ to some point in $(\pi|_{B^c_{10L}(z)})^{-1}(x_0)$. By definition of $Z$ this endpoint is necessarily $z$, hence $\psi_2$ actually contained in $B^c_{4L}(z)$. 

We deduce that $\psi_1$ and $\widetilde \psi_1$ have the same endpoint $z$.
The homotopy between $\pi(\psi_1)$ and $\psi'_1$ can thus be lifted as an homotopy (endpoints fixed).
between $\psi_1$ and $\widetilde \psi_1$. Consequently the arcs $\psi_1,\widetilde \psi_1$ are homotopic and
have the same endpoints. This implies $y_1=y_2$, a contradiction.
\end{proof}

Let $z_1\in Z$, $z_2\in \cW^u_{Z}(z_1)$ and $y\in B^c_L(z_1)$.
The points in $B_r^u(y)\cap B^c_{2L}(z_2)$ belong to $\mathfrak{T}$ (by Property (f))
and have the same projection by $\pi$ (since they are all contained in a ball $B_r^u(y)$).
The injectivity in the previous claim implies that $B_r^u(y)\cap B^c_{2L}(z_2)$
contains at most one point and by Property (g) the intersection is non-empty.
Hence $\operatorname{Card}(B_r^u(y)\cap B^c_{2L}(z_2))=1$.
By Property (d), we have that $B^c_{L-1}(z_2)\subset\fh^u_{z_1,z_2}(B^c_{L}(z_1))\subset B^c_{L+1}(z_2)$.
Item (iv) is now proved.
\bigskip

\paragraph{\emph{Item (v).}} It has an analogous proof: observe that $\operatorname{Card}(B_r^s(y)\cap B^c_{2L}(z_2))>1$ implies $\cW^{su}_{\mathfrak{T}}(y)\cap B^c_{3L}(z_2)>1$.
\bigskip

\paragraph{\emph{Item (iii).}}
If this item is not satisfied, there exist $\ell>0$ and a neighborhood $U$ of $x_0$ such that $\mu(\cup_{z\in Z\cap U} B^c_\ell(z))=0$. Take a new point $x_1$ in the support of $\mu$, close to $x_0$, such that $x_1\in B^c_{L/3}(z)$ for some $z \in\cW^{su}_{\delta,\delta}(x_0)\setminus Z$. In particular
$B^c_{10L+L/3}(x_1)$ intersects $\cW^{su}_{r,r}(x_0)$ at least twice.

With the new point $x_1$, we repeat the constructions done for $x_0$.
Note that one can modify slightly $L$ (and hence $\mathfrak{T}$)
but keep Properties (a-j) we have obtained for $x_0$.
One can in this way select new numbers $r_1,\gamma_1,\delta_1$ and build new sets
$\mathfrak{T}_1, Z_1$ so that Properties (a-j) hold for $x_1$.
By choosing $r_1$ small enough, one can furthermore require the following additional property:
\begin{itemize}
\item[(k$_1$)]
every $y\in \cW^{su}_{r_1,r_1}(x_1)$ belongs to some $B^c_{L/2}(z)$ with $z\in \cW^{su}_{r,r}(x_0)$;
furthermore $B^c_{10L}(z)$ intersects $\cW^{su}_{r,r}(x_0)$ twice.
\end{itemize}

While Item (iii) is not satisfied,
one repeats inductively the constructions and build a sequence of points $x_n$ in the support of $\mu$
and numbers $\delta_n,r_n$.
Moreover one can require that:
\begin{itemize}
\item[(k$_n$)]
every $y\in \cW^{su}_{r_n,r_n}(x_n)$ belongs to a $B^c_{L/2^n}(z)$, $z\in \cW^{su}_{r_{n-1},r_{n-1}}(x_{n-1})$;
furthermore $B^c_{10L}(z)$ intersects $\cW^{su}_{r_{n-1},r_{n-1}}(x_{n-1})$ twice.
\end{itemize}
By construction the points $y\in \cW^{su}_{r_n,r_n}(x_n)$ belong to some $B^c_{\tfrac{2^n-1}{2^n}L}(z)$,
where $z\in \cW^{su}_{r,r}(x_0)$.
Moreover by Property (k$_n$),
the plaque $B^c_{10L+\tfrac{2^n-1}{2^n}L}(z)$ intersects $\cW^{su}_{r,r}(x_0)$
at least one more time than the plaque $B^c_{10L+\tfrac{2^{n-1}-1}{2^{n-1}}L}(x_{n-1})$
One deduces that
$B^c_{10L+\frac{2^{n}-1}{2^{n}}L}(x_{n})$ intersects $\cW^{su}_{r,r}(x_0)$ at least $n+1$ times.

The intersection points of a leaf $\cW^c(z)$ with $\cW^{su}_{r,r}(x_0)$ are separated from each other
by a uniform distance $2\varepsilon_0>0$ inside $\cW^c(z)$.
Since the volume of $B^c_{10L+\frac{2^{n}-1}{2^{n}}L+\varepsilon_0}(x_{n})$
is uniformly bounded in $n$, the number of its intersection points with $\cW^{su}_{r,r}(x_0)$ is bounded.
This shows that the construction has to stop after some step $n$.
We then replace $x_0,r,\delta,Z$ by $x_n,r_n,\delta_n$.
During the construction the number $L$ has slightly changed also.
Item (iii) is now satisfied, while Items (i), (ii), (iv), (v) remain unchanged.
This ends the proof of Lemma~\ref{l.small.product.structure}.
\end{proof}

\subsection{Proof of Theorem~\ref{t.prod.struct}}
We first prove that for each $x\in M$ and any $y\in \cW^{cu}(x)$,
the intersection $\cW^c(x)\cap \cW^u(y)$ is non-empty.
\begin{claim}
For any $x\in M$,
the union $\cW^u\cW^c(x)$ of the leaves $\cW^u(z)$ for $z\in \cW^c(x)$
coincides with $\cW^{cu}(x)$.
\end{claim}
\begin{proof}
Note that it is enough to prove that there exists $\delta>0$ such that
for any $x\in M$ and any $y\in \cW^u\cW^c(x)$, the $\delta$-neighborhood of $y$ in
$\cW^{cu}(x)$ is also contained in $\cW^u\cW^c(x)$.
The uniform transversality of the foliations $\cW^u$ and $\cW^c$ ensures that
for $\varepsilon>0$ small enough and for any $x$, the $\varepsilon$-neighborhood of $\cW^c(x)$
in $\cW^{cu}(x)$ is contained in $\cW^u\cW^c(x)$.

Let us fix $x\in M$ and $y\in \cW^u\cW^c(x)$.
Since $f$ does not shrink the center, for $n$ large enough,
the preimage of the $\delta$-neighborhood of $y$ by $f^{n}$ in $\cW^{cu}(x)$
is contained in the $\varepsilon$-neighborhood of $\cW^c(f^{-n}(x))$,
hence in $\cW^u\cW^c(f^{-n}(x))$. By invariance of the foliations,
the $\delta$-neighborhood of $y$ in $\cW^{cu}(x)$
is contained in $\cW^u\cW^c(x)$ as required.
\end{proof}

Let us consider the invariant set $X$ of points $x$ which satisfy the conclusion of Theorem~\ref{t.prod.struct}.
It is enough to prove that it has positive $\mu_0$-measure for any ergodic measure $\mu_0$.
Indeed if $X$ does not have not full measure for some invariant measure $\mu$, one would find
an ergodic component of $\mu$ which gives measure zero to $X$, a contradiction.

Let us fix an ergodic measure $\mu_0$ and let us assume by contradiction that there exists a full measure set $\widetilde X$
of points $x$ such that
there are two different points $y_1,y_2\in \cW^{c}(x)$
with $\cW^u(y_1)=\cW^u(y_2)$.
Note that one can reduce the set $\widetilde X$ and assume that some number $L^->0$
satisfies, for each such $x,y_1,y_2$, the inequalities
$d^u(y_1,y_2)<L^-$
and $L^->K_0 (\max\{d^c(x,y_1),d^c(x,y_2)\}+c_0)$ where $K_0,c_0>0$ are the numbers which appears in the definition
at the beginning of Section~\ref{s.global.prod.structure} (the dynamics does not shrink the center).
We also set $L^+=L^-+1$.
Lemma~\ref{l.small.product.structure} gives us $r>0$ and a set $Z$.
Item (iii) of Lemma~\ref{l.small.product.structure} and the ergodicity of $\mu_0$ ensure that there exists
$x\in \widetilde X$ which has arbitrarily large backward iterates $f^{-n}(x)$
which belong to plaques $B^c_{L-L^-}(z)$ for some $z\in Z$.

Let $y_1,y_2 \in \cW^{c}(x)$ be the points associated to $x$.
Since $n$ is large, one has $d^u(f^{-n}(y_1),f^{-n}(y_2))<r$.
Since $f$ does not shrink the center and
$K_0(d^c(x,y_i)+c_0)<L^-$, one gets
$d^c(f^{-n}(x),f^{-n}(y_i))<L^-$, so that $f^{-n}(y_1)$ and $f^{-n}(y_2)$ belongs to
a plaque $B^c_{L}(z)$ for some $z\in Z$.
Setting $z_1=z_2=z$,
one has found a point $y\in B^c_L(z)$ such that $W^u_r(y)$ intersects $B^c_L(z)$.
This contradicts Lemma~\ref{l.small.product.structure}, Item (iv).
Hence $\mu_0(X)>0$ and Theorem~\ref{t.prod.struct} is proved.
\qed

\subsection{An additional property}
Having proved Theorem~\ref{t.prod.struct}, one can improve
Lemma~\ref{l.small.product.structure} and obtain additional property on the $Z$:

\begin{lemma}\label{l.small.product.structure2}
Let us consider $Z$ as in Lemma~\ref{l.small.product.structure}.
Then, there exists an invariant full measure set $\Omega\subset M$ such that $Z$ contains any point $z'\in W^{su}_{\delta,\delta}(x_0)$
satisfying $B^c_{2L}(z')\cap \Omega\neq \emptyset$
and $W^u_{2\delta}(z')\cap Z\neq \emptyset$.
\end{lemma}
\begin{proof}
Let $\Omega$ coincide with the full measure set $X$ given by Theorem~\ref{t.prod.struct}
and consider $r,L>0$ and a set $Z$ as in Lemma~\ref{l.small.product.structure}.
Let us consider $z\in Z$ and $z' \in \cW^{su}_{\delta,\delta}(x_0)\cap \cW^u_{2r}(z)$, such that $B^c_{2L}(z')\cap X\neq\emptyset$. We have to prove that $z'$ belongs to $Z$.
By Theorem~\ref{t.prod.struct}, since $\cW^c(z')$ meets $X$,
the unstable leaf $\cW^u(z)$ intersects $\cW^c(z')$ in a unique point, which has to be $z'$.
Let $z''\in B^c_{10L}(z')\cap \cW^{su}_{r,r}(x_0)$. By the first claim in the proof of Lemma~\ref{l.small.product.structure},
$z''\in \cW^u_{\mathfrak{T}}(z)\cap \cW^c(z')$, hence $z'=z''$.
We have proved that $\operatorname{Card}(B^c_{10L}(z')\cap \cW^{su}_{r,r}(x_0))=1$, and $z'$ belongs to $Z$.
\end{proof}
This property will be used twice:
first to check a covering property (Proposition~\ref{p.crossing}),
then to conclude the proof of the Main Theorem (Section 5.5).


\section{Transverse entropy on a product space}\label{s.trans.entropy}
In this section we temporarily abandon the dynamics and
define the notion of transverse entropy.
We then establish a criterion for u-invariance in section~\ref{ss.zerogap}.

\subsection{Transverse entropy of a partition}
Let us consider two standard Borel spaces $X^c,X^u$ and a probability measure $\mu^{cu}$ on the product $X^{cu}:=X^u\times X^c$.
To any point $x=(x^u,x^c)$, one associates the horizontal $X^u(x):=X^c\times \{x^u\}$
and vertical $X^c(x):=\{x^c\}\times X^u$.

Rokhlin's theorem~\cite{rohklin} also associates horizontal and vertical disintegrations of $\mu^{cu}$,
i.e. collections $\{\mu^u_x\}$ and $\{\mu^c_x\}$ of probabilities
on the horizontal $X^u(x)$ and vertical $X^c(x)$ of almost every point $x$.

For any mesurable set $R\subset X^{cu}$ with positive $\mu^{cu}$-measure, we define
$$\mu_R^{cu}:= \frac {\mu^{cu}|_R} {\mu^{cu}(R)},$$
$$\mu^{u}_{R,x}:= \frac {\mu^{u}_x|_{R\cap X^u(x)}} {\mu^{u}_x(R\cap X^u(x))}
\quad \text{ for $\mu^{cu}_R$-almost every $x$}.$$

If $\cP$ is a finite measurable partition of $X^{cu}$ we denote by $\cP(x)$ the atom which contains $x$ and by $\cP|_R$ the partition induced by $\cP$ on $R$. We denote $\cP'\prec \cP$ when $\cP$ is a partition finer than $\cP'$.
We then introduce the entropy of the partition $\cP|_R$ and the entropy of $\cP|_R$ along the horizontals:
$$
H_{\mu^{cu}}(R,\cP)=-\int_R \log \mu_R^{cu}(R\cap \cP(x))d\mu^{cu}_R(x),
$$
$$
H^u_{\mu^{cu}}(R,\cP)=-\int_R \log \mu^{u}_{R,x}(R\cap \cP(x)\cap X^u(x))d\mu_R^{cu}(x).
$$

\begin{definition}
The \emph{transverse entropy} of the partition $\cP|_R$ for the measure $\mu^{cu}_R$ with respect to the horizontals is:
$$
H^{tr}_{\mu^{cu}}(R,\cP):=H_{\mu^{cu}}(R,\cP)-H^u_{\mu^{cu}}(R,\cP).
$$
\end{definition}
When there is no ambiguity we omit the subindex $\mu^{cu}$.

\subsection{Properties of the transverse entropy}
\begin{proposition}\label{p.gap.prop}
The following properties hold:
\begin{enumerate}
\item[(i)] $H^{tr}(R,\cP)\geq 0$.
\item[(ii)] $H^{tr}(R,\cP)= 0$ if and only if for $\mu^{cu}$-almost every $x\in R$,
$$\mu^u_{R,x}(R\cap\cP(x)\cap X^u(x))=\mu^{cu}_{R}(R\cap\cP(x)).$$
\item[(iii)] If $\cP'\prec \cP$ then $H^{tr}(R,\cP')\leq H^{tr}(R,\cP)$. Moreover
$$H^{tr}(R,\cP)=H^{tr}(R,\cP')+\sum_{P'\in \cP'}\mu^{cu}_R(P')H^{tr}(R\cap P',\cP).$$ 
\item[(iv)] If $R'\subset R$ with $X^u(x)\cap R=X^u(x)\cap R'$ for $\mu^{cu}$-almost every $x\in R'$,
$$H^{tr}(R,\cP)\geq \mu^{cu}_{R}(R')H^{tr}(R',\cP).$$
\end{enumerate}
\end{proposition}

\begin{proof}
Let $\alpha$ be the partition into local unstable sets $X^u(x)\cap R$ of $R$, and let $\tmu$ be measure induced on the quotient $R/\alpha$. Then,
$$
H^u(R,\cP)=\int_{R/\alpha} \sum_{P\in \cP}-\mu^u_x(P)\log \mu^{u}_x(P)d\tmu(x).
$$
Jensen inequality implies:
$$
\begin{aligned}
H^u(R,\cP)&\leq -\sum_{P\in \cP}\left(\int_{R/\alpha} \mu^u_x(P)d\tmu(x)\right)\log \left(\int_{R/\alpha}\mu^{u}_x(P)d\tmu(x)\right)\\
&= H(R,\cP)
\end{aligned}
$$
and equality holds if and only if, for every $P\in \cP$, $\mu^u_x(P\cap X^u(x)\cap R)$ is constant for $\mu^{cu}$-almost every $x$, proving (i) and (ii).

For (iii) observe that
$$H(R,\cP)=\int_{R} \left(-\log \frac{\mu^{cu}_{R}(\cP(x))}{\mu^{cu}_{R}(\cP'(x))}-\log\mu^{cu}_{R}(\cP'(x))\right) d\mu^{cu}_{R}(x),$$
so
$$
H(R,\cP)=H(R,\cP')+\sum_{P'\in \cP'}\int_{P'} -\log \frac{\mu^{cu}_{R}(\cP(x))}{\mu^{cu}_{R}(P')}d\mu^{cu}_{R}.
$$
As $\frac{1}{\mu^{cu}_{R}(P')}\mu^{cu}_{R}=\mu^{cu}_{R\cap P'}$ we have
$$
H(R,\cP)=H(R,\cP')+\sum_{P'\in \cP'}\mu^{cu}_{R}(P')H(R\cap P',\cP).
$$
An analogous formula is true for $H^u$, and taking the difference we get
$$
H^{tr}(R,\cP)=H^{tr}(R,\cP')+\sum_{P'\in \cP'}\mu^{cu}_{R}(P')H^{tr}(R\cap P',\cP).
$$

For (iv) observe that $H^u(R,\cP)=\int_R -\log \mu^u_{R,x}(\cP(x))d\mu^{cu}_R(x)$, and 
$$
H(R,\cP)=\int_{R'}-\log \mu^u_{R,x}(\cP(x))d\mu^{cu}_R(x)+\int_{R\setminus R'}-\log \mu^u_{R,x}(\cP(x))d\mu^{cu}_R(x).
$$
If $R'$ is u-saturated inside $R$, we have $\mu^u_{R,x}=\mu^u_{R',x}$. Then,
\begin{equation}\label{eq.psi}
H^u(R,\cP)=\mu^{cu}_{R}(R')H^u(R',\cP)+\mu^{cu}_{R}(R\setminus R')H^u(R\setminus R',\cP).
\end{equation}
Now observe that 
$$
H(R,\cP)=\sum_{P\in \cP}-\mu^{cu}_{R}(P)\log \mu^{cu}_{R}(P).
$$
For any fixed $P\in \cP$ we have 
$$
\mu^{cu}_{R}(P)=\mu^{cu}_{R}(R')\mu^{cu}_{R'}(P\cap R')+\mu^{cu}_{R}(R\setminus R')\mu^{cu}_{R\setminus R'}(P\cap (R\setminus R')).
$$
Then applying Jensen inequality to the function $\phi(x)=-x\log x$ we have 
\begin{equation}
\begin{aligned}
H(R,\cP)&\geq \sum_{P\in \cP}\mu^{cu}_{R}(R')\phi(\mu^{cu}_{R'}(P\cap R'))+\mu^{cu}_{R}(R\setminus R')\phi(\mu^{cu}_{R\setminus R'}(P\cap (R\setminus R')))\\
\label{eq.zeta}&=\mu^{cu}_{R}(R')H(R',\cP)+\mu^{cu}_{R}(R\setminus R')H(R\setminus R',\cP).
\end{aligned}
\end{equation}
Subtracting \eqref{eq.psi} from \eqref{eq.zeta} we get 
$$
\begin{aligned}
H^{tr}(R,\cP)&\geq \mu^{cu}_{R}(R')H^{tr}(R',\cP)+\mu^{cu}_{R}(R\setminus R')H^{tr}(R\setminus R',\cP)\\
&\geq\mu^{cu}_{R}(R')H^{tr}(R',\cP). \qquad\qquad\qquad\qquad\qquad\qquad\qquad\qquad\qquad \hfill {\qedhere} 
\end{aligned} $$
\end{proof}

\begin{corollary}\label{c.iterates.gap}
Consider a sequence of measurable partitions $(\cP_n)$ satisfying $\cP_n\prec \cP_{n+1}$.
Then for every $n\in\natural$,
$$
H^{tr}(R,\cP_n)=H^{tr}(R,\cP_0)+\int_{R} \sum_{j=0}^{n-1} \sum_{P_j\in\cP_j}\chi_{P_j}(x)H^{tr}(R\cap P_j,\cP_{j+1})d\mu^{cu}_{R}(x).
$$
\end{corollary}
\begin{proof}
By Proposition~\ref{p.gap.prop} (iii)
$$
H^{tr}(R,\cP_n)=H^{tr}(R,\cP_{n-1})+\sum_{P_{n-1}\in \cP_{n-1}}\mu^{cu}_{R}(P_{n-1}) H^{tr} (R\cap P_{n-1},\cP_n).
$$
Inductively we have
$$
H^{tr}(R,\cP_n)=H^{tr}(R,\cP_{0})+\sum_{j=0}^{n-1} \sum_{P_j\in \cP_j}\mu^{cu}_{R}(P_j) H^{tr} (R\cap P_j,\cP_{j+1}),
$$
so we can write this as the announced formula.
\end{proof}

\subsection{Criterion for u-invariance}\label{ss.zerogap}
In the present setting, the center disintegration $\{\mu^c_x\}$ is u-invariant if and only if $\mu^{cu}$ is a product $\mu^u\times\mu^c$.

\begin{corollary}\label{c.criterion}
Let us consider a sequence of measurable partitions $(\cP^u_n)$ which generates $X^u$,
and let us define the partitions $\cP_n:=\cP_n^u\times X^c$ of $X$.

Then $\mu^{cu}$ is a product $\mu^u\times\mu^c$ if and only if
$H^{tr}(X^{cu},\cP_n)=0$ for each $n$.
\end{corollary}
\begin{proof}
By Rokhlin disintegration theorem we can write $\mu^{cu}=\int_{X^c} \mu^u_x d\mu^c(x)$, we can identify $X^u\times\{x^c\}$ with $X^u$, now by Proposition~\ref{p.gap.prop}(ii) $\mu^u_x(\cP_n^u(x))=\mu^{cu}(\cP_n(x))$ if and only if $H^{tr}(X^{cu},\cP_n)=0$. As $\cP_n^u$ generates the sigma algebra of $X^u$ then we also have $\mu^u_x(A)=\mu^{cu}(A\times X^c)$ for any measurable set $A\subset X^u$, so this implies $\mu^u_x$ is constant equal to $\mu^u(A):=\mu^{cu}(A\times X^c)$.
\end{proof}

\section{Partitions with small boundary}\label{s.partition.small.bound}
We here construct a special class of local partitions inside an unstable plaque of a partially hyperbolic diffeomorphism.
The construction is standard and goes back to~\cite{LS82}.
We consider a $C^1$ diffeomorphism $f$ which is partially hyperbolic, dynamically coherent, quasi-isometric in the center and an ergodic measure $\mu$.

\subsection{Definition of partitions with small boundary} Let us consider:
\begin{itemize}
\item[--] a point $x_0\in M$,
\item[--] an unstable ball $D^u:=B^u_\delta(x_0)$ of radius $\delta>0$ inside $\cW^u(x_0)$,
\item[--] a finite measurable partition $\cP^u$ of $D^u$ such that Lebesgue a.e. point $x\in D^u$
belongs to the interior of $\cP^u(x)$ (relative to $\cW^u(x_0)$).
\end{itemize}
Let $ \partial \cP^u$ be the boundary of $\cP^u$ inside $\cW^u(x_0)$, and for $L>0$ let
$$\partial^{cs}_L\cP^u=\cup_{x\in \partial \cP^u}B^{cs}_L(x).$$
We also fix $\lambda\in(0,1)$ such that $\lambda>\limsup_{n\to +\infty} \|Df^{-n}|_{E^u}\|^{1/n}$.

\begin{definition}\label{d.boundary}
We say that $\partial^{cs}_L\cP^u$ has \emph{small measure} (or that $\cP^u$ has \emph{small boundary} when $L$ is fixed) if there exists
$\lambda'\in(\lambda,1)$ and $C>0$ such that
$$\mu\big\{x\in M;d^u(\partial^{cs}_L \cP^u,x)<\lambda^n\big\}\leq C\lambda'^n,\quad \forall n\geq 0.$$
When $\cP^u=\{D^u\}$ we say that $D^u$ has small boundary.
\end{definition}

\subsection{Existence of partitions with small boundary}
The partitions are obtained thanks to the next statement.
\begin{proposition}\label{p.cover}
For every $L,\delta_0>0$ there exists $\delta$ arbitrarily close to $\delta_0$
such that $D^u:=B^u_\delta(x_0)$ has small boundary.
Moreover, for any $\varepsilon>0$, there exists a partition $\cP^u$ of $D^u$ into sets with diameter smaller than $\varepsilon$ which has small boundary. 
\end{proposition} 
The proof requires two lemmas. The first is proved as~\cite[Lemma A.1]{Jiagang-entropy-along-expanding}.
\begin{lemma}\label{l.exponential}
Let $\nu$ be a finite measure supported on an interval $[0,R]$. Then for any $0<\lambda''<\lambda'<1$, there is a full Lebesgue measure subset $I\subset (0,R)$ with the following property. For every $t\in I$, there is $C_t>0$ such that:
$$
\nu([t-{\lambda''}^n,t+{\lambda''}^n])\leq C_t\lambda'^n, \qquad \forall n\geq 0
$$ 
\end{lemma}

The second one asserts that cs-holonomies are H\"older continuous.
The result is classical in the case of s-holonomies. In the case of cs-holonomies,
we use the fact that $f$ is quasi-isometric in the center.

\begin{lemma}\label{l.holder-holonomy}
For any $L>0$, there exist $r,\alpha,C>0$ such that
if $D,D'$ are two subsets of unstable leaves with diameter smaller than $r$
and if $\fh\colon D\to D'$ is a cs-holonomy satisfying $d^{cs}(y,\fh(y))<L$
for each $y\in D$, then $\fh$ is $(\alpha,C)$-H\"older continuous, ie. for any $y_1,y_2\in D$ it satisfies:
$$d^u(\fh(y_1),\fh(y_2))\leq Cd^u(y_1,y_2)^\alpha.$$
\end{lemma}
\begin{proof}
Recall that we have fixed a small number $\varepsilon_0$ which measures the size of
the local manifolds.

Let us consider two close points $y_1,y_2$ in a same leaf of $\cW^u$
and their image by a cs-holonomy $\fh$ satisfying
$d^{cs}(\fh(y_i),y_i)< L$. There exists two arcs $\gamma_i\colon[0,1]\to M$
connecting $y_i$ to $\fh(y_i)$, with length smaller than $L$, contained in cs-leaves.
These arcs are arbitrarily close if $y_1,y_2$ are close enough, i.e. if $r$ is chosen small.
One can thus require that $\gamma_1(t),\gamma_2(t)$ are contained in a same local unstable leaf for each $t\in [0,1]$.

By forward iteration, the arcs separate in the unstable direction:
there exists a first time $N$ such that $d^u(f^N(\gamma_1(t_0)),f^N(\gamma_2(t_0)))\geq \varepsilon_0$ for some $t_0\in [0,1]$.
Note that there exists $C_1>0$ uniform such that
\begin{equation}\label{e.lower}
d^u(\fh(y_1),\fh(y_2)) e^{C_1 N}\leq \varepsilon_0.
\end{equation}
Since $f$ is quasi-isometric in the center, the diameter of the arcs $f^N\circ \gamma_i$
is bounded by $K_0L+c_0$, independently from $N$.
In particular there exists a constant $\eta>0$ (which does not depend on $y_1,y_2$, nor $N$)
such that $d^u(f^N(\gamma_1(t)),f^N(\gamma_2(t)))\geq \eta$ for all $t\in [0,1]$.
One deduces that there exists $C_2>0$ uniform such that
\begin{equation}\label{e.upper}
d^u(y_1,y_2)e^{C_2 N}\geq \eta.
\end{equation}
Combining estimates~\eqref{e.upper} and~\eqref{e.lower}, one gets the announced inequality with
$\alpha=C_1/C_2$ and $C=\varepsilon_0.\eta^{-\alpha}$.
\end{proof}

\begin{proof}[Proof of Proposition~\ref{p.cover}]
Let $\alpha>0$ be given by Lemma~\ref{l.holder-holonomy}.
We fix arbitrarily $\lambda''<\lambda'$ in $(\lambda^\alpha,1)$.

Let us fix $\eta>0$ small and introduce the disc $D^u_0=B^u_{\delta_0+\eta}(x_0)$ which will contain our constructions.
By Proposition~\ref{p.cs-u.product}, if $\gamma>0$ is small,
then the set $\cN^{cs\times u}:=\cup_{x\in B^{cs}_\gamma(x_0)}\fh^{cs}_{x_0,x}(D^u_0)$ has a cs$\times$u-product structure:
it is the image of $D^u\times B^{cs}_\gamma(x_0)$ by the homeomorphism
$\Psi:(x,y)\mapsto \fh^{cs}_{x_0,y}(x)$.
\medskip

\paragraph{\emph{First case: $L>0$ is small so that $\Lambda:=\cup_{x\in D^u}B^{cs}_L(x)$ is contained in $\cN^{cs\times u}$.}}
Each partition $\cP^u$ of $D^u_0$ can be extended to a partition of $\cN^{cs\times u}$
as $\cP=\Psi(\cP^u\times B^{cs}_{\gamma}(x_0))$.
We have $\partial\cP\supset \partial^{cs}_L\cP^u$, so it is enough to prove that $\partial\cP$ has small measure.

Given any subset $A\in D^u_0$, we call \emph{strip of base} $A$ its cs-saturation inside $\cN^{cs\times u}$ of $A$,
i.e. the set $\Psi(A\times B^{cs}_{\gamma}(x_0))$. 

Let $\{x_1,\dots, x_m\}$ be a $\frac{\varepsilon}{4}$-dense subset of $D^u_0$ and $S(x_i,b)$ be
the strip of base $B^u_b(x_i)$ for $b>0$.
For each $x_i$ we define a finite measure $\nu_i$ on $(0,\frac \varepsilon 2)$ by:
$$\nu_i([a,b))=\mu(S(x_i,b)\setminus S(x_i,a)).$$
By Lemma~\ref{l.exponential}, for Lebesgue a.e. $\varepsilon_i\in (\frac{\varepsilon}{4},\frac \varepsilon 2)$
there is $C_i>0$ satisfying
$$\mu(S(x_i,\varepsilon_i+{\lambda''}^n)\setminus S(x_i,\varepsilon_i-{\lambda''}^n))<C_i\lambda'^n, \quad \forall n\geq 0.$$
Let $\partial^{cs}_L B^u_{\varepsilon_i}(x_i):=\cup_{y\in \partial B^u_{\varepsilon_i}(x_i)}B^{cs}_L(y)$.
By $\alpha$-H\"older continuity of the cs-holonomy (see Lemma~\ref{l.holder-holonomy}) and since $\lambda^\alpha<\lambda''$,
one gets for $n$ large enough:
$$
\{x\in M\colon\; d(\partial^{cs}_L B^u_{\varepsilon_i}(x_i),x)<\lambda^n\}\subset S(x_i,\varepsilon_i-{\lambda''}^n)\setminus S(x_i,\varepsilon_i+{\lambda''}^n),
$$
So $\partial^{cs}_L B^u_{\varepsilon_i}(x_i)$ has small measure. As there are finitely many $i$ we can take the same $\varepsilon':=\varepsilon_i$ for every $i=1,\dots,m$,  smaller than $\frac \varepsilon 2$.

Analogously one can define a measure $\nu_0$ on $(\delta_0-\eta,\delta_0+\eta)$ by setting $\nu_0([a,b))=\mu(S(x_0,b)\setminus S(x_0,a))$. We can thus get $\delta$ with $|\delta-\delta_0|<\eta$
such that $\partial_L B^u_{\delta}(x_0)$ has small measure.

Let us take $D^u:= B^u_\delta(x_0)$ and let $\cP^u=\cP^u(x_1,\dots,x_m,\varepsilon')$ be the partition of $D^u$ generated by intersecting the sets $B^u_{\varepsilon'}(x_i)$, $i=1,\dots,m$.
The proof of the proposition is done in this case.
\bigskip

\paragraph{\emph{General case: $L$ is arbitrary.}}
We cover $M$ by finitely many sets with a cs$\times$u-product structure.
We choose points $z_j$ and number $\delta_j>0$
with $0\leq j\leq \ell$ and introduce sets of the form
$\cN^{cs\times u}_j=\Psi_j(D^u_j\times B^{cs}_\gamma(z_j))$ which cover $M$,
where $D_j=B^u_{\delta_j}(z_j)$.
We also take $D^u_0=B^u_{\delta_0+\eta}(x_0)$ as before.

For each $x\in D^u$ and $0<j\leq \ell$ we consider the
geodesic arcs $\psi:[0,1]\to B^{cs}_{2L}(x)$ with length less than $2L$,
that connect $x$ to some point $y\in D^u_j$. There are finitely many such arcs.
Each of them defines some cs-holonomy
$\fh$ from a neighborhood $\Sigma\subset D^u_0$ of $x$
to a neighborhood of $y$ in $D^u_j$.
One has thus associated to the point $x$ a finite number of holonomy maps
$\fh^k_x$, and one can assume that they are defined on the same domain $\Sigma_x\subset D^u_0$.
Each holonomy $\fh^k_x$ takes its values in a disc $D^u_{j(k)}$ and by construction:
\begin{equation}\label{e.decompose}
B_L^{cs}(y)\subset \bigcup_k \Psi_{j(k)}(\{\fh^k_x(y)\}\times B^{cs}_\gamma(z_{j(k)})), \quad \text{ for each } y\in \Sigma_x.
\end{equation}

We then consider some $\tfrac{\varepsilon}4$-dense subset $\{x_1,\dots,x_m\}$ of $D^u_0$
whose associated domains $\Sigma_1,\dots,\Sigma_m$ cover $D^u_0$. To each point $x_i$
are also associated finitely many cs-holonomies $\fh_i^k\colon \Sigma_i\to D^u_{j(k)}$
and denote $y^k_i:=\fh^k_i(x_i)$.
To each of them, one considers for $b>0$ the strip in $\cN_{j(k)}^{cs\times u}$ defined by:
$$S(y_i^k,b)=\Psi_{j(k)}(\fh^k_i(B_b^u(x_i))\times B^{cs}_\gamma(z_{j(k)})),$$
and the measure $\nu_i^k$ on $(0,\tfrac \varepsilon 2)$ by
$$\nu_i^k([a,b))=\mu(S(x^k_i,b)\setminus S(x^k_i,a)).$$

\begin{figure}[h]
\centering
\includegraphics[scale= 0.5]{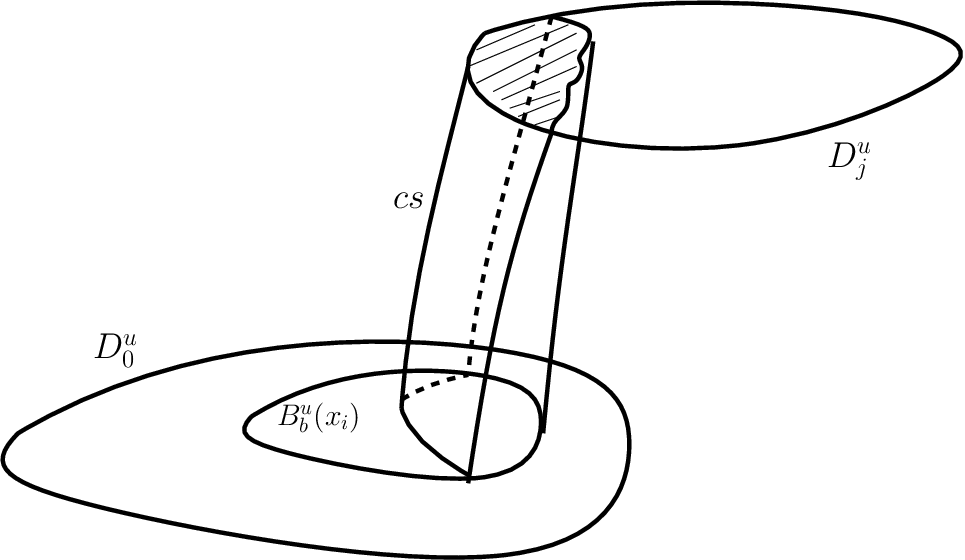}
\put(-150,130){\tiny $\fh^k_i(B^u_b(x_i))$}
\caption{Cs-holonomy $\fh^k_i$ between $D^u_0$ and  $D^u_j$.}
\label{fig.Bjim}
\end{figure}

As before shows that there exists a total Lebesgue measure subset of $\varepsilon'\in (\tfrac \varepsilon 4,\tfrac \varepsilon 2)$
such that each set
$\Psi_{j(k)}(\fh^k_i(\partial B^u_{\varepsilon'}(x_i))\times B^{cs}_\gamma(z_{j(k)}))$ has small measure.
Analogously there exists $\delta$ with $|\delta-\delta_0|<\eta$ such that each set
$\Psi_{j(k)}(\fh^k_i(\partial B^u_\delta(x_0))\times B^{cs}_\gamma(z_{j(k)}))$ has small measure.

As before we take $D^u:= B^u_\delta(x_0)$ and let $\cP^u=\cP^u(x_1,\dots,x_m,\varepsilon')$ be the partition of $D^u$ generated by intersecting the sets $B^u_{\varepsilon'}(x_i)$, $i=1,\dots,m$. Its atoms have diameter smaller than $\varepsilon$.
Moreover from~\eqref{e.decompose}, we get
$$
\partial^{cs}_L \cP^u(x_1,\dots,x_k,\varepsilon')\subset \bigcup_{i,k} \Psi_{j(k)}(\fh^k_i(\partial B^u_{\varepsilon'}(x_i))\times B^{cs}_\gamma(z_{j(k)}))
$$
$$
\partial^{cs}_L D^u\subset \bigcup_{i,k} \Psi_{j(k)}(\fh^k_i(\partial D^u)\times B^{cs}_\gamma(z_{j(k)}))
$$
hence $D^u$ and $\partial P^u$ have small boundary,
concluding the proof of Proposition~\ref{p.cover}.
\end{proof}

\subsection{Size of local unstable manifolds}
The small boundary property implies that for a large set of point $x$, many backward iterates of a
local unstable manifold $B^u_\rho(x)$ do not intersect $\partial^{cs}_L\cP$.

\begin{proposition}\label{p.rho}
Let $L>0$ and let $\cP^u$ be a partition of $D^u$ with small boundary. Then, for each $\beta>0$ there exists $\rho=\rho(\beta)>0$ such that the set
$$
M^n_\beta:=\bigg\{x\in M,\; \bigcup_{0\leq j\leq n}f^j(\partial^{cs}_L\mathcal{P}^u) \cap B^u_\rho(f^n(x))=\emptyset \bigg\}
$$
satisfies $\mu(M\setminus M^n_\beta)<\beta$ for every $n\geq 0$.
\end{proposition}
\begin{proof}
The definition of small boundary fixes $C>0$ and $\lambda'\in (\lambda,1)$.
Given $\beta>0$, 
we choose $n_0$ such that $C\sum_{j=n_0}^\infty \lambda'^j<\frac{\beta}{3}$
and $\rho>0$ small such that
$$\mu \{y| f^{-j} B^u_\rho(y) \cap \partial^{cs}_L\mathcal{P}^u\neq \emptyset\}<\tfrac{\beta}{2n_0}, \qquad \forall 0\leq j\leq n_0$$
$$\text{and}\qquad f^{-j}B^u_\rho(y)\subset B^u_{{\lambda}^j}(f^{-j}(y))
\qquad \forall 0\leq j, \forall y\in M.$$
For any $n\geq 0$, the complement of $M^n_\beta$ decomposes in two sets: the first
$$\bigcup_{j=0}^{n_0} \{x| f^{-j} B^u_\rho(f^n(x)) \cap \partial^{cs}_L\mathcal{P}^u\neq \emptyset\}$$
has measure smaller than $\beta/3$ by our choice of $\rho$; the second
$$\bigcup_{j=n_0}^{n} \{x| f^{-j} B^u_\rho(f^n(x)) \cap \partial^{cs}_L\mathcal{P}^u\neq \emptyset\}
\subset \bigcup_{j=n_0}^{n} \{x| B^u_{{\lambda}^j}(f^{n-j}(x)) \cap \partial^{cs}_L\mathcal{P}^u\neq \emptyset\}$$
is contained in
$\bigcup_{j=n_0}^{n} \{x| d^u(\partial^{cs}_L\mathcal{P}^u,f^{n-j}(x))<{\lambda}^j\}$
which has $\mu$-measure less than $C\sum_{j=n_0}^\infty \lambda'^j<\frac{\beta}{3}$ by
definition~\ref{d.boundary}.
\end{proof}

\begin{remark}\label{r.iterate}
The previous proposition remains true if we replace $f$ by some iterate $f^m$, since
the set $M^n_\beta(f^m)$ for $f^m$ contains the set $M^{nm}_\beta$ for $f$.\end{remark}


\section{transverse entropy of a diffeomorphism}\label{s.gaps.entropy}
We now prove the Main Theorem. Intermediate steps are stated in section~\ref{ss.integrated-entropy}.
Note that the proof will use both inequalities in the definition of quasi-isometric center. We will iterate
center-unstable plaques forwardly, and the lower bound is used in order to guarantee that the plaque
does not collapse under iterations. But the upper bound is also used, so that the plaque keep bounded center
geometry and can be compared to a reference center-unstable plaque, even after a large iterate.

\subsection{Preliminary constructions}
\subsubsection{Initial setting} 
In the whole section, we consider:
\begin{itemize}
\item[--] a partially hyperbolic diffeomorphism $f$ quasi-isometric in the center:
there is $K>1$ such that for $l>0$ large enough, $x\in M$, and $n\geq 0$,
\begin{equation}\label{e.quasi-restated}
B^c_{l/K}(f^n(x))\subset f^n(B^c_l(x))\subset B^c_{Kl}(f^n(x)),
\end{equation}
\item[--] an ergodic probability $\mu$.
\end{itemize}
We consider three large center scales $\ell_1,\ell_2,L_0$ satisfying:
$$4K\ell_1<\ell_2 < \tfrac1{4K}L_0.$$

Proposition~\ref{l.small.product.structure} can be restated as follows:
\begin{proposition}\label{p.def-N}
There exist $x_0\in \supp(\mu)$, $L>0$ close to $L_0$,
$\delta>0$ small,
and a measurable set $Z\subset \cW^{su}_{\delta,\delta}(x_0)$
with the following properties:
\begin{itemize}
\item[(a)] The set $Z$ has a cs$\times$u-product structure. 

\item[(b)] Two sets $B^c_{2L}(z)$, $B^c_{2L}(z')$ are disjoint for any $z\neq z'$ in $Z$.

\item[(c)] For any $z\in Z$, $z'\in \cW^{u}_Z(z)$, the u-holonomy defines a homeomorphism
$\fh^u_{z,z'}$ from $B^c_{L}(z)$ to a subset of $B^c_{L+1}(z')$ containing $B^c_{L-1}(z')$.

\item[(d)] For any $z\in Z$, $z'\in \cW^{s}_Z(z)$, the s-holonomy defines a homeomorphism
$\fh^s_{z,z'}$ from $B^c_{L}(z)$ to a subset $B^c_{L+1}(z')$ containing $B^c_{L-1}(z')$.

\item[(e)] $\mu(\cup_{z\in Z\cap U} B^c_{\ell_1/2})>0$, for any neighborhood $U$ of $x_0$.

\item[(f)] There exists an invariant full measure set $\Omega$ such that
for any $z\in Z$,
$$\{z'\in W^u_{2\delta}(z)\cap \cW^{su}_{\delta,\delta}(x_0); \;\;B^c_{2L}(z')\cap \Omega\neq \emptyset\}\subset Z.$$
\end{itemize}
\end{proposition}
\begin{proof}
The Proposition~\ref{l.small.product.structure}
provides us with a point $x_0$, some $L$ close to $L_0$,
and a small number $r>0$.
Fixing $\delta>0$ small, one also gets
a measurable set $Z\subset \cW^{su}_{\delta,\delta}(x_0)$
which satisfies properties (a-e).
The last property (f) is given by Lemma~\ref{l.small.product.structure2}.
\end{proof}

We fix $r>0$ small. Note that we can reduce $\delta>0$ keeping the previous properties. One can thus furthermore require
(from Proposition~\ref{p.cover}):
\begin{itemize}\it
\item[(g)] For any $z\in Z$, $z'\in \cW^s_Z(z)$, we have $B^c_L(z')\subset B^{cs}_{L+1}(z)$.
\item[(h)] For any $z\in Z$, $z'\in  \cW^s_Z(z)$ (resp. $\cW^u_Z(z)$) and $y\in B^c_L(z)$, we have
$d(\fh^s_{z,z'}(y),y)<r$ (resp. $d(\fh^u_{z,z'}(y),y)<r$).
\item[(i)] $D^u:=B^u_\delta(x_0)$ has small boundary
$\partial^{cs}_{2L}D^u$ (as in definition~\ref{d.boundary}).
\end{itemize}

\subsubsection{The regions $R_1\subset R_2\subset \cN$}
We first define $\cN\subset\cup_{z\in Z} B^c_L(z)$, using Proposition~\ref{p.def-N}(a-c):
a point $y\in B^c_L(z)$ belongs to $\cN$ if for any $z'\in \cW^{u}_Z(z)$
we have $\fh^u_{z,z'}(y)\in B^c_L(z')$. We then associate the sets:
$$\cW^c_\cN(y):=B^c_{L}(z)\cap \cN, \;\; \cW^u_\cN(y):=\{\fh^u_{z,z'}(y),z'\in \cW^u_Z(z)\},$$
$$\cW^{cu}_\cN(y):=\bigcup_{z'\in \cW^{u}_Z(z)} B^c_L(z'),\;\;
\cW^{cs}_\cN(y)=\bigcup_{z'\in \cW^{s}_Z(z)} B^c_L(z').$$
This induces measurable partitions $\cW^c_\cN,\cW^u_\cN,\cW^{cu}_\cN,\cW^{cs}_\cN$ of $\cN$.
Note also that $ \cW^{cu}_\cN(y)= \cup_{y'\in \cW^{c}_\cN(y)} \cW^u_\cN(y')$ has a c$\times$u-product structure
and that $B^c_{L-1}(z)\subset \cW^c_\cN(z)\subset B^c_{L}(z)$ for each $z\in Z$.
Since $\cW^s$ and $\cW^u$ are not jointly integrable,
$\cW^{cs}_{\cN}(y)$ do not have in general a c$\times$s-pro\-duct structure.
\medskip

One defines $R_1$ by cutting $\cN$ in the center at scale $\ell_1$:
it is the set of points $y\in \cN$ such that
$ \cW^u_\cN(y)$ is contained in $\bigcup_{z\in Z} B^c_{\ell_1}(z)$.
By Proposition~\ref{p.def-N}(e),
\begin{equation}\label{e.R1-positive}
\mu(R_1)>0.
\end{equation}
For any $y\in R_1$ and $*\in\{c,u,cu,cs\}$, we define
the set $\cW^*_{R_1}(y):=\cW^*_\cN\cap R_1$.
The families $\cW^*_{R_1}$ are measurable partitions of $R_1$.
\medskip

One defines $R_2$ analogously, by cutting $\cN$ at scale $\ell_2$
and measurable partitions $\cW^*_{R_2}$.
The choices of $\ell_1,\ell_2,L$ imply the following.

\begin{lemma}\label{l.return0}
If $\delta$ is small enough, then for any $n\in \mathbb{Z}$ and any $x\in \cN$:
\begin{itemize}
\item[--] if $x\in R_1\cap f^{-n}(R_1)$, we have $f^n(\cW^c_{R_2}(x))\supset \cW^c_{R_1}(f^n(x))$,
\item[--] if $x\in R_2\cap f^{-n}(R_2)$, we have $f^n(\cW^c_{R_2}(x))\subset \cW^c_{\cN}(f^n(x))$,
\item[--] if $x\in R_2$, we have $f^n(\cW^c_{R_2}(x))\subset B^c_{L-1}(f^n(x))$.
\end{itemize}
\end{lemma}
\begin{proof}
Taking $\delta>0$ small,
for any $z\in Z$, we have $\cW^c_{R_2}(z)\supset B^c_{0.9\ell_2}(z)$.

Let us take $x\in R_1\cap f^{-n}(R_1)$.
For any $y\in \cW^c_{R_1}(f^n(x))$, we have
$d^c(y, f^n(x))<2\ell_1$. The quasi-isometric property, the choice of $K$ and of the scales $\ell_1,\ell_2$ give
$d^c(f^{-n}(y),x)<2\ell_1K<0.9\ell_2 -\ell_1$. Since $x\in B^c_{\ell_1}(z)$ for some $z\in Z$,
one deduces $f^{-n}(y)\in B^c_{0.9\ell_2}(z)\subset \cW^c_{R_2}(x)$.
This gives the first property. The second and third ones are proved analogously.
\end{proof}

\subsubsection{Separation of center plaques}
Given $z,z'\in Z$ and $n\geq 0$ we say that $z$ does not separate from $z'$ until time $n$ if
$f^j\fh^u_{z,z''}(y)\in \cW^u_{loc}(f^j(y))$
for any $y\in \cW^c_{R_2}(z)$ and $0\leq j\leq n$
where $z''=[z,z']$.

\begin{lemma}\label{l.scales}
If $r$ is small enough, then for each $n\geq 0$, the property
$$\text{``$z$ does not separate from $z'$ until time $n$"}$$
is an equivalence relation on
$\{z\in Z, \; f^n\cW^c_{R_2}(z)\subset \cN\}$.
(By symmetry of the relation, one can thus say that ``$z,z'$ do not separate until time $n$".)
\end{lemma}
\begin{proof}
We first claim that, by taking $r$ small enough, the separation property holds
for any point $y\in \cW^c_{\cN}(z)$ (rather than $\cW^c_{R_2}(z)$) at return times to $\cN$:
\begin{claim}
If $z$ does not separate from $z'$ until time $n$ and
$f^n(\cW^c_{R_2}(z)\cup \cW^c_{R_2}(z'))\subset \cN$, then
$f^j\fh^u_{z,z''}(y)\in \cW^u_{loc}(f^j(y))$
for any $y\in \cW^c_{\cN}(z)$ and $0\leq j\leq n$.
\end{claim}
\begin{proof}
By (h), if $\bar y\in  \cW^c_{R_2}(z)$, then $d(f^j(\bar y),f^j(\fh^u_{z,z''}(\bar y)))<r$ for $j=n$,
hence for any $0\leq j\leq n$ since $E^u$ is contracted by backward iterates.
Assuming $r$ small enough, and using that $f^j(\cW^c_{\cN}(z))$
has diameter bounded by $2KL$, one deduces that
$f^j(\cW^c_{\cN}(z))$ and $f^j(\cW^c_{\cN}(z'))$ are close for $0\leq j\leq n$.
Hence $f^j\fh^u_{z,z''}(y)$ is close to $f^j(y)$ for $y\in \cW^c_{\cN}(z)$.
\end{proof}

The relation is obviously reflexive.
We prove that it is symmetric.
Let us take $z,z'\in Z$ such that $f^n(\cW^c_{R_2}(z)\cup \cW^c_{R_2}(z'))\subset \cN$, and $z$ does not separate from $z'$ until time $n$. We define $z'':=[z,z']$ and $z'''=[z',z]$.
Given $y'\in \cW^c_{R_2}(z')$ let $y'''=\fh^u_{z',z'''}(y')$. We have to prove that $f^j(y''')\in \cW^u_{loc}(f^j(y'))$ for every $0\leq j\leq n$.

Let $y'':=\fh^s_{z',z''}(y')$ and $y=\fh^u_{z'',z}(y'')$.
By (h), we have $d^u(y,y'')<r$ and $d^s(y',y'')<r$. Moreover as $f^n\cW^c_{R_2}(z)\cup f^n\cW^c_{R_2}(z')\subset \cN$ and since $z$ does not separate from $z'$,
the previous claim implies $f^j(y'')\in \cW^u_{loc}(f^j(y))$ for every $0\leq j\leq n$. We also have $d^u(f^n(y),f^n(y''))<r$.
Since $E^u$ (resp. $E^s$)  is contracted by backward (resp. forward) iterates, we can conclude
for any $0\leq j\leq n$:
\begin{equation}\label{eq.sep1}
d^u(f^j(y),f^j(y''))<r\quand d^s(f^j(y'),f^j(y''))<r.
\end{equation}
Fix $0\leq j\leq n$. We set $y_1:=f^j(y)$, $y_2:=f^j(y'')$, $y_3:= f^j(y')$, $y_4:=f^j(y''')$.

\begin{claim}
If $r>0$ is small, and if the points $y_1,y_2,y_3,y_4$ satisfy $d^u(y_1,y_2)<r$, $d^s(y_2,y_3)<r$, $y_4\in \cW^{cs}_{loc}(y_0)\cap \cW^u_{loc}(y_3)$ then $d^u(y_3,y_4)<2r$.
\end{claim}
\begin{proof}
For $r$ small, the points $y_i$ are close and belongs to a chart where the bundles $E^s,E^c,E^u$ are close
to constant bundles which correspond to coordinate axes. The claim can be concluded by estimating successively
the coordinates of the points $y_i$.
\end{proof}
By the claim and \eqref{eq.sep1}, $d^u(f^j(y'),f^j(y'''))<2r$, hence $f^j(y''')\in \cW^u_{loc}(f^j(y'))$. As this holds for all $0\leq j \leq n$, $z'$ does not separate from $z$ in time $n$.
\medskip

For the transitivity, we fix $z_1,z_2,z_3$ with $f^n(\cW^c_{R_2}(z_1)\cup \cW_{R_2}(z_2)\cup  \cW^c_{R_2}(z_3))$
contained in $\cN$, such that $z_1$ does not separate from $z_2$ until time $n$ and $z_2$ does not separate from $z_3$ until time $n$.
Take $z_{i,j}=[z_i,z_j]$, $i\neq j$, $1\leq i,j \leq 3$. Let $x_1\in \cW^c_{R_2}(z_1)$ and $x_{1,i}=\fh^u_{z_1,z_{1,i}}(x_1)$, for $i=2,3$ 

For every $0\leq j \leq n$ we have 
\begin{equation}\label{eq.sep2}
d^u(f^j(x_1),f^j(x_{1,3}))\leq d^u(f^j(x_1),f^j(x_{1,2}))+d^u(f^j(x_{1,2}),f^j(x_{1,3})).
\end{equation}
Because $z_1$ does not separate from $z_2$ and the centers come back to $\cN$, analogously to $\eqref{eq.sep1}$ we have $d^u(f^j(x_1),f^j(x_{1,2}))<r$.

Let $y_4=f^j x_{1,2}$, $y_3=f^j x_{1,3}$, $ y_2=f^j\fh^s_{z_{1,3},z_{2,3}}(x_{1,3})$, $y_1=f^j\fh^u_{z_{2,3},z_{2}}(y_3)$ and $y_0=f^j\fh^s_{z_2,z_{1,2}}(y_1)$. As $z_2$ does not separate from $z_3$ until time $n$ and $f^n\cW^c_{R_2}(z_i)\subset \cN$ $i=2,3$, the first claim implies again $d^u(y_1,y_2)<r$. The second claim and \eqref{eq.sep2} give
$
d^u(f^j(x_1),f^j(x_{1,3}))\leq 3r,
$
so that $f^j(x_{1,3})\in \cW^u_{loc}(f^j(x_1))$. This concludes the transitivity.
\end{proof}

\subsubsection{Pre-partitions $\cQ_n$}
We define for each $n\geq 0$ a partition $\cQ_n$ of $R_2$ saturated by $\cW^c_{R_2}$ sets in the following way.
Two sets $\cW^c_{R_2}(z), \cW^c_{R_2}(z')$ are contained in the same atom if:
\begin{itemize}
\item[--] either both $f^n(\cW^c_{R_2}(z))$ and $f^n(\cW^c_{R_2}(z'))$ are not contained in $\cN$,
\item[--] or $f^n(\cW^c_{R_2}(z)\cup \cW^c_{R_2}(z'))\subset \cN$ and $z,z'$ do not separate until time~$n$.
\end{itemize}

\subsection{A criterion for zero transverse entropy}
\subsubsection{Integrated transverse entropy}\label{ss.integrated-entropy}
Consider a partition $\cP^u=\{P^u\}$ of $D^u$ with small boundary
$\partial^{cs}_{2L}\cP^u$
as in Proposition~\ref{p.cover}.
It induces a partition $\cP_{R_1}$ of $R_1$ (also denoted by $\cP$) into sets $P$ of the form:
$$P=\bigcup_{y\in P^u} \cW^{cs}_{R_1}(y).$$
Each set $\cW^{cu}_{R_1}(z)$ has a product structure $\cW^c_{R_1}(z)\times \cW^u_{R_1}(z)$.
One introduces its transverse entropy for the measure $\mu^{cu}_{R_1,z}$ on $\cW^{cu}_{R_1}(z)$, as in section~\ref{s.trans.entropy}.
$$\Gap(R_1,\cP)(z):=\Gap(\cW^{cu}_{R_1}(z),\cP\mid_{\cW^{cu}_{R_1}(z)}).$$
One then defines the \emph{integrated transverse entropy} by
$$\overline{\Gap}(R_1,\cP)=\int_{R_1}\Gap(R_1,\cP)(z)d\mu(z).$$

We are going to prove that the partition $\cP$ has zero transverse entropy.
\begin{proposition}\label{p.entropy.gap0}
If $h(f,\mu)=h^u(f,\mu)$, then every partition $\cP^u$ of $D^u$ with small boundary induces a partition
$\cP$ of $R_1$ satisfying $\overline{\Gap}(R_1,\cP)=0$.
\end{proposition}

\subsubsection{The partitions $\cP_n^f$}

\begin{figure}[h]
\centering
\includegraphics[scale= 0.5]{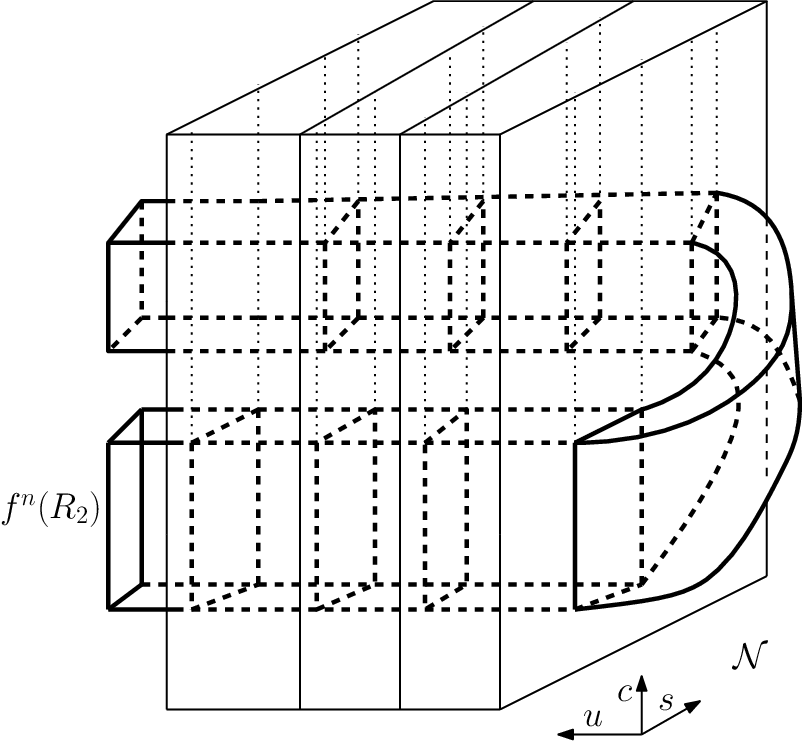}
\caption{construction of $\cP_n^f$.}\label{Qn.partition}
\end{figure}

On the set $\cN$ we can define a partition from $\cP^u$ using the fibered structure
in the same way as we defined $\cP$ on $R_1$:
$$\cP_{\cN}=\{\cup_{y\in P^u} \cW^{cs}_{\cN}(y);P^u\in \cP^u\}.$$

Now we define a partitions $\widetilde{\cQ}_n$ of $R_2$ which refines ${\cQ}_n$.
For each $x\in R_2$:
\begin{itemize}
\item[--] either $f^n(\cQ_n(x))\not \subset \cN$ and $\widetilde{\cQ}_n(x)=\cQ_n(x)$,
\item[--] or $f^n(\cQ_n(x))\subset \cN$ and $\widetilde{\cQ}_n(x):=f^{-n}(\cP_{\cN}(f^n(x)))\cap \cQ_n(x)$.
\end{itemize}
We then define the partition $\cP^f_n:=\vee_{k=0}^n \widetilde \cQ_k$ of $R_2$.
\medskip

\subsubsection{Intermediate steps}
Proposition~\ref{p.entropy.gap0} is a consequence of the next ones.

\begin{proposition}\label{p.gap.persistence}
If $\overline{\Gap}(R_1,\cP)>0$,  then 
$\displaystyle  \liminf \tfrac{1}{n}\overline{\Gap}(R_2,\cP^f_n)>0.$
\end{proposition}

\begin{proposition}\label{p.entropy.gap}
$ \displaystyle \limsup \tfrac{1}{n}\overline{\Gap}({R_2},\cP^f_n) \leq \mu(R_2)(h(f,\mu)-h^u(f,\mu))$.
\end{proposition}
They are proved in the two following sections.

\subsection{Persistence of the transverse entropy}
In this section we assume $\overline{\Gap}(R_1,\cP)>0$ and prove Proposition~\ref{p.gap.persistence}.

\subsubsection{Choice of parameters}\label{ss.choice}  We first select some numbers.
\begin{itemize}
\item[--] We fix $\eta>0$ and a measurable subset $R_1'\subset R_1$ which is a union of sets $\cW^{cu}_{R_1}(z)$
and which satisfies
$$\mu(R'_1)>0 \text{ and } \Gap(R_1,\cP)(z)>\eta \text{ for every } z\in R_1'.$$

\item[--] We fix $0<3\beta<\mu(R_1')^2$ and then $\rho>0$ as given by Proposition~\ref{p.rho}.

\item[--]
There is $0<\rho'<\rho$ as follows.
For any $x$ and $y\in B^{cs}_{L+1}(x)$,
there is a cs-holonomy map
$\fh^{cs}:B^{u}_{\rho'}(x)\to \cW^u(y)$ (a priori not unique) satisfying $\fh^{cs}(x)=y$ and $\fh^{cs}(\zeta)\in B^{cs}_{2L}(\zeta)$ for each $\zeta\in B^{u}_{\rho'}(x)$.

\item[--]
We fix $m\geq 1$ and $g:=f^m$ such that  for every $z\in M$ $$g(B^u_{\rho'}(z))\supset B^u_{r}(g(z)).$$

\item[--] By remark~\ref{r.iterate}, Proposition~\ref{p.rho} is satisfied for $g=f^m$ and $\rho$: the sets

$$M^n_\beta:=\bigg\{x\in M,\bigcup_{j=0}^n g^j(\partial_{2L}\cP^u)\cap B^u_\rho (g^n(x))=\emptyset\bigg\}$$
satisfy $\mu(M\setminus M^n_\beta)<\beta$ for every $n\geq 0$.
\end{itemize}
\medskip

\subsubsection{The partition $\cP_n^g$.}
For $n\geq 0$, we define the partition $\cP^g_{n}$ of $R_2$ by $\cP^g_{n}:=\vee_{j=0}^{n}\widetilde{\cQ}_{n m}$.
Note that the partition $\cP^f_{nm}$ defined for $f$ is thinner than $\cP_{n}^g$ for $g$. 
Hence for proving Proposition~\ref{p.gap.persistence}, it is enough to show that
\begin{equation}\label{e.reduction}
\liminf \tfrac{1}{n}\overline{\Gap}(R_2,\cP^g_{n})>0.
\end{equation}

In the following we only work with the map $g$ and the partitions $\cP_{n}^g$ will be denoted by $\cP_{n}$
in order to keep the notations simpler.

\subsubsection{Properties}
We first see that the iterate $g^{n}(\cP_{n})$ somehow refines $\cP_\cN$.
\begin{proposition}\label{p.separtion}
If $x\in R_2\cap g^{-n}(R_2)$, then $g^n(\cP_n(x))\subset \cP_\cN(g^n(x))$.
\end{proposition}
\begin{proof}
By Lemma~\ref{l.return0},
since $g^n(x)\in R_2$ we have $g^n(\cW^c_{R_2}(x))\subset \cN$,
hence $g^n(\cQ_{nm}(x))\subset \cN$.
The definition of $\widetilde \cQ_{nm}$ and $\cP_n^g$ conclude.
\end{proof}

Now we prove that when a point $x\in M^{n-1}_\beta$ comes back to $R_1$ by $g^n$, the image of $\cP_{n-1}(x)$ covers $R_1$.
\begin{proposition}\label{p.crossing}
For $\mu$-almost every point $x\in R_1\cap M^{n-1}_\beta\cap g^{-n}(R_1)$, the set $\cW^{cu}_{R_1}(g^{n}(x))$ is essentially included inside the image $g^{n}(\cP_{n-1}(x))$, i.e.
$\cW^{cu}_{R_1}(g^{n}(x))\setminus g^{n}(\cP_{n-1}(x))$ has zero $\mu^{cu}_{g^n(x)}$-measure.
\end{proposition}
\begin{proof}
Let us assume by contradiction that this is false and let us take $y$ such that $g^n(y)\in \cW^{cu}_{R_1}(g^{n}(x))\setminus g^{n}(\cP_{n-1}(x))$. We may assume that $x$ and $y$ belong to the full measure set $\Omega$ introduced in
Proposition~\ref{p.def-N}(f).
As the set $\cW^{cu}_{R_1}(g^n(x))$ has a c$\times$u-product structure,
there exists $y'$ such that $g^n(y')\in \cW^{u}_{R_1}(g^{n}(x))\cap \cW^{c}_{R_1}(g^{n}(y))$.
By Lemma~\ref{l.return0}, the set $\cW^{cu}_{R_1}(g^{n}(x))\setminus g^{n}(\cP_{n-1}(x))$
is saturated by plaques $\cW^c_{R_1}(z)$, hence $g^n(y')\notin g^{n}(\cP_{n-1}(x))$.
By definition of $\cP_{n-1}$, there exists $0\leq k<n$ such that $\widetilde{Q}_{m k}(y')\neq \widetilde{Q}_{m k}(x)$.
Two cases have to be considered.
\medskip

\paragraph{\it First case:
$g^k\cW^c_{R_2}(x)$ and $g^k\cW^c_{R_2}(y')$ are contained in $\cN$, but in different element of $\cP_\cN$.}
Let $z$ be the intersection point between
$\cW^{cs}_{\cN}(g^k(x))$ and $\cW^u_{\cN}(x_0)$.
By (g) we have $z\in B^{cs}_{L+1}(g^k(x))$. We can thus consider a cs-holonomy map
$\fh^{cs}$ between $B^u_{\rho'}(g^k(x))$ and a subset of $\cW^u_{loc}(x_0)$ satisfying
$\fh^{cs}(\zeta)\in B^{cs}_{2L}(\zeta)$.

Since $g^{n}(y')\in \cW^u_{R_1}(g^n(x))$,
the choice of $g$ gives $g^{n-1}(y')\in B^u_{\rho'}(g^{n-1}(x))$.
Let us consider an arc $\gamma$ joining $g^k(x),g^k(y')$ in $g^{k-n+1}B^u_{\rho'}(g^{n-1}(x))$.
It is contained in $B^u_{\rho'}(g^k(x))$, hence one can considers its image by $\fh^{cs}$.

Since $g^k(x),g^k(y')$ are not contained in a same element of $\widetilde Q_{mk}$,
the image of $\gamma$ by $\fh^{cs}$ meets the boundary of $\cP^u$ inside $\cW^u_{loc}(x_0)$.
We have thus proved that $g^{-j} B^u_{\rho}(g^{n-1}(x))$ intersects
$\partial_{2L}^{cs}\cP^u$, with $j=n-1-k$.
This is a contradiction since $x$ belongs to $M^{n-1}_\beta$.
\medskip

\paragraph{\it Second case:
$g^k\cW^c_{R_2}(x)$ or $g^k\cW^c_{R_2}(y')$ is contained in $\cN$, but not both.}
Let us assume for instance that $g^k(\cW^{c}_{R_2}(x))\subset\cN$
(the other situation is similar) and let $z_x\in Z$ be the point such that
$g^k(x)\in B^c_L(z_x)$.

The argument is similar to the first case.
We consider a cs-holonomy map
$\fh^{cs}$, satisfying $\fh^{cs}(\zeta)\in B^{cs}_{2L}(\zeta)$, between $B^u_{\rho'}(g^k(x))$ and a subset of $\cW^u_{loc}(x_0)$:
it may be obtained by first projecting by center holonomy on $\cW^u_{loc}(z_x)$,
and then by a local cs-holonomy on $\cW^u_{loc}(x_0)$.
In particular any point $\zeta\in B^u_{\rho'}(g^k(x))$ whose projection
$\fh^{cs}(\zeta)$ belongs to $D^u=\cW^u_\delta(x_0)$ is contained in
a center plaque $B^c_{L+1}(z)$ for some $z\in \cW^u_{2\delta}(z_x)\cap \cW^{su}_{\delta,\delta}(x_0)$.
There exists an arc $\gamma\subset B^u_{\rho'}(g^k(x))$ joining $g^k(x)$ and $g^k(y')$.
By construction, the image $\fh^{cs}(g^k(x))$ belongs to $D^u$.

We claim that the image $\fh^{cs}(g^k(y'))$ does not belong to $D^u$.
If the claim does not hold, $g^k(y')$ would be contained in some set $B^c_{L+1}(z')$, with $z'\in \cW^{su}_{\delta,\delta}(x_0)\cap \cW^u_{2\delta}(z_x)$. Note that $g^k(y)$ belongs to $B^c_{L-1}(g^k(y'))$, by Lemma~\ref{l.return0}.
Consequently $B^c_{2L}(z')$ contains $g^k(y)\in \Omega$.
By Proposition~\ref{p.def-N}(f), this implies $z'\in Z$.
The set $\cW^c_{R_2}(y')$ is the image of $\cW^c_{R_2}(x)$ by unstable holonomy
(which is uniquely defined). Hence $g^k\cW^c_{R_2}(y')$ is the image of $g^k\cW^c_{R_2}(x)$
by unstable holonomy. Since $g^k\cW^c_{R_2}(x)\subset \cN$ and since $\cN$ is invariant by the unstable holonomy
$\fh^u_{z_x,z'}$, we also have
$g^k\cW^c_{R_2}(y')\subset \cN$, a contradiction. The claim holds.

Since $\fh^{cs}(g^k(y'))$ does not belong to $D^u$,
the projection $\fh^{cs}(\gamma)$ meets the boundary of $D^u$,
and $g^{-j} B^u_{\rho}(g^{n-1}(x))$ intersects
$\partial_{2L}^{cs}D^u$, with $j=n-1-k$.
This is a contradiction since $x$ belongs to $M^{n-1}_\beta$.
\end{proof}

The following controls the diameter of the iterates of the partitions $\cP_n$.
\begin{proposition}\label{p.bound.diam}
For any $\varepsilon_0>0$ there exists $r'>0$ (large) and for every $n\geq 0$ there exists $Y_n\subset R_2$ satisfying $\mu(R_2\setminus Y_n)<\varepsilon_0$ such that if $x\in Y_n$ then $g^n(\cP_n(x)\cap \cW^u_{R_2}(x))$ has u-diameter smaller than $r'$.
\end{proposition}
The proof uses the following lemma.
\begin{lemma}\label{l.return}
Let $g:M\to M$ be a measurable transformation preserving a probability measure $\mu$, and let $A$ be a positive measure subset. We define $$
A_N^n=\lbrace x\in A,\text{ such that }g^j(x)\notin A\text{ for }n-N\leq j\leq n\}.
$$
Then for every $\varepsilon>0$ there is $N>0$ such that $\mu(A_N^n)<\varepsilon$ for every $n>N$. 
\end{lemma}
\begin{proof}
Let us consider the ergodic decomposition of $\mu=\int \mu_e d\widetilde{\mu}(e)$, where each $\mu_e$ is ergodic and gives total measure to disjoint sets $M_e$.

Let $E=\lbrace e,\text{such that }\mu_e(A)>0\rbrace$ and $M'=\cup_{e\in E} M_e$.
Since $\mu(A)>0$,
we have $\mu(M')>0$ and $g^{-n}(A)\subset_{\text{mod }0} M'$ for every $n\in \integer$.

By ergodicity, for each $e\in E$ and for $\mu_e$ almost every $x\in M_e$, there exists a smallest $N_e(x)\geq 0$ such that $g^{N_e(x)}(x)\in A$. Then we can define $\Psi:M'\to \natural$ by $\Psi(x)=N_e(x)$ if $x\in M_e$.

Now take $M_N=\lbrace x,\Psi(x)\geq N+1\rbrace$ and take $N$ sufficiently large such that $\mu(M_N)\leq \varepsilon$, then $A_N^n=A\cap g^{-n+N}M_N$.
\end{proof}
\begin{proof}[Proof of Proposition~\ref{p.bound.diam}]
To prove the proposition,
it is enough to bound uniformly the u-diameter of $g^n(\cP_n(x)\cap \cW^u_{R_2}(x))$ for $n$ large enough.
Let $A=R_2$ and from Lemma~\ref{l.return}, take $N$ large such that
$\mu(A_N^n)\leq \varepsilon_0$ for any $n\geq N$.
Let $Y_n:=A\setminus A^n_N$ so that
$\mu(R_2\setminus Y_n)<\varepsilon_0$ as required.

Let us choose any $x\in Y_n$. By definition of $Y_n$ and $A^n_N$, there exists $n-N\leq j \leq n$ such that $g^j(x)\in A$. By Proposition~\ref{p.separtion}, $g^j(\cP_j(x))\subset \cP_\cN(g^j(x))\subset \cN$; hence by (h) the u-diameter of $g^j(\cP_j(x))$ is smaller than~$r$. 

So the diameter of $g^n(\cP_j(x)\cap \cW^u_{R_2}(x))$ along $\cW^u(g^n(x))$ is bounded by $\norm{Dg^{n-j}}r\leq r':=\norm{D g}^N r$ for $n$ large enough.
\end{proof}
\begin{remark}\label{r.bound.diam}
The proofs of Propositions~\ref{p.separtion}
and~\ref{p.bound.diam} do not use the properties of the diffeomorphism $g$. Hence they are
also satisfied by the diffeomorphism $f$ and its associated partitions $\cP_n^f$.
However the proof of Proposition~\ref{p.crossing} uses strongly our definition of $g$.
\end{remark}

\subsubsection{Proof of Proposition~\ref{p.gap.persistence}}
The idea for proving~\eqref{e.reduction} is to decompose $\Gap(R_2,\cP_n)(z)$ by using corollary~\ref{c.iterates.gap} and to estimate
$\Gap(R_2\cap \cP_{n-1}(z),\cP_n)(z)$ by comparing it with $\Gap(R_1,\cP)(g^n(z))$ when $g^n(z)\in R_1$.
This is done by using
inclusions (up to a zero measure set) $$\cW^{cu}_{R_1}(g^n(z))\overset{\mu}\subset g^n\cP_{n-1}(z)\subset B^{cu}_{L,r'}(g^n(z)).$$
We recall that $B^{cu}_{L,r'}(g^n(z))$ has been defined in section~\ref{ss.ball}.

We have fixed $\beta>0$ in section~\ref{ss.choice}.
We apply Proposition~\ref{p.bound.diam} with $\varepsilon_0:=\beta$ to find $r'>0$ and $Y_n$ for each $n\in \natural$.
By Lusin theorem, one can replace the set $R'_1$ (chosen at section~\ref{ss.choice}) by a subset (with measure arbitrarily close to $\mu(R'_1)$) where
$z\mapsto \mu^{cu}_{R_1,z}(B^{cu}_{L,r'}(z))$ varies continuously.
Note that $z\mapsto \mu^{cu}_{R_1,z}(R'_1)$ is constant on each cu-plaque $\cW^{cu}_{R_1}(z)$.
We can thus replace $R'_1$ by a subset with measure arbitrarily close to $\mu(R'_1)$
where $z\mapsto \mu^{cu}_{R_1,z}(R'_1)$ varies continuously.
Consequently we can find $\alpha>0$ such that for every $z\in R_1'$,
$$\mu^{cu}_z(R_1')>\alpha \mu^{cu}_z(B^{cu}_{L,r'}(z)).$$
Since the new set $R'_1$ has been obtained by removing a small measure set of cu-plaques $\cW^{cu}_{R_1}(z)$,
the condition $\Gap(R_1,\cP)(z)>\eta$ in section~\ref{ss.choice} is still satisfied.

Let $z$ in a full measure subset of $R_1\cap M^{n-1}_\beta\cap Y_{n}$ such that $g^{n}(z)\in R_1'$.
By Proposition~\ref{p.crossing}, $\cW^{cu}_{R_1}(g^{n}(z))$ is essentially included in $g^{n}(\cP_{n-1}(z))$.
By Proposition~\ref{p.gap.prop} items (iv)
$$
\begin{aligned}
\Gap\big(g^{n}(\cP_{n-1}(z)\cap &\cW^{cu}_{R_2}(z)),g^{n}\cP_{n}\big)\geq\\
&\frac{\mu^{cu}(\cW^{cu}_{R_1}(g^{n}(z)))}{\mu^{cu}(g^{n}(\cP_{n-1}(z)\cap\cW^{cu}_{R_2}(z)))}\Gap(R_1,g^{n}\cP_{n})(g^{n}(z)).
\end{aligned}
$$

Proposition~\ref{p.bound.diam} and Lemma~\ref{l.return0} imply that $g^{n}(\cP_{n-1}(z)\cap \cW^{cu}_{R_2}(z))$ is included in $B^{cu}_{L,r'}(g^n(z))$.
With the choice of $\alpha$ one thus gets
$$
\Gap(g^{n}(\cP_{n-1}(z)\cap \cW^{cu}_{R_2}(z)),g^{n}\cP_{n})>
\alpha\Gap(R_1,g^{n}\cP_{n})(g^{n}(z)).
$$
 
By Proposition~\ref{p.separtion}, the restriction of $g^{n}\cP_{n}$ to $R_1$ is finer than $\cP$, so by Proposition~\ref{p.gap.prop} item (iii)
and the choice of $\eta$ in section~\ref{ss.choice} we get 
\begin{equation}\label{eq.delta}
\Gap(g^{n}(\cP_{n-1}(z)\cap \cW^{cu}_{R_2}(z)),g^{n}\cP_{n})>\alpha\eta.
\end{equation} 
Let $B_n=R_1\cap M^{n-1}_\beta\cap Y_n$. By corollary~\ref{c.iterates.gap}, the $\mu$-invariance, \eqref{eq.delta}, at any~$x$,
$$
\begin{aligned}
\Gap(&R_2,\cP_n)(x)\\
&\geq \int \sum_{j=1}^{n}  \sum_{P_{j-1}\in \cP_{j-1}}\chi_{R_1'}(g^j z)\chi_{B_{j}}(z)\chi_{P_{j-1}}(z)\Gap(P_{j-1}
\cP_{j})   d\mu^{cu}_{R_2,x}(z)\\
&\geq \alpha\eta\int \sum_{j=1}^{n} \chi_{R_1'}(g^j z)\chi_{B_{j}}(z) d\mu^{cu}_{R_2,x}(z)
\geq \alpha\eta \sum_{j=1}^{n}\mu^{cu}_{R_2,x}(g^{-j}R_1'\cap B_{j}).
\end{aligned}
$$
Integrating on $R_2$ and using the measure estimates on $Y_j$ and $M^j_\beta$, we get
\begin{equation}\label{eq.ineq.gap}
\tfrac{1}{n}\overline{\Gap}(R_2,\cP_n)\geq \alpha\eta\bigg(\frac{1}{n}\sum_{j=0}^{n-1}\mu(g^{-j}R_1'\cap R_1')-2\beta\bigg).
\end{equation}

Since $\mu$ is not necessarily ergodic for $g$ we will need the following lemma.
\begin{lemma}\label{l.vonneuman}
Let $g:M\to M$ be a measurable map, $\mu$ be an invariant measure and $A\subset M$. Then
$\lim_{n\to +\infty}\frac{1}{n}\sum_{j=0}^{n-1}\mu(g^{-j}A\cap A)\geq \mu(A)^2$.
\end{lemma}
\begin{proof}
By Von Neumann ergodic theorem 
$$
\lim_{n\to +\infty}\frac{1}{n}\sum_{j=0}^{n-1}\mu(g^{-j}A\cap A)=\int \chi_{A}P\chi_{A}d\mu,
$$
where $P$ is the orthogonal projection on invariant functions in $L^2(\mu)$. Then 
$\langle \chi_{A},P\chi_{A}\rangle=\langle P\chi_{A},P\chi_{A}\rangle=\langle P\chi_{A},P\chi_{A}\rangle \langle 1,1\rangle
\geq \langle P\chi_{A},1 \rangle^2=\mu(A)^2.$
\end{proof}
Taking the $\liminf$ on equation~\eqref{eq.ineq.gap} and applying Lemma~\ref{l.vonneuman} we have
$$
\liminf_{n\to +\infty} \tfrac{1}{n}\overline{\Gap}(R_2,\cP_n)\geq \alpha\eta(\mu(R_1')^2-2\beta)>0.
$$
This concludes the proof of Proposition~\ref{p.gap.persistence}.
\qed

\subsection{Upper bound on the transverse entropy}
We prove Proposition~\ref{p.entropy.gap}.

\begin{lemma}\label{l.ineq1}
$\displaystyle
\underset{n\to+\infty}
\limsup \frac{1}{n} \int \operatorname{H}_{\mu^{cu}_z}(R_2,\cP^f_n)d\mu(z) \leq h(f,\mu)\mu(R_2).
$
\end{lemma}
\begin{proof}
We first extend the partition $\cP_\cN$ to $M$:
$$\cP'_\cN:=\cP_\cN\cup\{M\setminus \cN\}.$$
In order to have a partition which refines $R_2$ by sets with small diameters, we consider
a finite measurable partition $\cA$ of $M$ whose atoms have diameter much smaller than the diameter of
the local unstable manifolds $\cW^u_{loc}(x)$ and define:
$$\widetilde \cP:=\cP'_\cN\vee\{R_2,M\setminus R_2\}\vee\cA.$$
For each $n\geq 0$ we consider the dynamical partition
$$\widetilde{\cP}^f_n=\bigvee_{j=0}^{n-1}f^{-j}\widetilde{\cP}.$$

\begin{claim}
For each $z\in Z$, the restriction of
$\widetilde \cP^f_n$ to $\cW^{cu}_{R_2}(z)$ is finer than the restriction of $\cP^f_n$.
\end{claim}
\begin{proof}
Let $x,y\in R_2$ with $\cW^{cu}_{R_2}(x)=\cW^{cu}_{R_2}(y)$
and assume that they belong to the same atom of $\widetilde \cP^f_n$.
We have to show that they also belong to the same atom of $\cP^f_n$.
Let us fix $0\leq j\leq n$.
By definition of $\cA$, we know that $f^j(x),f^j(y)$ are close. We also know that they
belong to the same element of $\cN$ or to $M\setminus \cN$.

First, let us assume that $f^j\cW^c_{R_2}(x)\subset \cN$
(the case where $f^j\cW^c_{R_2}(y)\subset \cN$ is done analogously).
In particular $f^j(y)\in \cN$.
One denotes by $z,z'\in Z$ the points satisfying
$x\in \cW^c_{R_2}(z)$ and $y\in \cW^c_{R_2}(z')$.
Since $f^k(x),f^k(y)$ remain close for $0\leq k\leq j$,
the point $f^j(y)$ belongs to $\cW^{cu}_{R_2}(f^j(x))$.
One can thus consider the set $Y\subset \cW^c_{\cN}(f^j(y))$
which is the image of $f^j\cW^c_{R_2}(x)$ by the local unstable holonomy:
$Y:=\fh^u_{z,z'}(\cW^c_{R_2}(x))\subset \cW^c_{\cN}(z')$.
Since unstable leaves are contracted by $f^{-j}$,
the set $f^{-j}(Y)$ is the image of $\cW^c_{R_2}(x)$ by local unstable holonomy.
By product structure of $\cW^{cu}_{R_2}(x)$, it coincides with $\cW^{c}_{R_2}(y)$.
One thus concludes that $f^j(\cW^{c}_{R_2}(y))=Y$ is contained in $\cN$.
Note also that the points $z,z'\in Z$ do not separate until time $j$. Hence $x,y$ belong to the same element of
the partition $\cQ_j$.

Consequently either the sets $f^j\cW^c_{R_2}(x)$ or $f^j\cW^c_{R_2}(y)$
are not included in $\cN$, then $x,y$ belong to the same atom of $\cQ_j$ and $\widetilde \cQ_j$;
or both of these sets are contained in $\cN$, then $x,y$ belong to the same element of $\cQ_j$
and by definition to the same element of $f^{-j}(\cP_\cN)$, hence to the same atom of $\widetilde \cQ_j$.
Since this holds for any $0\leq j\leq n$, this concludes the claim.
\end{proof}

The quotient of $R_2$ by the partition into sets $\cW^{cu}_{R_2}(z)$ coincides (up to a zero measure set) with the set $B^s_\delta(x_0)$.
We denote by $\widehat \mu$ the quotient of the measure $\mu|_{R_2}$.

By the claim and Jensen inequality we have
$$
\begin{aligned}
\int_{B^s_\delta(x_0)} \operatorname{H}_{\mu^{cu}_z}(R_2,\cP^f_n)&d\widehat \mu(z)= \int_{B^s_\delta(x_0)} -\sum_{P\in \cP^f_n}\mu^{cu}_{R_2,z}(P) \log \mu^{cu}_{R_2,z}(P)d\widehat \mu(z) \\
&\leq \int_{B^s_\delta(x_0)} -\sum_{\widetilde P\in \widetilde \cP^f_n}\mu^{cu}_{R_2,z}(\widetilde P) \log \mu^{cu}_{R_2,z}(\widetilde P)d\widehat \mu(z) \\
&\leq \sum_{\widetilde P\in \widetilde \cP^f_n,\;\widetilde P\subset R_2}-\mu(\widetilde P)\log \mu(\widetilde P)= \int_{R_2} -\log \mu(\widetilde \cP^f_n(x))d\mu.
\end{aligned}
$$
Dividing by $n$, taking the limsup, and using the definition of the entropy, we get the result.
\end{proof}

\begin{lemma}\label{l.ineq2}
$\displaystyle
\underset{n\to+\infty}\liminf \frac{1}{n} \int_{R_2} \operatorname{H}^u_{\mu^{cu}_z}(R_2,\cP^f_n)d\mu(z)\geq \mu(R_2)h^u(f,\mu).
$
\end{lemma}
\begin{proof}
We have that 
$\operatorname{H}^u_{\mu^{cu}_z}(R_2,\cP^f_n)=\int_{R_2}-\log \mu^u_x(\cP^f_n(x))d\mu^{cu}_z(x)$.

By \cite[Proposition 7.2.1 and Sections 9.2, 9.3]{LY85b} for $\mu$ almost every $x\in M$
$$
\lim_{\eta\to 0} \liminf_{n\to +\infty} -\tfrac{1}{n}\log\mu^u_x(B^u_\eta(x,n))=h^u(f,\mu),
$$
where 
$B^u_\eta(x,n)=\cap_{j=0}^{n} f^{-j} B^u_\eta (f^j(x)).$
Since $f$ expands uniformly in the unstable direction, given any $\eta'>{\eta}>0$ there exists $n_0$ such that $$
B^u_\eta(x,n)\subset B^u_{\eta'}(x,n)\subset B^u_\eta(x,n-n_0).$$ 
Hence
$$
\underset{n\to+\infty} \liminf -\tfrac{1}{n} \log\mu^u_x(B^u_\eta(x,n)) =\underset{n\to+\infty} \liminf -\tfrac{1}{n}\log \mu^u_x(B^u_{\eta'}(x,n)),
$$
$$
\lim_{n\to +\infty} -\tfrac{1}{n}\log\mu^u_x(B^u_\eta(x,n))=h^u(f,\mu).
$$
Fix some $\varepsilon>0$.
As noticed in remark~\ref{r.bound.diam}, Proposition~\ref{p.bound.diam} also holds for the diffeomorphism $f$:
this gives some number $r'>0$ and for each $n\geq 0$ a set $Y_n$.
There are $n_0\geq 1$, $R'\subset {R_2}$
with  $\mu({R_2}\setminus R')<\varepsilon$ such that for every $x\in R'$ and $n\geq n_0$,
$$
-\tfrac{1}{n}\log \mu^u_x(B^u_{r'}(x,n)\geq h^u(f,\mu)-\varepsilon.
$$

For $x\in Y_n$ we have $f^n(\cP^f_n(x)\cap\cW^u_{R_2}(x))\subset B^u_{r'}(f^n(x))$ and
$$
\begin{aligned}
\frac{1}{n}\int_{R_2} -\log \mu^u_x(\cP^f_n(x))d\mu &\geq \frac{1}{n}\int_{R'\cap Y_n} -\log \mu^u_x(\cP^f_n(x))d\mu\\
&\geq \frac{1}{n}\int_{R' \cap Y_n} -\log \mu^u_x(B^u_{r'}(x,n))d\mu\\
& \geq \mu(R'\cap Y_n)(h^u(f,\mu)-\varepsilon)\\
& \geq (\mu({R_2})-2\varepsilon)(h^u(f,\mu)-\varepsilon).
\end{aligned}
$$
Since $\varepsilon>0$ is arbitrary, this implies the inequality.
\end{proof}

The definition of $\overline{\Gap_{\mu^{cu}_z}}(R_2,\cP^f_n)$ and Lemmas~\ref{l.ineq1} and~\ref{l.ineq2} together
conclude the proof of Proposition~\ref{p.entropy.gap}. \qed

\subsection{Proof of the Main Theorem}
By Theorem~\ref{t.prod.struct} there exists a set $X$ with full $\mu$-measure
which has a global c$\times$u-product structure inside each leaf $\cW^{cu}(x)$.
We have to prove that the measure $\mu^{cu}_x$ on $\cW^{cu}(x)$ is a product.

Let us fix $\ell>0$ large and let us apply the construction of the beginning of section~\ref{s.gaps.entropy}
with scales $L,\ell_2,\ell_1> K\ell$.
This provides a set $R_1$ centered at some point $x_0$, with positive $\mu$-measure (by~\eqref{e.R1-positive}) and where Proposition~\ref{p.entropy.gap0} holds.
By Proposition~\ref{p.cover}, it can be applied to partitions with small boundary $\cP^u$ of $D^u$ into subsets with arbitrarily
small diameter, hence which generate the $\sigma$-algebra of $D^u$. Since $Z$ has a cs$\times$u-product structure, they induce partitions of each set $\cW^u_Z(y)$, which generate their $\sigma$-algebra.
Let $\cP$ be the partitions of $R_1$ induced by the partitions $\cP^u$.
Proposition~\ref{p.entropy.gap0} implies that for $\mu$-almost every point $z\in R_1$,
the transverse entropies $\Gap(\cW^{cu}_{R_1}(z),\cP|_{\cW^{cu}_{R_1}(z)})$ vanish.
By corollary~\ref{c.criterion}, each probability measure $\mu^{cu}_{R_1,z}$ is a product on the corresponding set
$\cW^{cu}_{R_1}(z)$ for almost every $z\in R_1$.

Almost every point $x\in X$ has arbitrarily large backward iterates in $R_1$ close to $x_0$.
We consider the large compact subset with product structure
$\cW^{cu}_{\ell,\ell}(x)$, i.e. the set of intersection points
$\cW^c(z)\cap \cW^u(y)$ for $(y,z)\in B^c_\ell(x)\times B^u_\ell(x)$.
By the quasi-isometric property,
there exists $n\geq 1$ large such that
$f^{-n}(\cW^{cu}_{\ell,\ell}(x))\subset 
\cup_{z\in \cW^{su}_{\delta,\delta}(x_0)} B^c_{K\ell}(z)$.
Note that $K\ell$ is smaller than the center size $\ell_1$ of $R_1$.
Then Property (f) in Proposition~\ref{p.def-N}, implies that
there exists a full measure set $\Omega$ such that
$f^{-n}(\cW^{cu}_{\ell,\ell}(x))\cap \Omega \subset \cW^{cu}_{R_1}(\zeta)$
for some $\zeta\in R_1$.
In particular, $\mu^{cu}_{f^{-n}(x)}|_{f^{-n}(\cW^{cu}_{\ell,\ell}(x))}$ is a product.
By invariance, this
proves that $\mu^{cu}_x|_{\cW^{cu}_{\ell,\ell}(x)}$ is a product.
As $\ell$ is arbitrary, this concludes that $\mu^{cu}_x$ is a product.
\qed


\section{measures with a local product structure}\label{s.meas.prod.structure}
In this section we prove Theorem~\ref{thm.continuous.disint}. As before, we consider a $C^1$ diffeomorphism $f$ which is partially hyperbolic, quasi-isometric in the center and an ergodic measure $\mu$. We show that if $\mu$ has some local product structure then its su-invariant disintegrations along the center plaques $\cW^c_{loc}(x)$ are continuous in the sense of Definition~\ref{d.local-center-measure}.
\medskip

\subsection{Quasi-invariant families of measures}
Given two transverse foliations $\cF$ and $\cG$,
points $x,y$ such that $\cG(x)=\cG(y)$,
$\varepsilon$-balls $B^{\cF}_{\varepsilon}(x)\subset \cF(x)$, $B^{\cF}_\varepsilon(y)\subset \cF(y)$ and some holonomy map
$\fh^{\cG}_{x,y}:B^{\cF}_\varepsilon(x)\to B^{\cF}_\varepsilon(y)$ along the leaves of $\cG$, we introduce the following definition:
\begin{definition}
The disintegration $\{\mu_x^\cF\}_{x\in M}$ of $\mu$ along the leaves of $\cF$
is \emph{$\cG$-quasi-invariant} if there exists $\varepsilon>0$ and a measurable set $X$
such that $\mu(X)=1$ and for any $x,y\in X$ with $y\in B^{\cG}_\varepsilon(x)$,
the measure $\mu^{\cF}_y|_{\fh^{\cG}_{y} B^\cF_\varepsilon(x)}$ is absolutely continuous with respect to ${\fh^{\cG}_{y}}_*\mu^{\cF}_x|_{B^\cF_\varepsilon(x)}$.

Observe that absolutely continuity of measures remains true if we multiply the measure by a factor, so the quasi-invariance is well defined.
\end{definition}
This is a weaker property than the invariance under $\cG$-holonomies.

\begin{proposition}\label{p.cs.prod.struct}
The following properties are equivalent:
\begin{enumerate}[(i)]
\item The family $\{\mu^u_x\}$ is cs-quasi-invariant.
\item The family $\{\mu^{cs}_x\}$ is u-quasi-invariant.
\item $\mu$ has local cs$\times$u-product structure.
\end{enumerate}
\end{proposition}
\begin{proof}
As in Definition~\ref{def.mes.prod.st},
given $x_0\in M$ and $\varepsilon>0$, we consider the set $\cN^{cs\times u}_\varepsilon(x_0)$ with local cs$\times$u-product structure,  i.e. the image of $B^{cs}_{\varepsilon}(x_0)\times B^{u}_\varepsilon(x_0)$
by the homeomorphism $\Psi\colon (x,y)\mapsto \fh^u_y(x)$.

Let us first assume that the family $\{\mu^u_x\}$ is cs-quasi-invariant
and let us fix the total measure set $X$ where the cs-quasi-invariance is satisfied. For $\mu$ almost every point $x_0\in X$, one notices that $\mu^{cs}_{x_0}$-almost every $y\in \cW^{cs}(x_0)$ belongs to $X$.
Let us take $\varepsilon>0$ small such that $\cN^{cs\times u}_\varepsilon(x_0)$ has local cs$\times$u-product structure. We are going to work on $B^{cs}_{\varepsilon}(x_0)\times B^{u}_\varepsilon(x_0)=\Psi^{-1}(\cN^{cs\times u}_\varepsilon(x_0))$ endowed with
the measure $\tmu=\Psi^*\mu\mid_{\cN^{cs\times u}_\varepsilon(x_0)}$. As $x_0$ and $\varepsilon$ are fixed, to simplify the notation we write $B^*:=B^*_{\varepsilon}(x_0)$ for $*=cs,\,u$.

By Rokhlin disintegration formula we can write 
$$
\tmu=\int_{B^{cs}\times \{x_0\}}\mu^u_y d\mu^{cs}(y),$$ 
where $\mu^{cs}$ is the projection of $\tmu$ by $\pi^u:B^{cs}\times B^{u}\to B^{cs}\times \{x_0\}$,
defined by $(x^{cs},x^{u})\mapsto (x^{cs},x_0)$.

Observe that using these coordinates with product structure, the cs-holo\-no\-mies take the form $\fh^{cs}_{(y^{cs},y^u)}(x^{cs},x^u)=(y^{cs},x^u)$. Then the cs-quasi-invariance and Radon-Nikodym theorem imply that for $\mu^{cs}_{x_0}$ almost every $y\in B^{cs}\times \{x_0\}$ we have $\mu^u_y=\rho(\cdot,y)\mu^u_{x_0}$. So we get $\tmu=\rho \mu^{cs}\times \mu^{u}_{x_0}$.

We have proved (i)$\Rightarrow$(iii). The same argument shows (ii)$\Rightarrow$(iii).
The converse direction (iii)$\Rightarrow$(i)\&(ii) is trivial.
\end{proof}

\begin{proposition}\label{p.prod.eq.csquasi}
If $\{\mu^s_x\}$ is cu-quasi-invariant and $\{\mu^c_x\}$ is u-invariant then $\{\mu^{cs}_x\}$ is u-quasi-invariant.
\end{proposition}
\begin{proof}
Using the local product structure of cs-manifolds we can find $\varepsilon>0$ sufficiently small, such that for every $x\in M$, the map $\Psi:B^{c}_\varepsilon(x)\times B^s_\varepsilon(x)\to \cW^{cs}(x)$ defined by $\Psi(y,z)= \fh^s_{z}(y)$ is a homeomorphism over its image, that will be denoted by $\cN^{cs}_\varepsilon(x)$.

Let $X$ be a full measure set of points that satisfy the cu-quasi-invariance of $\{\mu^s_x\}$ and the u-invariance of $\{\mu^c_x\}$. 

By the $cu$-quasi invariance we can take $x_0\in M$ and $y_0\in \cW^u(x_0)$ such that $\mu^c_{x_0}(M\setminus X)=\mu^c_{y_0}(M\setminus X)=0$.

To see this observe that it exists $Y\subset M$, $\mu(Y)=1$ such that for $y\in Y$, $\mu^c_y(M\setminus X)=0$, then for $\mu$ almost every $x\in M$, $\mu^s_x(M\setminus Y)=0$.
We can take $x_0$ with this last property and a total measure set of $y\in M$ such that  $W^{cu}(x_0)\cap W^s(y)$ and $\mu^s_y(M\setminus Y)=0$, by the $cu$-quasi invariance, up to changing $x_0$ to a $\mu^s_{x_0}$ total measure set, we can assume that $x_0$ and some $y'\in W^{cu}(x_0)\cap W^s(y)$ belongs to $Y$. Taking $y_0\in W^u(x_0)\cap W^c(y')$ we have the property we wanted.

There exists $0<\varepsilon'\leq \varepsilon$ such that the u-holonomy $\fh^u:\cN^{cs}_{\varepsilon'}(x_0)\to \cW^{cs}(y_0)$ is a well defined homeomorphism over its image. Up to taking $\varepsilon'$ smaller we can assume that $\fh^u(\cN^{cs}_{\varepsilon'}(x_0))\subset \cN^{cs}_\varepsilon(y_0)$. See Figure~\ref{fig.teocsu}.
\begin{figure}[h]
\centering
\includegraphics[scale= 0.6]{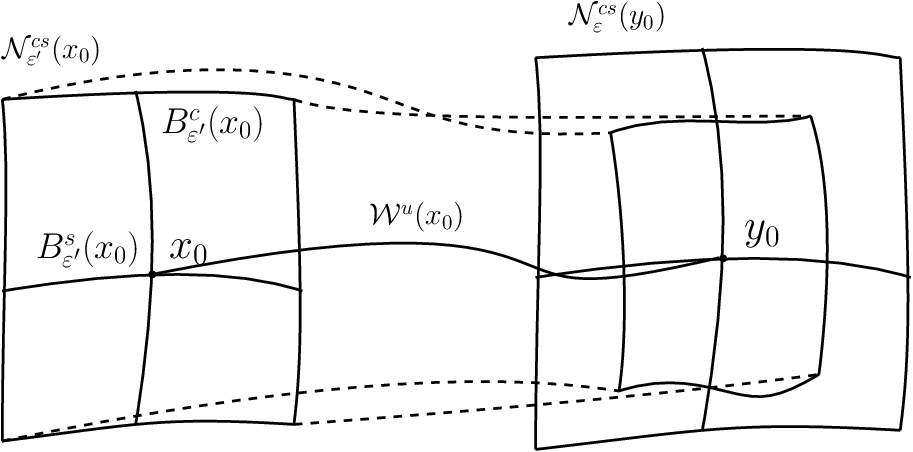}
\caption{Definition the holonomy $\fh^u:\cN^{cs}_{\varepsilon'}(x_0)\to \cN^{cs}_{\varepsilon}(y_0)$.}
\label{fig.teocsu}
\end{figure}

Take any measurable set $A\subset \cN^{cs}_{\varepsilon'}(x_0)$. To conclude the proposition, we have to prove that $\mu^{cs}_{x_0}(A)=0$ if and only if $\mu^{cs}_{y_0}(\fh^u(A))=0$. As the normalization of the measures does not matter, from now on we restrict $\mu^{cs}_{x_0}$ to $\cN^{cs}_{\varepsilon'}(x_0)$ and normalize it such that $\mu^{cs}_{x_0}(\cN^{cs}_{\varepsilon'}(x_0))=1$.

We have 
$$
\mu^{cs}_{x_0}(A)=\int_{B^s_{\varepsilon'}(x_0)} \mu^c_z(A) d \pi^c_* \mu^{cs}_{x_0}(z),
$$
where $\pi^c:\cN^{cs}_{\varepsilon'}(x_0)\to B^s_{\varepsilon'}(x_0)$ is the projection by the center holonomy. 

Observe that by the cu-quasi-invariance of $\{\mu^s_x\}$, for $\mu^c_{x_0}$ almost every $x\in B^c_{\varepsilon'}(x_0)$ the measure $\mu^s_x$ is absolutely continuous with respect to $({\fh^{cu}_{x}})_* \mu^s_{x_0}$. Now observe that for any $C\subset B^s_{\varepsilon'}(x_0)$ 
$$
\pi^c_* \mu^{cs}_{x_0}(C)=\int_{B^c_{\varepsilon'}(x_0)}\mu^s_x(\fh^{cu}_{x}(C))d \pi^s_* \mu^{cs}_{x_0}(y),
$$
where  $\pi^s:\cN^{cs}_{\varepsilon'}(x_0)\to B^c_{\varepsilon'}(x_0)$ is the projection using the stable holonomy.
Since $\{\mu^s_x\}$ is cu-quasi-invariant,
$(\fh^{cu}_{x_0})_*\mu^s_x$ is absolutely continuous with respect to $\mu^s_{x_0}$.
So $\pi^c_* \mu^{cs}_{x_0}$ is absolutely continuous with respect to $\mu^s_{x_0}$. Hence
there exists a positive function $\rho$ such that
\begin{equation}\label{eq.mucsx0}
\mu^{cs}_{x_0}(A)=\int_{B^s_{\varepsilon'}(x_0)} \mu^c_z(A) \rho(z)d \mu^{s}_{x_0}(z).
\end{equation}
Analogously there exists a positive measurable function $\rho'$ such that
$$
\mu^{cs}_{y_0}(\fh^u(A))=\int_{B^s_{\varepsilon}(y_0)} \mu^c_{z'}(\fh^u(A)) {\rho'}(z')d \mu^{s}_{y_0}(z').
$$
Here we normalize such that $\mu^{cs}_{y_0}(\fh^u(\cN^{cs}_{\varepsilon'}(x_0))=1$, so $\fh^u_*\mu^c_z=\mu^c_{\fh^u(z)}$ for $\mu^{cs}_{x_0}$ almost every $z\in \cN^{cs}_{\varepsilon'}(x_0)$. 

Let us consider the cu-holonomy $\fh^{cu}:B^s_{\varepsilon'}(x_0)\to  B^s_{\varepsilon}(y_0)$ such that $\fh^{cu}(z)\in B^c_\varepsilon(\fh^u(z))\cap B^s_{\varepsilon}(y_0)$.
The u-invariance of $\{\mu^c_x\}$ gives (see Figure~\ref{fig.teocsu2}),
$$
\mu^{cs}_{y_0}(\fh^u(A))=\int_{B^s_{\varepsilon}(y_0)} \mu^c_{{(\fh^{cu}})^{-1}(z')}(A) {\rho'}(z')d \mu^{s}_{y_0}(z').
$$ 
\begin{figure}[h]
\centering
\includegraphics[scale= 0.6]{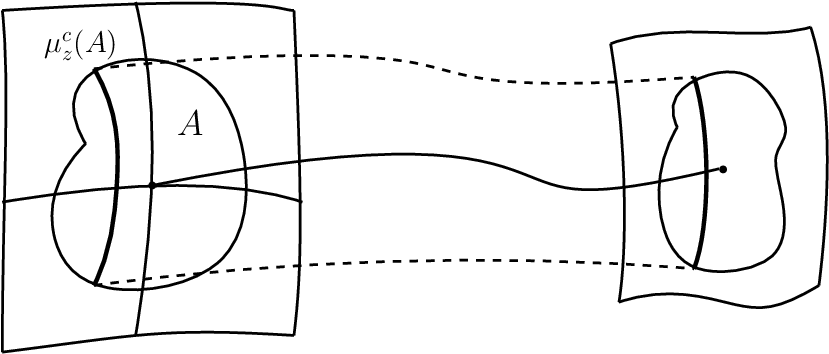}
\put(-35,60){\small $\fh^u(A)$}
\put(-45,85){\tiny $\mu^c_{z'}(\fh^u(A))$}
\caption{$\mu^c_z(A)=\mu^c_{z'}(\fh^u(A))$.}
\label{fig.teocsu2}
\end{figure}

The cu-quasi-invariance of $\{\mu^s_x\}$ gives a positive measurable function $\widetilde \rho$ such that
\begin{equation}\label{eq.mucsy0}
\mu^{cs}_{y_0}(\fh^u(A))=\int_{B^s_{\varepsilon'}(x_0)} \mu^c_{z}(A) {\rho}(\fh^{cu}(z))\widetilde\rho(z)d \mu^{s}_{x_0}(z).
\end{equation}
So from \eqref{eq.mucsx0} and \eqref{eq.mucsy0} we get $\mu^{cs}_{x_0}(A)=0$ if and only if $\mu^{cs}_{y_0}(\fh^u(A))=0$.
\end{proof}

\subsection{Continuous families of center measures: proof of Theorem~\ref{thm.continuous.disint}}
\label{ss.continuous-family}
We first complement Definition~\ref{d.local-center-measure}:
\begin{definition}\label{d.local-center-measure2}
A family of local center measures $\{\nu^c_x\}_{x\in X}$ is \emph{$f$-invariant} if there is $\varepsilon>0$ such that,
for any $x\in X$, there exists $K'_{x}>0$ satisfying
$$f_*(\nu^c_x|_{B^c_\varepsilon(x)})=K'_{x}. \nu^c_{f(x)}|_{f(B^c_\varepsilon(x))}.$$
It is \emph{u-invariant} if there is $\varepsilon>0$ such that,
for any $x,y\in X$ with $y\in B^{cu}_\varepsilon(x)$, there exists $K_{x,y}>0$ satisfying
$$(\fh^u_{y})_*(\nu^c_x|_{B^c_\varepsilon(x)})=K_{x,y}. \nu^c_y|_{\fh^u_y(B^c_\varepsilon(x))}.$$
We define analogously the \emph{s-invariance}.
\end{definition}

By the Main Theorem, if a measure $\mu$ satisfies $h^s(f,\mu)=h^u(f,\mu)=h(f,\mu)$,
then its center disintegrations $\{\mu^c_x\}$ are s- and u-invariant.
If moreover $\mu$ has local cs$\times$u-product structure, then the unstable disintegrations
$\{\mu^u_x\}$ are cs-quasi-invariant.
Theorem~\ref{thm.continuous.disint} is now a consequence of the following proposition.

\begin{proposition}\label{p.cont.dis}
If $\{\mu^u_x\}$ is cs-quasi-invariant and $\{\mu^c_x\}$ is s- and u-in\-va\-ri\-ant then there there exists a
family of local center measures $\{\nu^c_x\}_{x\in \supp(\mu)}$
which is continuous, $f$-invariant, s-invariant, u-invariant and extends the center disintegration of $\mu$.\end{proposition}

\begin{addendum}\label{a.cont.dis}
If $f$ is a discretized Anosov flow and the support of $\mu$ is not contained in finitely many compact center leaves,
then one can:
\begin{itemize}
\item normalize the center measures so that $\mu^c_x([x,f(x)))=1$ for $\mu$-ae $x$,
\item choose $\{\nu^c_x\}$ so that $K_x=K'_x=K_{x,y}=1$ in Definitions~\ref{d.local-center-measure} and~\ref{d.local-center-measure2}.
\end{itemize}
\end{addendum}
\begin{proof}[Proof of Proposition~\ref{p.cont.dis}]
Let $X$ a set of total $\mu$-measure satisfying the definition of s- and u-invariance for $\{\mu^c_x\}$
and of cs-quasi-invariance for $\{\mu^u_x\}$.
There exists a dense set of points $x\in \supp(\mu)$ such that $\mu^{cs}_x$ almost every point in $\cW^{cs}(x)$ and $\mu^s_x$ almost every point in $\cW^{s}(x)$ are contained on $X$, and moreover $x$ belongs to $X$. Fix one of these points and denote it by $x_0$.

For $\varepsilon>0$ small, the map $\Psi: B^{cs}_\varepsilon(x_0)\times B^{u}_\varepsilon(x_0)\to \cN^{cs\times u}_\varepsilon(x_0)$ defined by $(y,z)\mapsto \fh^u_{z}(y)$ is a homeomorphism. Take $\varepsilon'>0$ sufficiently small such that $\cN^{cs}_{\varepsilon'}(x_0)$ is contained in $B^{cs}_\varepsilon(x_0)$, where $\cN^{cs}_{\varepsilon'}(x_0)$ is defined as in the proof of Proposition~\ref{p.prod.eq.csquasi}. Now take $V=\Psi(\cN^{cs}_{\varepsilon'}(x_0)\times B^u_\varepsilon(x_0))$, see Figure~\ref{fig.teocontdis}.
For $x\in V$, we define $\cW^*_V(x):= \cW^*_{loc}(x)\cap V$, for $*=c,u,s,cu,cs$.
\begin{figure}[h]
\centering
\includegraphics[scale= 0.6]{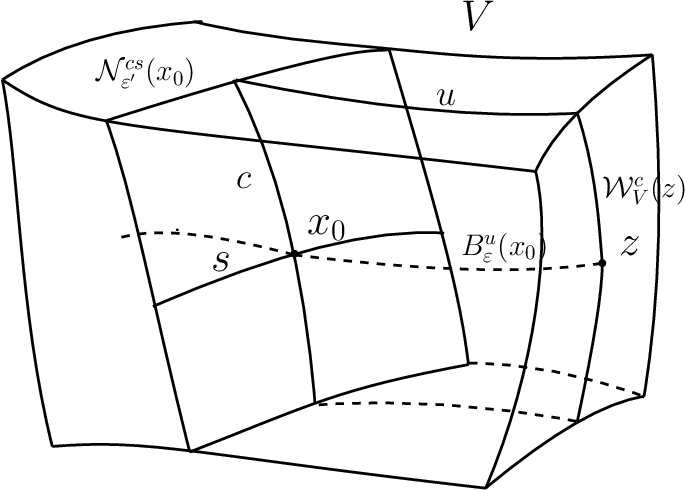}
\caption{Definition of $V$.}
\label{fig.teocontdis}
\end{figure}

By local c$\times$s-product structure of $ \cN^{cs}_{\varepsilon'}(x_0)$, for every $y,z\in \cN^{cs}_{\varepsilon'}(x_0)$ there is a homeomorphism $\fh^s_{z}:\cW^c_V(y)\to \cW^c_V(z)$ where $\fh^s_{z}(t)$ is the unique intersection point between $\cW^s_V(t)$ and $\cW^c_V(z)$. Let us consider the center disintegration $\{\mu^c_x\}_{x\in V}$ of $\mu|_V$ normalized by $\mu^c_x(\cW^c_V(x))=1$. By our assumption, one can assume that it is defined at each point of
$X$ and that it is s- and u-invariant.

Let us consider the measure $\mu^c_{x_0}$.
For every $z\in \cW^s_V(x_0)$ we can define $\widehat \nu^{x_0}_z=({\fh^s_{z}})_*(\mu^c_{x_0}|_{\cW^c_V(x_0)})$.
For $\mu^{cs}_{x_0}$ almost every $x\in \cN^{cs}_{\varepsilon'}(x_0)$,
there exists $z\in \cW^s_{\varepsilon'}(x_0)$ such that $\cW^c_V(z)=\cW^c_V(x)$; then
the s-invariance and the choice of the normalization give $\mu^c_x|_{\cW^c_V(x)}=\widehat \nu^{x_0}_z|_{\cW^c_V(z)}$.

By construction of $V$, for every $z\in \cW^s_V(x_0)$ and any $y\in \cW^u_V(z)$ there exists a homeomorphism $\fh^u_{y}:\cW^c_V(z)\to \cW^c_V(y)$ such that $\fh^u_{y}(t)$ is the only intersection point between $\cW^u_V(t)$ and $\cW^c_V(y)$. 
We then define $\widehat \nu^{x_0}_y=({\fh^u_{y}})_*\widehat \nu^{x_0}_z$. 

Let us take $x$ such that $\mu^{cs}_x(X)=1$. By the u-quasi-invariance of $\{\mu^{cs}_\zeta\}$, the projection of $X\cap \cW^{cs}_V(x_0)$ to $\cW^{cs}_V(x)$ by the u-holonomy has total $\mu^{cs}_x$ measure,
and the u-invariance of $\{\mu^c_\zeta\}$ gives $\widehat \nu_x^{x_0}=\mu^c_x$.
We have defined $\widehat \nu^{x_0}_x$ for any $x\in \Psi(B^s_{\varepsilon'}(x_0)\times B^u_\varepsilon(x_0))$.
We can extend it to $V$ by requiring $\widehat \nu_x^{x_0}=\widehat \nu_{x'}^{x_0}$ for any $x,x'\in V$ satisfying $\cW^c_V(x)=\cW_V^c(x')$.

The u-invariance of $\{\mu^c_x\}$ and the u-quasi-invariance of $\{\mu^{cs}_x\}$
then imply that $\mu$-almost every $x\in V$ satisfies $\widehat \nu_x^{x_0}=\mu^c_x|_{\cW^c_V(x)}$.
Hence the family of measures $\{\widehat \nu_x^{x_0}\}_{x\in V}$ disintegrates $\mu|_V$ alongs the plaques $\cW^c_V(x)$;
by construction it is u-invariant as in Definition~\ref{d.local-center-measure2} and continuous.
\medskip

Note that there exist $\delta,\eta>0$ such that for any $x$ that is $\eta$-close to $x_0$,
the $\delta$-neighborhood of $x$ in $\cW^c(x)$ is contained in $\cW^c_V(x)$.
Let us choose a continuous function
$\rho \colon [0,\delta]\to [0,1]$ which satisfies
$\rho([0,\delta/2])=1$ and  $\rho(\delta)=0$.
We also denote $d_x\colon \cW^c_V(x)\to \mathbb{R}_+$ the distance to $x$ inside the plaque $\cW^c(x)$.
Finally, we define the measures
$\nu^{x_0}_x:= (\rho\circ d_x).\widehat \nu_x^{x_0}$ along the plaques $\cW^c_V(x)$ for any point $x\in B_\eta(x_0)$.
We have thus defined a family of local center measures
$\{\nu_x^{x_0}\}_{x\in B_\eta(x_0)}$ which is u-invariant and continuous.
\medskip

Since $x_0$ can be chosen in any full measure set of $\mu$,
the measure $\nu^{x_0}_{x_0}$ gives positive mass
to any small neighborhood of $x_0$ in $\cW^c(x_0)$.
In particular, the construction we have done implies that
for any point $x$ in a uniform neighborhood of $x_0$ the measure $\nu^{x_0}_{x_0}$ does not vanish.

Considering a different point $y_0$, one associates a different set $V$
and a different family of measures $\nu^{y_0}_x$.
Note that if $x_0,y_0$ are close enough and if one considers a point $x\in B_\eta(x_0)\cap B_\eta(y_0)\cap X$,
then there exists $K>0$ such that
$\widehat \nu_x^{y_0}|_{B^c_\delta(x)}=K. \widehat \nu_x^{x_0}|_{B^c_\delta(x)}$.
By construction of the measures $\nu^{x_0}_x$ and $\nu^{y_0}_x$, we then have $\nu_x^{y_0}=K. \nu_x^{x_0}$.
By continuity of the families and using the fact that
$\nu_x^{x_0}$ and
$\nu_x^{y_0}$ do not vanish,
this also holds for any
$x\in B_\eta(x_0)\cap B_\eta(y_0)\cap \supp(\mu)$.
\medskip

In order to globalize the construction, one considers
a finite subset $\{x_i\}\subset \supp(\mu)$ which is enough dense in $\supp(\mu)$
and associates a partition of the unity, i.e. a family of continuous functions
$\varphi_i\colon M\to [0,1]$ such that $\sum \varphi_i=1$,
and each $\varphi_i$ is supported on a small neighborhood of $x_i$.
We then define $\nu^c_x=\sum_i\varphi_i(x)\cdot \nu^{x_i}_x$
for any $x\in \supp(\mu)$. (By convention, we set $\varphi_i(x)\cdot \nu^{x_i}_x=0$
when $\varphi_i(x)=0$ even if the measure $\nu^{x_i}_x$ has not been defined.)
We have thus defined the family of local center measures
$\{\nu^c_x\}_{x\in \supp(\mu)}$ and we have to check that the announced properties are satisfied.
Note that by construction, the mass of the measures $\nu^c_x$ is uniformly bounded away from~$0$.

Since the family $\{\nu^{x_i}_x\}$, for each $i$, is continuous,
the obtained family $\{\nu^c_x\}_{x\in \supp(\mu)}$ is also continuous.
By construction, for $\mu$-almost every point $x$ and any point $x_i$,
there exists $K_i> 0$ such that
$\nu^{x_i}_x$ coincides with $K_i\cdot \mu^c_x$ on the ball $B^c_{\delta/2}(x)$;
hence the family $\{\nu^c_x\}_{x\in \supp(\mu)}$ extends the center disintegrations of $\mu$.
Since each family $\{\nu^{x_i}_x\}$ is u-invariant,
one deduces that $\{\nu^c_x\}_{x\in \supp(\mu)}$ is u-invariant as well.

There exists $\delta'<\delta/2$ such that for any $x\in M$,
the image of $B^c_{\delta'}(x)$ is contained in $B^c_{\delta/2}(f(x))$.
Since the family of center disintegrations $\mu^c_x$
is $f$-invariant, for $\mu$-almost every point $x$ there exists
$K_x>0$ such that $f_*\nu^c_{x}=K_x\cdot \nu^c_{f(x)}$.
Taking the limit, this properties extends to any $x\in \supp(\mu)$
and $\{\nu^c_x\}_{x\in \supp(\mu)}$ is $f$-invariant.

By Proposition~\ref{p.prod.eq.csquasi} we can exchange the role of s and u to build in a same way
a family of local center measures $\{\widetilde \nu^c_x\}_{x\in \supp(\mu)}$ which is continuous
and s-invariant and extends the center disintegration of $\mu$.
In particular for $\mu$-almost every point $x$ there exists $K>0$
such that $\nu^c_x=K\cdot \widetilde \nu^c_x$.
Since the mass of the measures $\nu^c_x$ and $\widetilde \nu^c_x$
are uniformly bounded away from $0$, this property extends to any point $x\in \supp(\mu)$.
The s-invariance of $\{\widetilde \nu^c_x\}_{x\in \supp(\mu)}$ then implies the s-invariance
of $\{\nu^c_x\}_{x\in \supp(\mu)}$.
\end{proof}
\begin{proof}[Proof of Addendum~\ref{a.cont.dis}]
As $f$ is a discretized Anosov flow and the measure is not concentrated on compact center leaves, for $\mu$-almost every point $x$, the leaf $\cW^c(x)$ is a line.
We follow the proof of Proposition~\ref{p.cont.dis}, but choose the different normalization:
$\mu^c_x([x,f(x)))=1$ for the center measures.

As $f$ is center fixing $f_* \mu^c_x=\mu^c_x$, see \cite[Proposition~3.3]{AVW11}. We also have that $f_*\mu^c_x=K_x \mu^c_{f(x)}$ but by the choice of the normalization $f_*\mu^c_x=\mu^c_{f(x)}$ so $\mu^c_x=\mu^c_{f(x)}$.
For any $y\in \cW^c(x)$ take $n$ such that $y\in [f^n(x),f^{n+1}(x))$. Then 
$$
\mu^c_{f^n(x)}\left([y,f(y))\right)=\mu^c_{f^n(x)}\left([y,f^{n+1}(x))\right)+\mu^c_{f^{n+1}(x)}\left([f^{n+1}(x),f(y))\right).$$
As $\mu^c_{f^{n+1}(x)}\left([f^{n+1}(x),f(y))\right)=\mu^c_{f^{n}(x)}\left([f^n(x),y)\right)$ we have $\mu^c_{f^n(x)}\left([y,f(y))\right)=1$. So by the choice of the normalization $\mu^c_x=\mu^c_{f^n(x)}=\mu^c_y$ for almost every $y\in \cW^c(x)$.

Now observe that for $x,y$ in the same $s/u$ manifold by the $s/u$ invariance we have that ${h^{s/u}_{x,y}}_*\mu^c_x=K_{x,y}\mu^c_y$, but by the invariance of the $s/u$ foliations by $f$, ${h^{s/u}_{x,y}}([x,f(x))])=[y,f(y))$ so $K_{x,y}=1$. 

By the $s/u$ invariance $x\mapsto \mu^c_x$, proceeding as in proposition~\ref{p.cont.dis} we can construct $\nu_x^c$ keeping the normalization, then we get $K_x,K'_x,K_{x,y}=1$.
\end{proof}

\subsection{Absolutely continuous center disintegrations}
A partially hyperbolic diffeomorphism is \emph{center-bunched}
if the functions $\nu,\widehat \nu,\gamma,\widehat \gamma\colon M\to \mathbb{R}$
in the definition of partial hyperbolicity (see the introduction) satisfy moreover:
$$\nu<\gamma\cdot\widehat \gamma\quad  \text{and} \quad \widehat \nu<\gamma\cdot\widehat \gamma.$$
This is satisfied when the center bundle is one-dimensional.

\begin{proposition}\label{p.su.implies.leb}
Let $f$ be a partially hyperbolic, quasi-isometric in the center, 
accessible, center bunched,  $C^2$ diffeomorphism.
For every ergodic measure $\mu$ with local cs$\times$u-product structure, full support and
satisfying $h(f,\mu)=h^u(f,\mu)=h^s(f,\mu)$,
the center disintegrations $\{\mu^c_x\}$ are equivalent to the Lebesgue measures
along center leaves.
\end{proposition}
\begin{proof}
By Proposition~\ref{p.cs.prod.struct}, $\{\mu^u_x\}$ is cs-quasi-invariant.
By the Main Theorem, $\{\mu^c_x\}$ is s- and u-invariant.
Since $\mu$ has full support, by Proposition~\ref{p.cont.dis}, one can find a family of local center measures
$\{\nu^c_x\}_{x\in M}$ which extends the center disintegration of $\mu$
and which is continuous with respect to $x$.

Given $x,y \in M$, by accessibility there exists a su-path from $x$ to $y$. Moreover by compactness of $M$ the number of legs and the lengths of each leg can be taken uniformly bounded (independent from the points $x,y$), see \cite[Lemma~4.5]{Wliv}. 

Fix $x,y\in M$ and the su-path connecting $x$ to $y$.
It induces a map $h:B^c_\gamma(x)\to \cW^c(y)$ for some $\gamma>0$,
satisfying  $h(x)=y$ and defined as a composition of stable and unstable holonomies. The center bunching condition implies that the holonomies are Lipchitz inside center stable/unstable manifolds, uniformly in paths with boun\-ded lengths, see \cite[Theorem~B]{PSW97},
hence $h$ is absolutely continuous. As the number and lengths of legs are bounded, the Jacobian of $h$ is uniformly bounded by some constant $K>1$.
Then for every $\varepsilon>0$ sufficiently small we have
$$
K^{-1} \vol^c(B^c_\varepsilon(x))\leq \vol^c(h(B^c_\varepsilon(x))) \leq K \vol^c(B^c_\varepsilon(x)).
$$
By s- and u-invariance of $\{\nu^c_x\}$ we get for any $x,y$:
$$
{K_{xy}}^{-1} \frac{\nu^c_x(B^c_\varepsilon(x))}{\vol^c(B^c_\varepsilon(x))}\leq \frac{\nu^c_y (h(B^c_\varepsilon(x)))}{\vol^c(h(B^c_\varepsilon(x)))} \leq K_{xy} \frac{\nu^c_x(B^c_\varepsilon(x))}{\vol^c(B^c_\varepsilon(x))},
$$
where $K_{xy}>0$ depends on $x,y$ but not on $\varepsilon$.

By the uniformly Lipchitz property there exists $C>0$ such that  $B^c_{C^{-1}\varepsilon}(y)\subset h(B^c_\varepsilon(x))\subset B^c_{C\varepsilon}(y)$. Hence:
\begin{equation}\label{e.density}
{K_{xy}}^{-1} \frac{\nu^c_y(B^c_{C^{-1}\varepsilon}(y))}{\vol^c(B^c_{C\varepsilon}(y))}\leq
\frac{\nu^c_x (B^c_\varepsilon(x))}{\vol^c(B^c_\varepsilon(x))} \leq
K_{xy} \frac{\nu^c_y(B^c_{C\varepsilon}(y))}{\vol^c(B^c_{C^{-1}\varepsilon}(y))}.
\end{equation}

If some center measure $\mu^c_z$ has a singular part with respect to the Lebesgue measure,
there exists $x$ such that $\lim_{\varepsilon\to 0} \frac{\nu^c_x(B^c_\varepsilon(x))}{\vol^c(B^c_\varepsilon(x))}=+\infty$. On the other hand, for Lebesgue almost every point $y\in \cW^c(x)$,
$\limsup_{\varepsilon\to 0} \frac{\nu^c_y(B^c_{C\varepsilon}(y))}{\vol^c(B^c_{C^{-1}\varepsilon}(y))}$ is finite.
This contradicts~\eqref{e.density}.
Analogously if $\vol^c$ has a singular part with respect to $\mu^c_z$ it will contradict \eqref{e.density}.
We have proved that $\mu^c_z$ is equivalent to $\vol^c$.
\end{proof}

\begin{corollary}\label{c.su.1d.implies.flow}
Let $f$ be an accessible $C^2$ discretized Anosov flow.
If there exists an ergodic measure $\mu$, with local cs$\times$u-product structure, full support and zero center Lyapunov exponent then $f$ is the time-$1$ map of a topological Anosov flow which is as smooth as $f$ along the center direction.
\end{corollary}
This proof follows~\cite{AVW11}; for completeness we give the details.
\begin{proof}
Since the center bundle is one-dimensional, $f$ is center-bunched.
Since $\mu$ has zero center Lyapunov exponent, $h(f,\mu)=h^u(f,\mu)=h^s(f,\mu)$,
hence by Proposition~\ref{p.su.implies.leb} the measures $\{\mu^c_x\}$ are equivalent to the Lebesgue measure along the center leaves.

Moreover by the Main Theorem, $\{\mu^c_x\}$ is s- and u-invariant.
Since $\mu$ has local cs$\times$u-product structure, $\{\mu^u_x\}$ is cs-invariant. Proposition~\ref{p.cont.dis} and Addendum~\ref{a.cont.dis} can be applied. One can thus normalize the center measure by requiring $\mu^c_x([x,f(x)))=1$.
One also gets an s- and u-invariant continuous family of local center measures $\{\nu^c_x\}$ and $\delta>0$
such that $\nu^c_x|_{B^c(x,\delta)}=\mu^c_x|_{B^c(x,\delta)}$
for $\mu$-almost every point $x$.

Since $\nu^c_{x}$ is equivalent to the Lebesgue measure on $\cW^c(x)$ for $\mu$-almost every point $x$,
the following limit exists:
$$\Delta(x):=\lim_{\varepsilon\to 0} \frac{\nu^c_{x}(B^c_\varepsilon(x) }{\vol^c(B^c_\varepsilon(x))}.$$
Since $\{\nu^c_y\}$ is s- and u-invariant and since s- and u-holonomies are $C^1$
inside the center stable and unstable leaves~\cite[Theorem~B]{PSW97},
the limit $\Delta(y)$ exists at any point $y$ that can be connected to such a point $x$ by an su-path. Since $f$ is accessible, $\Delta(x)$ exists at any point $x\in M$.

The family $\{\nu^c_x\}$ is continuous, hence $\Delta$ is the pointwise limit of continuous functions $x\mapsto \frac{\nu^c_{x}(B^c_\varepsilon(x)) }{\vol^c(B^c_\varepsilon(x))}$.
Baire's Theorem implies that $\Delta$ admits a continuity point $x_0$ (see~\cite[Theorem~7.3]{Ox80}).
Let us consider any point $x\in \cW^u(x_0)$. There exists a homeomorphism
$\fh^u$ from a neighborhood $U$ of $x_0$ to a neighborhood $V$ of $x$
such that for any $y\in U$, the restriction $\fh^u$ to $\cW^c_{loc}(y)\cap U$ takes its values
in $\cW^c_{loc}(\fh^u(y))\cap V$ and coincides with the unstable holonomy.
Since the unstable holonomies between center plaques are differentiable with uniformly continuous derivatives,
and since $\{\nu^c_x\}$ is u-invariant with a normalization $K_{x,y}=1$ (see Addendum~\ref{a.cont.dis}),
one deduces that $x$ is also a continuity point of $\Delta$.
The same argument holds for points $x\in \cW^s(x_0)$.
Since $f$ is accessible, the function $\Delta$ is continuous everywhere. 

For discretized Anosov flows, the center bundle is orientable (see~\cite{santiago}).
One can thus consider a continuous unit vector field $Z_0$ tangent to $E^c$ and define
$Z(x)=\frac{1}{\Delta(x)}Z_0(x)$. 
The flow $(\phi_t)$ generated by $Z$ is continuous and $C^n$ along the center direction.
Moreover, $\mu$-almost every point $x\in M$ satisfies $\mu^c_x[x,\phi_t(x))=t$, by our normalization of $\mu^c_x$
and the fact that at $\mu^c_x$-almost every point $y\in [x,\phi_t(x)]$, the measures
$\mu^c_x$ and $\nu^c_y$ coincides on a neighborhood of $y$ in $\cW^c(y)$.
So we conclude that $\phi_1(x)=f(x)$ on a dense subset of $\supp(\mu)$, hence $\phi_1=f$ everywhere.
By~\cite[Proposition 3.16]{santiago}, $(\phi_t)$ is a topological Anosov flow.
Moreover, by the same arguments as~\cite[Lemma~7.5]{AVW11} if $f$ is $C^n$, $n\in \natural$, then $\Delta$ is $C^n$,
and the same holds for the flow along the center leaves.
\end{proof}

\begin{remark}
When the center leaves are the fibers of a fiber bundle $\pi\colon M\to N$,
the dynamics induces a hyperbolic homeomorphism of N, see for example~\cite{ViY13}.
In this case, the existence of a continuous su-invariant center disintegration is obtained
under the assumption that the quotient measure $\pi_*\mu$ is locally equivalent to
a product between measures inside stable and unstable sets, see~\cite{Extremal}, \cite{ViY13}.

In our setting as the center foliation in general does not form a fiber bundle we don't have a natural projection to a quotient space. A natural definition for this projective product structure can be the following.

A measure $\mu$ has \emph{projective local product structure} if for $\mu$-almost every $x\in M$ there exists
a set $V=\cup_{z\in \cW^{su}_{r,\delta}(x)}B^c_{\ell}(z)$ with $r,\delta,\ell>0$ which satisfies:
\begin{itemize}
\item the sets $B^c_\ell(z)$ are pairwise disjoint for different $z\in \cW^{su}_{r,\delta}(x)$, so that
the projection along center discs $\pi^c:V\to \cW^{su}_{r,\delta}(x)$ is well defined;
\item $\nu=\pi^c_*\mu\mid_V$ is equivalent to a product measure $\nu^s\times \nu^u$ for the product structure on $\cW^{su}_{r,\delta}(x)$.
\end{itemize}
 
We believe that Theorem~\ref{p.cont.dis} can be adapted to measures having such projective local product structure, but as in our applications it is easier to check cs- or cu-quasi-invariance we did not explore this definition. 
\end{remark}


\section{measures of maximal entropy}\label{s.maximal.entropy}
In this section we prove Theorem~\ref{thm.MME}, Theorem~\ref{thm.MME.An.Fl} and Corollary~\ref{c.continuity.mme}.

\subsection{Perturbation of the time-$1$ map of Anosov flows}

\begin{proposition}\label{p.perturb.twisting}
Let  $(\phi_t)$ be a $C^r$ transitive Anosov flow, $1\leq r\leq \infty$.
There exists $f$ arbitrarily close to $\phi_1$ in $\Dif^r(M)$ such that:
\begin{itemize}
\item $f$ preserves the foliations $\cW^{cs},\cW^{cu}$ of $\phi_1$,
\item there exists a $f$-invariant compact center leaf $\cW^c(P)$ such that
the restriction $f|_{\cW^c(P)}$ has a Morse-Smale dynamics,
\item $\cW^u,\cW^s$ are minimal for any diffeomorphism $C^1$ close to $f$,
\item $f$ is accessible for any diffeomorphism $C^1$ close to $f$.
\end{itemize}
\end{proposition}
\begin{proof}
Let $Z$ be the vector field of the Anosov flow.
We follow \cite{BD96} to build a diffeomorphism $f$ arbitrarily $C^r$-close to $\phi_1$
such that for any perturbation $g$ of $f$ there exist periodic points
whose homoclinic intersections are dense in $M$.
The perturbation is done along the orbits of the flow $Z$
in such a way that a compact center leaf supports a Morse-Smale dynamics.
The two first items will thus be satisfied.

In~\cite{BD96}, $f$ is a perturbation of the time-$\tau$ map
of the flow where $\tau$ is the period of a periodic point of $(\phi_t)$.
The proof can be easily adapted in the following way.
First one perturbs the parametrization of the flow so that:
\begin{itemize}
\item[--] there is a center leaf $\mathcal{C}$ and an integer $k\geq 1$ such that $\phi_k|_\mathcal{C}=\id$,
\item[--] $\mathcal{C}$ contains exactly two periodic orbits,
\item[--] there is a curve $\gamma\subset \mathcal{C}$
bounded by two periodic points, $P$ attracting and $Q$ repelling, which does not contain any other periodic point,
\item[--] for some $m\in \mathbb{Z}$, the intersection $f^m(\cW^u(P))\cap \cW^s(Q)$ is non empty.
\end{itemize}
We continue as in the proof of Theorems~A and~2.1 in~\cite{BD96}.
With a second perturbation along center leaves,
we get the following open property:
\begin{itemize}
\item[(a)] The continuation of the points $P$ and $Q$ belong to two blenders.
\end{itemize}
In particular, for any diffeomorphism $g$ $C^1$-close, the hyperbolic continuations $P_g$ and $Q_g$
have the following properties: the closure of $\cW^u(P_g)$ contains the unstable manifold of $Q_g$
and the closure of $\cW^s(Q_g)$ contains the stable manifold of $P_g$, see Lemma 1.12 in~\cite{BD96}.

With a third perturbation, we get a new open property:
\begin{itemize}
\item[(b)] 
$\cW^u(Q)$ intersects stable manifold of the orbit of $P$,
 and $\cW^s(P)$ intersects the unstable manifold of the orbit of $Q$.
\end{itemize}
Indeed the foliation $\cW^{cs}$ for $f$ coincides with the center stable foliation of the transitive Anosov flow, hence is minimal. One deduces that $\cW^u(Q)$ intersects the local center-stable leaf of $P$.
If it does not intersects the stable manifold of the orbit of $P$, it intersects $\cW^s(Q)$:
this intersection is fragile and can be removed after a perturbation which preserves the foliations
$\cW^{cs}$ and $\cW^{cu}$.
Hence $\cW^u(Q)$ intersects the stable manifold of the orbit of $P$.
This property is open, we can repeat this construction and get a perturbation such that
$\cW^s(P)$ intersects the unstable manifold of $Q$ as wanted.

Finally we perform a fourth perturbation preserving the center foliation as in~\cite[Lemma A.4.3]{HHU08b}, which gives:
\begin{itemize}
\item[(c)] The system is not jointly integrable, i.e.
there does not exist a foliation whose leaves are tangent to
the bundle $E^s\oplus E^u$.
\end{itemize}
At this stage we have built the diffeomorphism $f$ and we have to check the two last properties
of the lemma.

We claim that the following open property holds:
\begin{claim}
Any strong unstable disc $D$ intersects the stable manifold of the orbit of $P$.
In particular the stable manifold of the orbit of $P$ is dense in $M$.
The analogous property holds for strong stable discs and the unstable manifold of the orbit of $Q$.
\end{claim}
\begin{proof}
As explain above, $D$ intersects $\cW^{cs}(P)$.
We only have to examine the case where $D$
intersects $\cW^u(f^k(Q))$ for some $k\in \mathbb{Z}$,
since otherwise the conclusion holds immediately.
In this case some of the iterates of $D$ will accumulate on $\cW^{u}(Q)$
and by (b) we conclude that $D$ intersects the stable manifold of the orbit of $P$.
\end{proof}

Let us assume for the moment that $P$ and $Q$ are fixed.
By compactness, the claim implies that a local stable manifold of $P$
intersects any strong unstable disc of radius $R$ large enough.
For any leaf $L$ of $\cW^u$ and $n\geq 1$ large enough,
the preimage $f^{-n}(L)$ contains an unstable disc of radius $R$,
hence by the inclination lemma the closure of $L$ contains $\cW^u(P)$
and, by (a), contains the unstable manifold of $Q$, which is dense in $M$ by the claim.
We have thus proved that the foliation $\cW^u$ is minimal.

When $P$ and $Q$ are not fixed, but have a common period $\tau$,
the same argument proves that $L\cup f(L)\cup\cdots\cup f^{\tau-1}(L)$
is dense in $M$.
Let us choose $L'$ in the closure of $L$ such that
$\Lambda=\overline{L'}$ is minimal for the inclusion.
Let $\ell\geq 1$ be the smallest integer such that
$M=\Lambda\cup f(\Lambda)\cup\cdots\cup f^{\ell}(\Lambda)$.
By minimality of $\Lambda$ and $\ell$, the intersections
$\Lambda\cap f^i(\Lambda)$ with $1< i\leq \ell$ (which are saturated by $\cW^u$-leaves) are empty.
Since $M$ is connected, one deduces $\ell=1$, hence $L$ is dense in $M$
and $\cW^u$ is minimal.
The same holds for the foliation $\cW^s$. This gives the third item.

The accessibility class $AC(x)$ of a point $x$ is the set of points $y$
that can be connected by a path which is piecewise tangent to $E^s$ or $E^u$.
By (c) there exists a point $x$ whose accessibility class $AC(x)$
is open, see~\cite[Lemma 3]{Di03}.
By the dynamical minimality of the unstable foliation we can joint any point to an iterate of $AC(x)$.
This implies that the iterates of $AC(x)$ cover $M$.
Since accessibility classes are either disjoint or equal, and that $M$ is connected
(since it supports a transitive flow), we have $M=AC(x)$ and $f$ is accessible.
By~\cite{Di03} the accessibility stil holds for diffeomorphisms that are $C^1$-close to $f$.
This settles the proof of the last item and of the proposition.
\end{proof}

\subsection{Proof of Theorem~\ref{thm.MME}}
By Corollary~\ref{cor.zero.exp}, $\{\mu^c_x\}$ is s- and u-invariant.
By \cite[Theorems~3.3 and~3.6]{buzzi-fisher-tahzibi}, since the foliations $\cW^s$ and $\cW^u$ are minimal,
the family $\{\mu^u_x\}$ is cs-quasi-invariant and $\mu$ has full support.
By Proposition~\ref{p.cont.dis} and Addendum~\ref{a.cont.dis},
one can normalize center measures by $\mu^c_x([x,f(x)))=1$
and consider a continuous family of local center measures $\{\nu^c_x\}_{x\in M}$.

Let us assume by contradiction that the center disintegrations of $\mu$ admit atoms.
Since $\mu$ is ergodic, $\mu$-almost every $x\in M$ satisfies $\mu^c_x(\{x\})>0$.
Let us fix such a point $x$.
The dynamics induced on compact center leaves has zero entropy, hence $\cW^c(x)$ is not compact.
Let us consider the interval $[x,f(x)]\subset \cW^c(x)$. By the minimality of $\cW^u$, there exist an infinite sequence of different points $y_i\in [x,f(x)]$ such that $\cW^s(y_i)\cap \cW^u(x)\neq \emptyset$. So we can define a sequence of maps $\fh_i$ that are compositions of stable and unstable holonomies such that $\fh_i(x)=y_i$.
By the su-invariance we have that $\mu^c_x(y_i)=\mu^c_x(x)$, a contradiction because $\mu^c([x,f(x)))=1$.

We now assume that $f$ is accessible. By Proposition~\ref{p.su.implies.leb} the center disintegration is absolutely continuous with respect to the Lebesgue measure. By Corollary~\ref{c.su.1d.implies.flow} we conclude that $f$ is the time-$1$ map
of a topological Anosov flow which is as smooth as $f$ alongs its orbits.

Let us prove the uniqueness of the m.m.e.
We consider two ergodic m.m.e. $\mu,\nu$.
By \cite{buzzi-fisher-tahzibi} they both have zero center Lyapunov exponent.
Moreover there exists a family of Margulis measures $x\mapsto m^u_x$ along the unstable foliations such that for $\mu$-almost every $x\in M$, $\mu^u_x=K_x. m^u_x$ and for $\nu$-almost every $y\in M$, $\nu^u_y=K_y.m^u_y$, for some positive numbers $K_x,K_y$. Moreover, the family of measures $\{m^u_x\}$ is $cs$-quasi invariant and each measure $m^u_x$ has full support inside $\cW^u(x)$.

Let $B(\mu)$ and $B(\nu)$ be the basins of $\mu$ and $\nu$, i.e. the set of points whose forward orbits equidistribute towards $\mu$ and $\nu$ respectively. There exists a set $B'(\mu)$, $\mu(B'(\mu))=1$, such that $\mu^c_x(M\setminus B(\mu))=0$ for every $x\in B'(\mu)$ (recall that zero measure sets are the same for any normalization of $\mu^c_x$). One defines an analogous set $B'(\nu)$ associated to $\nu$.

By the $cs$-quasi invariance of $\{m^u\}$, there exist $x\in B'(\mu)$ and $y\in B'(\nu)$ close
such that $y\in \cW^{cs}(x)$. Let us consider the stable holonomy $\fh^s:B^c_\rho(x)\to \cW^c(y)$ for some $\rho>0$.
As $\mu^c_{x}$ and $\nu^c_{y}$ are absolutely continuous with respect to the Lebesgue measure along the center and the stable foliation is absolutely continuous inside the center-stable leaves, we have $\fh^s(B(\mu))\cap B(\nu)\neq \emptyset$. Since the basins are $s$-saturated, $B(\mu)\cap B(\nu)\neq \emptyset$ and this implies $\mu=\nu$.
\qed

\subsection{Proof of Theorem~\ref{thm.MME.An.Fl}}

We apply Proposition~\ref{p.perturb.twisting} and get a $C^1$ open set $\mathcal{U}\subset \Dif^r(M)$
such that every $f\in \mathcal{U}$ has a compact center leaf $\cW^c(P)$ where the dynamics is Morse-Smale
is accessible and has minimal foliations $\cW^s,\cW^u$.

Let us assume by contradiction that there exists an ergodic m.m.e. $\mu$ with zero center exponent.
Theorem~\ref{thm.MME} implies that $f$ is the time-$1$ map of a non-singular flow whose orbits coincide with
the leaves of the center foliation. In particular this contradicts the fact that the restriction of $f$ to $\cW^c(P)$ is Morse-Smale.

We have proved that all ergodic m.m.e. of $f$ have non-zero Lyapunov center exponents. By \cite[Theorem~1.1]{buzzi-fisher-tahzibi} this implies that there are exactly two ergodic m.m.e., one with positive and one with negative center Lyapunov exponent; moreover both are Bernoulli.
\qed

\subsection{Proof of Corollary~\ref{c.continuity.mme}}

Since $\mu_n$ is ergodic and has non-negative center Lyapunov exponent, $h^s(f_n,\mu_n)=h(f_n,\mu_n)$.
Let $\mu$ be any accumulation point of $\mu_n$. By \cite{Jiagang-entropy-along-expanding} we have $\limsup_{n\to \infty} h^s(f_n,\mu_n)\leq h^s(f,\mu)$, so by our hypothesis we get $h^s(f,\mu)=h_{top}(f)$. In particular $\mu$ is a m.m.e. of $f$.
By~\cite[Theorem~1.1]{buzzi-fisher-tahzibi} two cases occur.
When all m.m.e. of $f$ have zero center exponents, $\lambda^c(\mu_n)\to 0$ and the conclusion of the corollary follows.
Otherwise $f$ admits exactly two ergodic m.m.e.: one with positive center exponent $\mu^+$ and one with negative center exponent $\mu^-$. This is the case that we have to consider.

We have $\mu=a\mu^+ + (1-a) \mu^-$ for some $a\in [0,1]$.
Let us assume by contradiction that $a<1$. By \cite{HHW} the stable entropy is an affine function of the measure, so we conclude $h^s(f,\mu^-)=h_{top}(f)$. As $\mu^-$ has a negative center exponent we get $h^s(f,\mu^-)=h^u(f,\mu^-)=h(f,\mu)$. We can apply~\cite{LY85b, ledrappier-xie} and concludes that the center disintegration of $\mu^-$
along $\cW^c(x)$ has an atom at $x$. This contradicts Theorem~\ref{thm.MME}.
We thus deduce that $(\mu_n)$ converges towards the m.m.e. $\mu^+$.
\qed



\section{Physical Measures}\label{s.phys.meas}
In this section we prove Theorems~\ref{thm.phys.meas.close.anosov} and~\ref{thm.phys.meas}.

\subsection{U-Gibbs measures with non-positive center exponents}
We first study the u-Gibbs measures without assuming the accessibility.

\begin{theorem}\label{t.physical.non.accessible}
Let $f$ be a partially hyperbolic, qua\-si-\-isometric in the center, $C^r$ diffeomorphism ($r>1$)
with 1D-center. If the u-foliation is minimal and
the center-stable foliation is absolutely continuous either:\newline
-- every u-Gibbs measure has zero center Lyapunov exponents,\newline
-- or is a unique u-Gibbs measure; its center Lyapunov exponent is negative.
\end{theorem}

The uniqueness when the center exponent is negative is classical:

\begin{proposition}\label{p.physical.unique.positive}
Let $f$ be a $C^r$ partially hyperbolic diffeomorphism ($r>1$),
whose u-foliation is minimal. If there is a u-Gibbs measure with negative center exponent, then it is the unique u-Gibbs measure and its basin has full Lebesgue measure
\end{proposition}
\begin{proof}
Let $\mu$ a u-Gibbs measure of $f$ with $\lambda^c(\mu)<0$.
There exists $x\in M$, $A\subset \cW^u(x)$ and $\delta>0$ such that
$\vol^u(A)>0$,  $A\subset B(\mu)$ and for every $x\in A$ the center-stable ball $B^{cs}_\delta(x)$ is contained in the Pesin stable manifold of $x$. Moreover the set of leaves $\cF^s:=\{B^{cs}_\delta(x), x\in A\}$ defines an absolutely continuous lamination.

Let us consider any u-Gibbs measure $\mu'$ and some unstable leaf
$\cW^u(y)$ such that $\vol^u(\cW^u(y)\setminus B(\mu'))=0$.
Since $\cW^u$ is a minimal foliation, $\cW^u(y)$ is dense.
Since the center-stable lamination $\cF^s$ is absolutely continuous,
$\cup_{x\in A}B^{cs}_\delta(x)$ intersects a subset of $\cW^u(y)$ which has positive volume for $\vol^u$.
So $B(\mu')\cap B(\mu)\neq 0$, which implies $\mu'=\mu$ and the uniqueness of the u-Gibbs measure.
By~\cite[Corollary 1.2]{CYZ}, its basin has full volume.
\end{proof}

Theorem~\ref{t.physical.non.accessible} is now a consequence of the following result.
For partially hyperbolic skew products, the analogous statement has been proved in~\cite{ViY13}.

\begin{proposition}\label{p.cs-abs-positive}
Let $f$ be a partially hyperbolic, qua\-si-\-isometric in the center, $C^1$ diffeomorphism
with 1D-center. If the u-foliation is minimal and
the center-stable foliation is absolutely continuous, then  every u-Gibbs measure $\nu$
has non-positive center Lyapunov exponent.
\end{proposition}

\begin{proof}
We follow~\cite{ViY13}. Arguing by contradiction, we assume that there is an ergodic u-Gibbs measure $\nu_*$ such that 
$$
\lambda:=\int \log \norm{Df\mid_{E^c(x)}}d \nu_*(x)>0.
$$
By ergodicity, we get $\nu_*(\Gamma)=1$ where $\Gamma$ is defined as:
$$
\Gamma:=\left\lbrace x\in M: \lim_{n\to \infty} \frac{1}{n} \sum_{j=1}^n \log \norm{Df\mid_{E^c(g^j(x))}}\geq\lambda/2\right\rbrace.
$$

\begin{claim}\label{c.bound.gamma}
For every $r>0$ there exists $n_r\in \natural$ such that $\# B_r^c(x)\cap \Gamma< n_r$.
\end{claim}
\begin{proof}
First observe that \cite[Lemma~3.8]{ViY13} remains valid in our setting, so there exists $\delta>0$ such that every neighborhood $U\subset \cW^c(x)$ of any $x\in \Gamma$ satisfies $\liminf \frac{1}{n} \sum_{i=0}^{n-1}\vol^c(f^{i}(U))\geq \delta$, where $\vol^c$ denotes the length along the center.
Let $K_0,c_0$ be the constants in the Definition~\ref{d.quasi-isometric} of quasi-isometric center
and let $L=\sup_{y\in M}\vol^c(B^c_{K_0 r+c_0}(y))$. We fix $n_r> L/\delta$.

Let us suppose by contradiction that there exist different $y_j \in B^c_r(x)\cap \Gamma$ with $j=1,\dots,n_r$
and let us choose disjoint neighborhoods $U_j\subset B^c_r(x)$ containing each of them.
By the quasi-isometric center property,
$$
L\geq  \frac{1}{n}\sum_{i=0}^{n-1} \vol^c(f^{i}(B^c_r(x)))\geq \sum_{j=1}^{n_r}\frac{1}{n}\sum_{i=0}^{n-1} \vol^c(f^{il}(U_j)).
$$
Taking the limit we get $L\geq n_r\delta$, a contradiction.
\end{proof}

Let us fix $r>0$ small and $n_r$ point $x_1,\dots, x_{n_r}$
in a same center ball $B^c_{r/2}(x_0)$.
Since $\cW^u$ is minimal, there exists points $y_i$ arbitrarily close to $x_i$
such that $B(\mu)$ has full Lebesgue measure inside $\cW^u(y_i)$.
Since $\cW^{cs}$ is absolutely continuous, there exists $x$ close to $x_0$ in $\cW^u(x_0)$
such that the intersection $y'_i$ between $\cW^{cs}_{loc}(x)$ and $\cW^u_{loc}(y_i)$
belongs to $B(\mu)$ for each $i$.
Let $z_i$ be the intersection between $\cW^s_{loc}(y'_i)$ and $\cW^c_{loc}(x)$.
We have found $n_r$ different points inside $B^c_r(x)$.
Since they belong to $B(\mu)$ and $E^c$ is one dimensional,
these points belong to $\Gamma$.
This contradicts the claim above.
The proof of Proposition~\ref{p.cs-abs-positive} is now complete.
\end{proof}

\subsection{Proof of Theorem~\ref{thm.phys.meas}}
Theorem~\ref{thm.phys.meas} will be a consequence of following more general result.
\begin{theorem}\label{thm2.phys.meas}
Let $f$ be a partially hyperbolic, accessible, qua\-si-\-isometric in the center, $C^r$ diffeomorphism ($r>1$) with one-dimensional center. If $\cW^u$ is minimal and
$\cW^{cs}$ is absolutely continuous, any u-Gibbs measure is a physical measure whose basin has full volume.

Moreover it is unique and its center Lyapunov exponent is non-positive.
\end{theorem}
\begin{proof}
By Theorem~\ref{t.physical.non.accessible} we are reduced to consider the case where all u-Gibbs measure
have their center Lyapunov exponent equal to zero. Let $\mu$ be one of them.

As the foliation $\cW^{cs}$ is absolutely continuous with respect to the Lebesgue measure
and $\mu$ has disintegrations along the $\cW^u$ leaves which are equivalent to the Lebesgue measure,
$\mu^u$ is cs-quasi-invariant and by Proposition~\ref{p.cs.prod.struct} has local cs$\times$u-product structure.
Since $\cW^u$ is minimal, $\mu$ has full support.
Since the center Lyapunov exponent vanishes,
$h(f,\mu)=h^u(f,\mu)=h^s(f,\mu)$.
By Proposition~\ref{p.su.implies.leb}, the center disintegrations $\mu^c_x$ are equivalent to the Lebesgue measure.

\begin{claim}
The center-stable disintegrations $\mu^{cu}_x$ are equivalent to
the Lebes\-gue measure along the center-unstable leaves.
\end{claim}
\begin{proof}
As before we can find $\varepsilon>0$ sufficiently small, such that for every $x\in M$, the map $\Psi:B^{c}_\varepsilon(x)\times B^u_\varepsilon(x)\to \cW^{cu}(x)$ defined by $\Psi(y,z)= \fh^u_{z}(y)$ is a homeomorphism over its image, denoted by $\cN^{cu}_\varepsilon(x)$. Let $\pi^u:\cN^{cu}_\varepsilon(x)\to B^c_\varepsilon(x)$ be the projection given by this product structure.

Restricted to $\cN^{cu}_\varepsilon(x)$ we have
$$
\mu^{cu}_x\mid_{\cN^{cu}_\varepsilon(x)}=\int_{B^c_\varepsilon(x)} \mu^u_y d\pi^u_*\mu^{cu}_x, 
$$

As $\mu$ is a u-Gibbs measure, $\mu^u_y$ is equivalent to the Lebesgue measure along $\cW^u(y)$.
Since $\cW^u$ is absolutely continuous,
it remains to prove that $\pi^u_*\mu^{cu}$ is equivalent to the Lebesgue measure on $B^c_\varepsilon(x)$.

To see this observe that for any measurable set $A\subset B^c_\varepsilon(x)$, and
$\mu^{cu}_x$-almost every points $z\in \cN^{cu}_\varepsilon(x)$,
we have $\mu^c_{z}(\cW^c_{loc}(z)\cap({\pi^u}^{-1}(A)))=0$ if and only if $\vol^c(\cW^c_{loc} (z)\cap({\pi^u}^{-1}(A)))=0$
(using that $\mu^c_z$ is equivalent to $\vol^c$),
if and only if $\vol^c(B^c_\varepsilon (x)\cap A)=0$ (since $\cW^u$ is absolutely continuous with respect to the Lebesgue measure inside the center-unstable leaves).
Consequently $\vol^c(B^c_\varepsilon (x)\cap A)=0$ if and only if $\mu^{cu}_x( \cN^{cu}_\varepsilon(x)\cap({\pi^u}^{-1}(A)))=0$,
which is equivalent to $\pi^u_*\mu^{cu}(A)=0$.
\end{proof}

Fix $x\in \supp(\mu)$, a cu-disk $D^{cu}$ centered at $x$ and $D= \cup_{y\in D^{cu}}\cW^s_{loc}(y)$.
There exists $x'$ close to $x$ such that $\vol^{cu}_{x'}(\cW^{cu}(x')\setminus B(\mu))=0$ and
$\cW^s_{loc}(y)\cap \cW^{cu}(x')\neq \emptyset$ for every $y\in D^{cu}$.
Using the previous claim and that $\cW^s$ is absolutely continuous,
Lebesgue almost every point $y$ inside $D^{cu}$
satisfies $\cW^s_{loc}(y)\cap B(\mu)\neq \emptyset$.
As $B(\mu)$ is $\cW^s$--saturated $\vol(D\setminus B(\mu))=0$. 
Since $\mu$ has full support, one deduces that $B(\mu)$ has full Lebesgue measure in $M$.

We have proved that $\mu$ is physical and its basin has full Lebesgue measure.
In particular, there is no other physical measure, hence no other u-Gibbs measure.
\qed

\subsection{Proof of Theorem~\ref{thm.phys.meas.close.anosov}}
Let  $(\phi_t)$ be a $C^r$ transitive Anosov flow ($2\leq r\leq \infty$),
let $f_0$ be the diffeomorphism given by Proposition~\ref{p.perturb.twisting} and let $\mu$ be a u-Gibbs measure of $f_0$. As $f_0$ is $C^1$ close to $\phi_1$, it is a discretized Anosov flow. By construction $\cW^{cs}$ coincides with the center stable foliation of the Anosov flow $(\phi_t)$. In particular is absolutely continuous
(see~\cite[Theorem 7.1.25]{fischer-hasselblatt}.
By Theorem~\ref{thm.phys.meas}, there exists a unique u-Gibbs measure $\mu$; moreover its center Lyapunov exponent is non-positive. By minimality of the strong unstable foliation $\cW^u$, it has full support.
As the foliation $\cW^{cs}$ is absolutely continuous with respect to the Lebesgue measure
and $\mu$ has disintegrations along the $\cW^u$ leaves which are equivalent to the Lebesgue measure,
$\mu^u$ is cs-quasi-invariant and by Proposition~\ref{p.cs.prod.struct} has local cs$\times$u-product structure.

If the center exponent is zero, then by Corollary~\ref{c.su.1d.implies.flow} $f_0$ is the time one map of a topological flow whose orbits coincides with the center leaves.
This is a contradiction because $f_0\mid_{\cW^c(P)}$ has a Morse-Smale dynamics.
So we conclude that the center exponent of $\mu$ is negative.

By \cite{Jiagang-entropy-along-expanding} there exists a $C^1$ open neighborhood $\mathcal{U}$ of $f_0$ in $\Dif^r(M)$
such that for every diffeomorphism $f$ in $\mathcal{U}$, every u-Gibbs measure has negative center exponent.
Proposition~\ref{p.physical.unique.positive} then implies that $f$ has a unique u-Gibbs measure,
and its basin has full Lebesgue volume. This measure is thus the unique physical measure of $f$.

\end{proof}


\bibliographystyle{plain-like-initial}

\information

\end{document}